\documentclass[10pt]{amsart}

\usepackage[english]{babel}
\usepackage{amssymb}
\usepackage{amsfonts} 
\usepackage{amsmath}
\usepackage{amsthm} 
\usepackage{hyperref}
\usepackage[initials]{amsrefs}
\usepackage{mathtools}
\usepackage{enumitem}
\usepackage{comment}
\usepackage{tikz-cd}
\usepackage{cleveref}
\usepackage{yfonts} 
\usepackage{microtype}
 
\usepackage[normalem]{ulem} 
\usepackage{mathrsfs}

\usepackage{xcolor}

\usetikzlibrary{positioning}

\newcommand{\enumicon}{$\,\,\,\triangleright$}

\newcommand{\bbN}{{\mathbb N}}
\newcommand{\bbQ}{{\mathbb Q}}
\newcommand{\bbR}{{\mathbb R}}
\newcommand{\bbE}{{\mathbb E}}

\newcommand{\bbZ}{{\mathbb Z}}
\newcommand{\bbC}{{\mathbb C}}
\newcommand{\bbH}{\mathbb{H}}

\newcommand{\calA}{\mathcal{A}}

\newcommand{\calC}{\mathcal{C}}

\newcommand{\calU}{\mathcal{U}}
\newcommand{\calH}{\mathcal{H}}
\newcommand{\calV}{\mathcal{V}}

\newcommand{\calR}{\mathcal{R}}
\newcommand{\calN}{\mathcal{N}}

\newcommand{\Rho}{\mathrm{P}}

\newcommand{\bs}{\backslash}
\newcommand{\betti}{\beta^{(2)}}

\newcommand{\defq}{\mathrel{\mathop:}=}
\newcommand{\id}{\operatorname{id}}
\newcommand{\im}{\operatorname{im}}

\newcommand{\map}{\operatorname{map}}
\newcommand{\pr}{\operatorname{pr}}
\newcommand{\const}{\operatorname{const}}

\newcommand{\supp}{\operatorname{supp}}
\newcommand{\vol}{\operatorname{vol}}
\newcommand{\voltrans}{\operatorname{vol}^{\mathrm{tr}}}

\newcommand{\half}{\frac{1}{2}}

\newcommand{\sign}{\operatorname{sign}}

\renewcommand{\star}{\operatorname{star}}

\newcommand{\xxm}{\ensuremath{X\times\widetilde{M}} }   
\newcommand{\ev}{\operatorname{ev}}
\newcommand{\cantornerve}{N^{\mathrm{Ca}}}
\newcommand{\neumann}{L^\infty(\mu)\bar\rtimes\Gamma}

\newcommand{\acts}{\curvearrowright}

\def\leq{\leqslant}
\newcommand{\abs}[1]{{\left\lvert #1\right\rvert}}

\newcommand{\norm}[1]{{\left\lVert #1\right\rVert}}


\newcommand{\inparens}[2][flex]{\csname #1l\endcsname(#2%
	\csname #1r\endcsname)\mathclose{}}
\newcommand{\inangles}[2][flex]{\csname #1l\endcsname\langle#2%
	\csname #1r\endcsname\rangle\mathclose{}} 
\newcommand{\innorm}[2][flex]{\csname #1l\endcsname|#2%
	\csname #1r\endcsname|\mathclose{}}
\newcommand{\indnorm}[2][flex]{\csname #1l\endcsname\|#2%
	\csname #1r\endcsname\|\mathclose{}}
\newcommand{\indnorml}[4][flex]{\csname #1l\endcsname\|#2%
	\csname #1r\endcsname\|_{#3}^{#4}\mathclose{}}

\newcommand{\sv}[2][flex]{\indnorm[#1]{#2}}%

\newcommand{\pfcl}[2][flex]{\csname #1l\endcsname[#2%
	\csname #1r\endcsname]}
\newcommand{\ifsv}[2][norm]{\csname #1l\endcsname\bracevert\!#2\!%
	\csname #1r\endcsname\bracevert}

\newtheorem{theorem}{Theorem}[section]
\newtheorem{lemma}[theorem]{Lemma}

\newtheorem{corollary}[theorem]{Corollary}

\theoremstyle{definition}
\newtheorem{definition}[theorem]{Definition}

\newtheorem{example}[theorem]{Example}

\newtheorem{remark}[theorem]{Remark}

\numberwithin{equation}{section}

\setlist[itemize]{label={\enumicon}}

\begin{document}

\title{Volume and Macroscopic scalar curvature}

\author{Sabine Braun}
\address{Karlsruhe Institute of Technology}
\email{sabine.r.braun@gmail.com}

\author{Roman Sauer}
\address{Karlsruhe Institute of Technology}
\email{roman.sauer@kit.edu}
\urladdr{topology.math.kit.edu}

\thanks{Parts of the results of this article were obtained in the PhD thesis~\cite{braunphd} of the first named author under the supervision of the second named author.
Both authors acknowledge funding by
the Deutsche Forschungsgemeinschaft (DFG) 281869850 (RTG 2229).}



\begin{abstract}
 We prove the macroscopic cousins of three conjectures: 1) a conjectural bound of the simplicial volume of a Riemannian manifold in the presence of a lower scalar curvature bound, 2) the conjecture that rationally essential manifolds do not admit metrics of positive scalar curvature, 3) a conjectural bound of $\ell^2$-Betti numbers of aspherical Riemannian manifolds in the presence of a lower scalar curvature bound. The macroscopic cousin is the statement one obtains by replacing a lower scalar curvature bound by an upper bound on the volumes of $1$-balls in the universal cover. 
\end{abstract}
\maketitle


\section{Introduction}\label{sec:intro}
\subsection{Results}
Scalar curvature is a microscopic concept. The scalar curvature at a point~$p$ of a Riemannian manifold~$M$  can be read off from the volumes of balls around~$p$ whose radii approach zero. A lower scalar curvature bound for~$M$ corresponds to an upper bound on the volumes of sufficiently small balls in $M$ or, equivalently, in $\widetilde M$ as the universal cover projection $\widetilde M\to M$ is locally isometric. 

How can we replace scalar curvature by a macroscopic concept? For instance, by replacing a lower scalar curvature bound by an upper bound on the volumes of balls of a fixed radius, say radius~$1$, in the universal cover. We form the \emph{macroscopic cousin} of a mathematical statement involving a lower scalar 
curvature bound for a Riemannian manifold~$M$ by the replacing it with an upper bound on the volumes of $1$-balls in $\widetilde M$. Guth's ICM report~\cite{guth-metaphors} describes the analogies and connections that emerge from the macroscopic point of view. \medskip

Our first main theorem is the macroscopic cousin of a conjecture by 
Gromov~\cite{gromov-large}*{Conjecture~3A} for which it is assumed that the scalar curvature is bounded from below by~$-1$. Via the Bishop-Gromov inequality one sees that it also generalizes Gromov's main inequality~\cite{gromovvbc}*{Section~0.5} for which it is assumed that the Ricci curvature is bounded from below by~$-1$.

\begin{theorem}\label{thm: main result}
	For every $V_1>0$ and $d\in\bbN$ there is constant $\const(d, V_1)>0$ with the following property. 
   If $M$ is a $d$-dimensional closed Riemannian manifold such that the volume of every $1$-ball in the universal cover of $M$ is at most~$V_1$, then 
   \[ \sv M\le \const(d, V_1)\cdot \vol(M),\]
   where $\sv M$ denotes the simplicial volume of~$M$. 
\end{theorem}

Theorem~\ref{thm: main result} generalizes another theorem, Guth's volume theorem, as the simplicial 
volume of a hyperbolic manifold coincides with its volume up to a dimensional constant by a result of Gromov and Thurston. 
Other generalizations of Guth's theorem can be found in~\citelist{\cite{balacheff+karam}*{Theorem~1.3}\cite{alpert-area}}. 

We adopt the following convention. Within a statement P about a manifold, a \emph{dimensional constant} just means a positive real constant that only depends on the dimension of the manifold. In other words, if the dimensional constant is denoted by $c(d)$ the statement P should be read with the preface: \emph{For every dimension $d$ there is a constant $c(d)>0$ such that the following holds true.}

 Let us denote the supremal volume of an $r$-ball in a Riemannian manifold 
 $(M,g)$ by $V_{(M,g)}(r)$. The induced metric on the universal cover is denoted by~$\tilde g$. 

\begin{theorem}[Guth's Volume theorem~\cite{guthvolume}]\label{thm: guth volume thm}
	Let $M$ be a $d$-dimensional closed hyperbolic manifold, and let $g$ be another metric on $M$. Suppose that 
	\[ V_{(\widetilde M, \widetilde g)}(1)\le V_{\bbH^d}(1),\]
	where $\bbH^d$ is $d$-dimensional hyperbolic space. Then 
	\[ \vol(M,g_\mathrm{hyp})\le \const(d)\cdot \vol(M, g)\]
   for a dimensional constant $\const(d)>0$. 
\end{theorem}

Guth's volume theorem is the macroscopic cousin of Schoen's conjecture~\cite{schoen}*{p.~127}
which says that $\operatorname{scal}_{M,g}\ge \operatorname{scal}_{M,g_\mathrm{hyp}}$ implies $\vol(M, g_\mathrm{hyp})\le \vol(M, g)$. More precisely, it is the \emph{non-sharp} macroscopic cousin because of the dimensional constant $\const(d)$. \medskip

Our second main theorem for $R=1$ is the non-sharp macroscopic cousin of the conjecture that rationally essential manifolds do not admit a metric of positive scalar curvature (the statement for $R=1$ readily implies the one for all $R>0$ by scaling the metric). 

\begin{theorem}\label{thm: main result about essential manifolds}
   There is a dimensional constant $\epsilon(d)>0$ with the following property. 
   For every rationally essential Riemannian manifold $(M, g)$ of dimension~$d$ and every $R>0$ we have 
   \[ V_{(\widetilde M, \widetilde g)}(R)> \epsilon(d)\cdot R^d.\]
   \end{theorem}

If $\epsilon(d)$ could be chosen to be the volume of a Euclidean $d$-ball, then the above conjecture would follow. A closed oriented manifold is \emph{rationally essential} if its classifying map sends the fundamental class to a non-zero class in rational homology. 
Guth proves the volume estimate in Theorem~\ref{thm: main result about essential manifolds} for Riemannian manifolds whose universal cover have  infinite filling radius~\cite{guthvolume}*{Theorem~1}. Not every rationally essential manifold has a universal cover with infinite filling radius according to~\cite{brunnbauer+hanke}*{Theorem~1.4 and Proposition~2.8}. \medskip

Our third main theorem is the macroscopic cousin of the combined conjectures~\citelist{\cite{gromov-large}*{Conjecture~3A}\cite{gromov-singularities}*{3.1 (e) on p.~769}} by Gromov (see also~\cite{gromov-asymptotic}*{p.~232}). It generalizes the main result in~\cite{sauerminvol} where a lower Ricci curvature bound was assumed. 
See also~\cite{sauer-homologygrowth} for related results in the residually finite case. 

\begin{theorem}\label{thm: main l2 result}
   For every $V_1>0$ and $d\in\bbN$ there is a constant $\const(d, V_1)>0$ with the following properties. 
   
   \begin{enumerate}
  \item If $(M, g)$ is a $d$-dimensional connected closed oriented Riemannian manifold with classifying map $c\colon M\to B\Gamma$ such that $V_{(\widetilde M, \widetilde g)}(1)\le V_1$, then the von Neumann rank of 
   $c_\ast([M])\in H_d(B\Gamma)$, where $[M]$ is the fundamental class of~$M$, is bounded from above by $\const(d,V_1)\cdot \vol(M)$. 
\item If, in addition, the manifold $M$ is aspherical, then its 
     $\ell^2$-Betti numbers 
	satisfy 
	\[ \betti_i(M)\le \const(d, V_1)\cdot \vol(M, g)\]
for every $i\in\bbN$, and the Euler characteristic satisfies  
\[ |\chi(M)|\le \const(d,V_1)\cdot \vol(M,g).\]
   \end{enumerate}
\end{theorem}
 
Subsection~\ref{sub: l2 betti numbers} contains an overview of what we need from the theory of $\ell^2$-Betti numbers, including the definition of von Neumann rank. 

We present a result which is an outcome of our methods but is of independent interest. See Remark~\ref{rem: relation to foliated simplicial volume} for the relevant notions. 

\begin{theorem}\label{thm: main foliated simplicial volume}
   For every $V_1>0$ and $d\in\bbN$ there is constant $\const(d, V_1)>0$ with the following properties: 
   
   The integral foliated simplicial volume of 
   every $d$-dimensional closed aspherical Riemannian manifold $(M,g)$ that satisfies $V_{(\widetilde M, \widetilde g)}(1)\le V_1$ is bounded from above by   
   $\const(d,V_1)\cdot \vol(M,g)$. More precisely, the inequality $|M|^\alpha\le \const(d,V_1)\cdot \vol(M,g)$ holds for any free measurable pmp action $\alpha$ on a standard probability space. 
\end{theorem}

Theorem~\ref{thm: main foliated simplicial volume} generalizes Corollary 1.2 of~\cite{fauser} since the assumption in~\cite{fauser}*{Corollary~1.2} implies the vanishing of the minimal volume~\cite{collapse}*{Theorem~3.1}. Theorem~\ref{thm: main foliated simplicial volume} shows the vanishing of the integral foliated simplicial volume (and its variants above) for $3$-dimensional graph manifolds since their minimal volume vanishes~\cite{collapse}*{Example~0.2 and Theorem~3.1}. This vanishing result is a special case of~\cite{fauser+friedl+loeh}*{Theorems~1.2 and~1.6}.

\subsection{Comment on the proof}
The proof of Theorem~\ref{thm: main result} involves an action of the fundamental group on the Cantor set. We will explain why. 

A close reading of Guth's proof of the volume theorem yields a proof of Theorem~\ref{thm: main result} for  
closed aspherical manifolds whose smallest non-contractible loop (\emph{systole}) is of length at least~$1$ (see the discussion~\cite{alpert-area}*{Section~4}). If the fundamental group is residually finite, then we can pass to a finite cover whose systole is of length at least~$1$. 
Since the stated inequality between simplicial volume and Riemannian volume may be verified on every finite cover, Theorem~\ref{thm: main result} follows for closed aspherical manifolds with residually finite fundamental groups. 

We have to get rid of the assumptions of asphericity and residual finiteness. Residual finiteness is hard to verify beyond locally symmetric spaces; we do not even know whether fundamental groups of closed negatively curved manifolds are residually finite. The attempt to get rid of residual finiteness leads to actions on the Cantor set. 

If the fundamental group~$\Gamma$ of our manifold~$M$ is not residually finite, we may not have enough  finite covers to enforce a large systole. Let us consider the inverse limit  
\[ \varprojlim_{i\in I} M_i\]
of the directed system of all connected finite regular covers of $M$ -- even if there none except $M$ itself in the most extreme case. By covering theory the system $M_i$, $i\in I$, corresponds to a directed system of finite index normal subgroups $\Gamma_i<\Gamma$, $i\in I$. 
Each $M_i$ is just the quotient $\Gamma_i\bs\widetilde M$. 
The inverse limit $\widehat \Gamma=\varprojlim \Gamma/\Gamma_i$ is the \emph{profinite completion} of $\Gamma$. We have 
\[ \varprojlim_{i\in I}M_i\cong \varprojlim_{i\in I}\Gamma_i\bs \widetilde M\cong \varprojlim_{i\in I}\Gamma/\Gamma_i\times_\Gamma\widetilde M\cong \widehat \Gamma\times_\Gamma\widetilde M.\]
The $\Gamma$-quotient on the right is the quotient by the diagonal action. The profinite completion has an obvious action by translations. The $\Gamma$-space $\widehat \Gamma$ has three important properties: 
\begin{enumerate}
\item It is homeomorphic to the Cantor set. 
\item It possesses an invariant probability measure (the Haar measure).
\item If $\Gamma$ is residually finite then the $\Gamma$-action is free.  
\end{enumerate}

Let $X$ be the Cantor set. In Theorem~\ref{thm: existence of action} we reprove an observation of Hjorth-Molberg that every countable group $\Gamma$ admits a free, continuous action on $X$ having a $\Gamma$-invariant probability measure. With regard to such an action we form the space 
\[X\times_\Gamma \widetilde M,\] 
which acts as a replacement for $\varprojlim M_i$ if the fundamental group is not residually finite. Guth's methods need the finite covers $M_i$. Our contribution is to generalize them so that we can work with the global object $X\times_\Gamma \widetilde M$ instead. 

Recent results of Liokumovich-Lishak-Nabutovsky-Rotman~\cite{fillingmetric} and Papasoglu~\cite{papasoglu} generalize Guth's theorem in~\cite{guthuryson}*{Theorem~0.1} on Uryson width. Papasoglu's proof is simpler than Guth's proof which underlies our work. Can one could combine 
the method in~\cite{papasoglu} with our ideas to obtain shorter proofs for our main results or to obtain explicit constants as in Nabutovsky's paper\cite{explicitbound}? We do not think this is possible in the case of Theorems~\ref{thm: main result} and~\ref{thm: main l2 result}, but it might be possible in the case of Theorem~\ref{thm: main result about essential manifolds}.  

\subsection{Structure of the proof}
We establish a framework of equivariant bundles over $X$ (\emph{Cantor bundles}). After Section~\ref{sec: topological preliminaries} on preliminaries we introduce the notion of Cantor bundle in Section~\ref{sec: category cantor bundles}. The 
space $\xxm$ with its diagonal action of~$\Gamma=\pi_1(M)$ is a trivial example. A more interesting toy example is Example~\ref{exa: non trivial Cantor bundle}. Cantor bundles can be regarded as spaces with a groupoid action, namely the groupoid given by the orbit equivalence relation on $X$, endowed with additional geometric data. Spaces with groupoid actions are considered in many contexts. Our main inspiration came from Gaboriau's $\calR$-simplicial complexes in the measurable world~\cite{gaboriau}. Another influence 
is Gromov's paper~\cite{gromov-foliated}. 

After discussing transverse Hausdorff measures on Cantor bundles  
in Section~\ref{sec: transverse measure} we introduce the rectangular nerve construction in the framework of Cantor bundles (Section~\ref{sec: rectangular cantor nerves}). The \emph{rectangular Cantor nerve} of an equivariant cover on $\xxm$ as described above is a non-trivial Cantor bundle. The toy example gives a good impression how such a rectangular Cantor nerve might look like. 

In Section~\ref{sec: volume estimates} we establish the existence of good covers in our framework and prove the analog of Guth's result on the exponential decay of the volume of the high multiplicity set in our framework. This is the main point about the auxiliary space~$X$: We cannot obtain a good, equivariant cover on $\widetilde M$, only on the Cantor bundle $\xxm$. 

We then bound the transverse volume of the image of the map to the rectangular Cantor nerve. In Section~\ref{sec: homotoping down} this map is homotoped as a Cantor bundle map to the $d$-skeleton where $d$ is the dimension of $M$. In Section~\ref{sec: from volume to simplicial volume} we relate what we have done so far to the simplicial volume of $M$. Here we use tools from homological algebra and equivariant topology.

\section{Topological preliminaries}\label{sec: topological preliminaries}

 In Subsection~\ref{sub: action on cantor set} we present a short proof of the existence of suitable actions on the Cantor set which is a result of Hjorth-Molberg. In~\ref{sub: equivariant CW complexes} we review the notion of an equivariant CW-complex and of a classifying space. 
In~\ref{sub: rectangular complexes} and~\ref{sub: rectangular nerves} we give a detailed review of rectangular complexes and Guth's rectangular nerve since special care is needed in our equivariant context. 

We adhere to the following notation. 
Let $M$ be a closed $d$-dimensional Riemannian manifold with 
fundamental group $\Gamma$. Its universal cover is denoted by $\widetilde M$ and endowed with the Riemannian metric induced by $M$. 
The Cantor set is denoted by $X$. We fix a free continuous $\Gamma$-action on $X$ and a $\Gamma$-invariant Borel probability measure $\mu$ on $X$ whose existence is stated in Theorem~\ref{thm: existence of action}. 
If $B=B(p,r)\subset\widetilde{M}$ is the open ball of radius $r$ around $p$, then $aB$ is the concentric ball of radius $a\cdot r$ around $p$.

\subsection{Free actions on the Cantor set}\label{sub: action on cantor set}

The following observation is due to Hjorth and Molberg~\cite{hjorth+molberg}*{Theorem~0.1}. 
Based on the notion of co-induction we formulate a shorter proof here for the convenience of the reader. A stronger statement, which we only need in the proof of Theorem~\ref{thm: main foliated simplicial volume}, was obtained by Elek~\cite{elek}. 

\begin{theorem}\label{thm: existence of action}
Let $\Gamma$ be a countable discrete group and let $X$ be the Cantor set. Then there is a free, continuous $\Gamma$-action on $X$ having a $\Gamma$-invariant probability measure. 	
\end{theorem}

\begin{proof}
The case of finite groups is easy. We may and will assume that $\Gamma$ is infinite. 
For every element $\gamma\in\Gamma$ let $X_\gamma$ be the profinite completion 
of the cyclic subgroup $\langle\gamma\rangle$ endowed with the left translation action by $\langle\gamma\rangle$ and the normalized Haar measure~$\nu_\gamma$. 
Depending on the order of $\gamma$, $X_\gamma$ is either a finite set or homeomorphic to the profinite completion $\widehat\bbZ$ of $\bbZ$, which is a Cantor set. 
Let $Y_\gamma$ be the co-induction of the $\langle\gamma\rangle$-space $X_\gamma$, that is 
\[ Y_\gamma\defq\map(\Gamma, X_\gamma)^{\langle\gamma\rangle}=\bigl\{f\colon\Gamma\to X_\gamma\mid \forall_{x\in\Gamma} f(\gamma x)=\gamma\cdot f(x)\bigr\}\]
endowed with the compact-open topology and the left $\Gamma$-action 
$(\lambda\cdot f)(x)=f(x\lambda)$ for $x\in \Gamma$ and $\lambda\in \Gamma$. 
Non-equivariantly, $Y_\gamma$ is homeomorphic to the product $\prod_{\langle\gamma\rangle\bs \Gamma}X_\gamma$, which is a Cantor set. One easily verifies that the product measure $\mu_\gamma$ of the $\nu_\gamma$ is invariant under the $\Gamma$-action on $Y_\gamma$. Finally, we define $X$ to be the product 
\[ X\defq \prod_{\gamma\in \Gamma} Y_\gamma\] 
endowed with diagonal $\Gamma$-action and the product measure of the measures~$\mu_\gamma$. 
The product measure is clearly $\Gamma$-invariant. As a countable product of Cantor sets, $X$ is a Cantor set. It remains to show that the $\Gamma$-action on $X$ is free. Let $x=(y_\gamma)\in X$ and $\gamma_0\in\Gamma$. Assume that $\gamma_0\cdot x=(\gamma_0\cdot y_\gamma)_{\gamma\in\Gamma}=(y_\gamma)_{\gamma\in\Gamma}$. Since the $\langle\gamma_0\rangle$-action on $Y_{\gamma_0}$ is free, it implies that $\gamma_0=e$. 
\end{proof}

\subsection{Equivariant CW-complexes}\label{sub: equivariant CW complexes}

We recall some terminology concerning equivariant CW-complexes and classifying spaces. For the notion of an \emph{(equivariant) $\Gamma$-CW-complex} we refer to~\cite{tomdieck}*{Section~II.1}. The skeleta $N^{(n)}$ of a $\Gamma$-CW-complex $N$ are built inductively via $\Gamma$-pushouts of the form 
\[\begin{tikzcd}
    \coprod_{i\in I_n} \Gamma/H_i\times S^{n-1}\ar[r]\ar[d,hook] & N^{(n-1)}\ar[d, hook]\\
    \coprod_{i\in I_n} \Gamma/H_i\times D^n\ar[r] & N^{(n)}
\end{tikzcd}
\]
The conjugates of the subgroups $H_i$, $i\in I_n$, $n\ge 0$, are precisely the isotropy groups of the $\Gamma$-space~$N$. If all subgroups $H_i$ are trivial, then $N$ is a \emph{free} $\Gamma$-CW-complex. 
If all subgroups $H_i$ are finite, then $N$ is a \emph{proper} $\Gamma$-CW complex. 
The universal cover of a CW-complex with fundamental group $\Gamma$ has a natural structure of a free $\Gamma$-CW-complex. 

A \emph{cellular action} of a discrete group $\Gamma$ on a CW-complex $W$ is a continuous action of $\Gamma$ on $W$ such that 
\begin{enumerate}
\item\label{eq: cell permute} for every open cell $e$ and $\gamma\in\Gamma$ the translate $\gamma e$ is an open cell and
\item\label{eq: pointwise stabilizer} if $\gamma\in\Gamma$ fixes an open cell set-wise then it does so point-wise. 
\end{enumerate}
A CW-complex with a cellular $\Gamma$-action is a $\Gamma$-CW-complex in the sense of~\cite{tomdieck}*{p.~98} (see~\cite{tomdieck}*{Proposition~(1.15) on p.~101}), which means that is obtained from glueing equivariant cells $\Gamma/H\times D^k$ along their boundaries $\Gamma/H\times S^{k-1}$ where $H<\Gamma$ is a subgroup. 

The equivariant homotopy category of free  $\Gamma$-CW complexes possesses a terminal object which is denoted by $E\Gamma$. The space $E\Gamma$ is unique up to equivariant homotopy and called the \emph{classifying space of~$\Gamma$}. The quotient of $E\Gamma$ is commonly denoted by $B\Gamma$ and also called \emph{classifying space}. 
Each free $\Gamma$-CW-complex admits an equivariant map to the classifying space of~$\Gamma$. Any such map -- they are unique up to homotopy -- is called \emph{classifying map}. 
\subsection{Rectangular complexes}\label{sub: rectangular complexes}

A \emph{rectangular complex} is a $M_\kappa$-polyhedral complex with $\kappa=0$ in the sense of Bridson-Haefliger~\cite{bridson+haefliger}*{Definition~7.37 on p.~ 114} such that each cell is isometric to a Euclidean $d$-cuboid $[0,a_1]\times [0, a_2]\times\cdots\times [0, a_d]\subset \bbR^d$ and the intersection of two cells is either empty or a single face. 
We recall some terminology and basic facts from the book of Bridson-Haefliger~\cite{bridson+haefliger}*{Chapter~I.7}. 

The faces of $[0,a]$ are just $\{0\}, \{a\}$ and $[0,a]$. The \emph{faces} of a Euclidean $d$-cuboid $[0,a_1]\times\cdot\times [0,a_d]$ are the subsets given by $F_1\times\cdots\times F_d$ where each $F_i$ is a face of $[0,a_i]$. Faces of dimensions $0$ and $1$ are also called \emph{vertices} and \emph{edges}, respectively. The \emph{barycenter} of a Euclidean $d$-cuboid $C=[0,a_1]\times [0, a_2]\times\cdots\times [0, a_d]$ with $d>0$ is the point $(\frac{1}{2}a_1, \dots, \frac{1}{2}a_d)$. It lies in the interior of $C$ and is fixed by any isometry of $C$. The barycenter of a vertex is the vertex itself. 

A rectangular complex has the structure of a CW-complex with the cells corresponding to the Euclidean cuboids. Depending on the context, we refer to the latter as cells or (Euclidean) cuboids or faces. A rectangular complex is endowed with the path metric that is induced by the Euclidean metric on each Euclidean cuboid. 

The second barycentric subdivision of a rectangular complex is simplicial complex, even a $M_0$-simplicial complex~\cite{bridson+haefliger}*{Proposition~7.49 on p.~118}. 

Let $J$ be a, possibly countably infinite, index set. The real vector space with basis $J$ will be denoted by $\bbE^J$. We regard $\bbE^J$ as the vector space of real sequences indexed over $J$ that have only 
finitely many non-zero components. We endow $\bbE^J$ with the Euclidean norm and metric. 

For a family $(a_j)_{j\in J}$ of positive real numbers we will define a rectangular complex 
\[N\bigl((a_j)_{j\in J}\bigr)\subset\bbE^J\]
as a subset of $\bbE^J$ in the following way. 
The vertices of $N$ are the sequences of $\bbE^J\bs\{0\}$ whose $j$-component, $j\in J$, is either $0$ or $a_j$. Two vertices are \emph{adjacent} if they differ in exactly one component. 
A family of $2^k$ vertices span a $k$-face (or $k$-cell, or $k$-cuboid) given by their convex hull if each vertex is adjacent to exactly $k$ vertices. 
We call $N\bigl((a_j)_{j\in J}\bigr)$ the \emph{rectangular complex associated to the family $(a_j)_{j\in J}$.} 

To see that the previous definition yields a rectangular complex we have to verify that the intersection of two faces is empty or a single face. To this end, we start with following remark. 

\begin{remark}\label{rem: subsets of index set}

A face $F$ in $N((a_j)_{j\in J})$ is a subset of $\bbE^J$ of the following type: There 
is a finite subset $J'\subset J$ and there are $c_j\in\{0,a_j\}$ for every $j\in J\bs J'$ with $(c_j)_{j\in J\bs J'}\ne 0$ such that
\begin{equation}\label{eq: type of cuboid}
F= \bigl\{ (b_j)_{j\in J}\mid \forall_{j\in J'} ~b_j\in [0,a_j]\wedge\forall_{j\in J\bs J'}~ b_j=c_j\bigr\}.
\end{equation}

Vice versa, every such subset is a face in $N((a_j)_{j\in J})$, namely the convex hull of the 
following set of vertices of cardinality $2^{\# J'}$ 
\[ \bigl\{  (b_j)_{j\in J}\mid \forall_{j\in J'} ~b_j\in \{0,a_j\}\wedge\forall_{j\in J\bs J'}~ b_j=c_j\bigr\}.\]
Depending on $F$, we define the following four subsets of the index set $J$: 
\begin{align*}
	J_0(F) &\defq \bigl\{ j\in J\bs J'\mid c_j=0\bigr\}&
	J_{\frac{1}{2}}(F)&\defq J'\\
	J_1(F) &\defq \bigl\{j\in J\bs J'\mid c_j=a_j\bigr\}&
	J_+(F)&\defq J_1(F)\cup J_{\frac{1}{2}}(F)
\end{align*}
Equivalently we could define $J_0(F)$, $J_{\frac{1}{2}}(F)$, and $J_1(F)$ as the subset of indices $j\in J$ for which the $j$-component of the barycenter of $F$ is $0$, $\frac{1}{2}a_j$, and $1$, respectively. 
\end{remark}

Let us consider two faces 
\begin{align*}
    F&=  \bigl\{ (b_j)_{j\in J}\mid \forall_{j\in J_{\frac{1}{2}}(F)} ~b_j\in [0,a_j]\wedge\forall_{j\in J\bs J_{\frac{1}{2}}(F)}~ b_j=c_j\bigr\}\\
    \tilde F&= \bigl\{ (b_j)_{j\in J}\mid \forall_{j\in J_{\frac{1}{2}} (\tilde F)} ~b_j\in [0,a_j]\wedge\forall_{j\in J\bs J_{\frac{1}{2}}(\tilde F)}~ b_j=\tilde c_j\bigr\}
\end{align*}

The intersection 
\[ F\cap \tilde F= \bigl\{ (b_j)_{j\in J}\mid \forall_{j\in J_{\frac{1}{2}}(F)\cap J_{\frac{1}{2}}(\tilde F)} ~b_j\in [0,a_j]\wedge\forall_{j\not\in J_{\frac{1}{2}}(F)}~ b_j=c_j\wedge\forall_{j\not\in J_{\frac{1}{2}}(\tilde F)}~ b_j=\tilde c_j\bigr\}
\]
is empty or again of the type~\eqref{eq: type of cuboid} and thus a single face of $N((a_j)_{j\in J})$. We conclude that $N((a_j)_{j\in J})$ is indeed a rectangular complex. 

Let $F$ be a face in $N((a_j)_{j\in J})$. We denote the dimension of $F$ by $d(F)$. 
By definition of the rectangular complex we have $J_1(F)\ne\emptyset$. One has $d(F)=\# J_{\frac{1}{2}}(F)$. 
Further, $F$ is a cuboid with side lengths $a_j$, $j\in J_{\frac{1}{2}}(F)$. For every face $F$ we enumerate these side lengths by 
\begin{equation}\label{eq: enumeration of side lengths}
r_1(F), \dots, r_{d(F)}(F)~\text{ such that }~r_1(F)\le \dots\le r_{d(F)}(F).
\end{equation}

\subsection{Rectangular nerves of covers}\label{sub: rectangular nerves}

We recall the definition of the rectangular nerve of a cover by balls which was introduced by Guth. 

\begin{definition}
Let $\calV=\{B_j\mid j\in J\}$ be a cover of a Riemannian manifold $W$ by open balls such that the balls $\frac{1}{2}B_j$ still cover $W$. Let $r_j$ be the radius of the ball $B_j$. The \emph{rectangular nerve} $N(\calV)$ of $\calV$ is the subcomplex of $N((r_j)_{j\in J})$ whose faces $F$ are precisely the ones for which 
\[ \bigcap_{j\in J_+(F)} B_j \ne\emptyset\]
and $J_1(F)\ne \emptyset$.  
\end{definition}

We turn to the equivariant setting with regard to the action of the fundamental group $\Gamma=\pi_1(M)$ on the universal cover~$\widetilde M$. 

\begin{lemma}\label{lem: proper Gamma CW}
	Let $J$ be a free cofinite $\Gamma$-set, and $\calV=\{B_j\mid j\in J\}$ be an equivariant cover of $\widetilde M$ by balls $B_j$ of radius $r_j$ in the sense that $\gamma B_j=B_{\gamma j}$ for every $j\in J$ and $\gamma\in \Gamma$. Then $N(\calV)$ is a locally finite rectangular complex. The left shift action 
	\[ \Gamma\acts \prod_{j\in J} [0,r_j],~\gamma\cdot\bigl(x_j\bigr)_{j\in J}=\bigl(x_{\gamma^{-1}j}\bigr)_{j\in J}\]
	restricts to 
	a proper $\Gamma$-action on $N(\calV)$ that permutes cells. Further, the barycentric subdivision of $N(\calV)$ is a proper $\Gamma$-CW-complex. \end{lemma}

\begin{proof}
Since $\Gamma$ is cofinite and and the $\Gamma$-action on $\widetilde M$ by deck transformations is proper, the cover $\calV$ is locally finite. Hence $N(\calV)$ is a locally finite CW-complex. 
Clearly, the action permutes cells, thus the action satisfies property~\eqref{eq: cell permute} of a cellular action. Each stabilizer of a cell is contained in the set-stabilizer of a finite subset of $J$, thus is a finite group. The CW-structure of $N(\calV)$ does not necessarily satisfy property~\eqref{eq: pointwise stabilizer} of a cellular action. Next we show that its barycentric subdivision does. A $k$-face of the barycentric subdivision is given by the convex hull of the barycenters of a strictly ascending chain $F_0\subset F_1\subset\dots\subset F_k$ where $F_i$ is an $i$-face of $N(\calV)$. Let $\gamma\in\Gamma$ fix the $k$-face $C$ associated with $F_0\subset F_1\subset\dots\subset F_k$ as a set. Then $\gamma$ fixes each face $F_i$ as a set. Since $F_0$ is a vertex, $\gamma$ fixes $F_0$ pointwise. By induction we may assume that $\gamma$ fixes $F_i$ pointwise for $i<k$. Since $\gamma$ fixes the $i+1$-dimensional face $F_{i+1}$ as a set and its $i$-dimensional subface $F_i$ pointwise, it must fix $F_{i+1}$ pointwise. 
Hence $\gamma$ fixes $C$ pointwise. So the $\Gamma$-action on the barycentric subdivision is cellular. The stabilizers of cells are finite as discussed above, which means that the barycentric subdivision of $N(\calV)$ is a proper $\Gamma$-CW-complex. 
\end{proof}

\section{Homological preliminaries}\label{sec: homological prelim}

In~\ref{sub: measure homology} 
we review Thurston's measure homology which is isomorphic and isometric to real singular homology on CW-complexes but has better functorial properties with regard to Cantor bundles which are considered later.  
In~\ref{sub: normed chain complexes} we discuss normed abelian groups and chain complexes. There we review integral variants ($X$-parametrised integral simplicial volume, integral foliated simplicial volume) of the simplicial volume that take an action of the fundamental group on a Cantor set or a probability space into account. In~\ref{sub: l2 betti numbers} we collect what we need from L\"uck's approach to~$\ell^2$-Betti numbers. We prove a bound on the von Neumann rank (Definition~\ref{def: von Neumann rank}) via the $X$-parametrised integral simplicial norm which slightly generalizes a bound of $\ell^2$-Betti numbers of a closed manifold by the foliated simplicial volume due 
to Schmidt~\cite{schmidt}. 

We collect some notation. 
The space of real-valued and integer-valued continuous functions on $X$ is denoted by $C(X)$ and $C(X;\bbZ)$, respectively. The action of $\Gamma$ on $X$ induces a (left) action on $C(X)$. 
Tensor products $M\otimes_\bbZ N$ over the ring $\bbZ$ are denoted by 
$M\otimes N$. The integral group ring of $\Gamma$ is denoted by $\bbZ[\Gamma]$. Modules over a (non-commutative) ring are assumed to be 
left modules unless said otherwise. Since $\bbZ[\Gamma]$ is a ring with involution -- induced by $\gamma\mapsto\gamma^{-1}$ -- we can turn any left $\bbZ[\Gamma]$-module into a right $\bbZ[\Gamma]$-module. We do implicitly so if we write $M\otimes_{\bbZ[\Gamma]} N$ for two left $\bbZ[\Gamma]$-modules. 
The singular chain complex of a space $Y$ is denoted by $C_\ast(Y)$. We write $C_\ast(Y;\bbR)$ for the singular chain complex with real coefficients. Similar for singular homology. If the group $\Gamma$ is acting continuously on $Y$ and $M$ is a $\bbZ[\Gamma]$-module, then 
we write the equivariant homology and cohomology as 
\[ H^\Gamma_p(Y;M)\defq H_p\bigl(M\otimes_{\bbZ[\Gamma]}C_\ast(Y) \bigr)\text{ and }H_\Gamma^p(Y;M)\defq H^p\bigl(\hom_{\bbZ[\Gamma]}(C_\ast(Y), M)\bigr).\]
The projection $\widetilde M\to M$ yields a canonical isomorphism $H_\ast^\Gamma(\widetilde M;\bbZ)\xrightarrow{\cong} H_\ast(M;\bbZ)$.

\subsection{Measure homology}\label{sub: measure homology}

Measure homology replaces the finite linear combinations in singular homology by signed measures on the space of singular simplices. It was invented by Thurston. We recall its basic notions. For more details we refer to~\cite{loeh-measure}. 

For a topological space $N$ we endow the space of continuous maps $\map(\Delta^n, N)$ from the standard $n$-simplex to $N$ with the compact-open topology. 
We define $\calC_n(N)$ as the $\bbR$-vector space of signed Borel measures on $\map(\Delta^n, N)$ that have compact support and finite variation. 
The elements of $\calC_n(N)$ are called \emph{measure chains}. 
The alternating sum of pushforwards of face maps turn $\calC_\ast(N)$ into a chain complex whose homology is called the \emph{measure homology} of $N$. 

The variation of measures induces a seminorm on the measure homology. 
The map from the singular chain complex to the chain complex of measure chains $C_\ast(N)\to \calC_\ast(N)$ that sends a singular $n$-simplex to the point measure on that simplex is a natural chain homomorphism, which induces an isometric isomorphism in homology~\cite{loeh-measure}.

\subsection{Norms on abelian groups and chain complexes}\label{sub: normed chain complexes}

We consider norms and seminorms on $\bbR$-modules, i.e.~real vector spaces, and on $\bbZ$-modules, i.e.~abelian groups. The defining properties of a (semi-)norm on an $\bbR$-module make sense for a 
$\bbZ$-module~$A$ and a function $|\_|\colon A\to \bbR^{\ge 0}$ with the slight modification that $|r\cdot a|_A=|r|\cdot |a|_A$ is only required for $r\in\bbZ$ and $a\in A$. 

\begin{theorem}[\cite{steprans}]
An abelian group endowed with a norm that induces the discrete topology is free. 
\end{theorem}

The supremum norm on the abelian group $C(X;\bbZ$) of integer valued continuous functions on $X$ induces the discrete topology. We record 
the following consequence for later use. 

\begin{corollary}\label{cor: continuous function free}
The abelian group $C(X;\bbZ)$ is free. 
\end{corollary}

A chain complex of $\bbZ$- or $\bbR$-modules equipped with a (semi-)norm on each chain group  is called a \emph{(semi-)normed chain complex} provided the boundary maps are continuous. 
We endow the quotient of a (semi-)normed module with the quotient semi-norm. In general, a norm does not induce 
a norm on the quotient but only a semi-norm. 
In the context of semi-norms being \emph{isometric} does not imply being \emph{injective}. 

The singular chain complexes $C_\ast(N)$ and $C_\ast(N;\bbR)$ of a topological space $N$ with integer or real coefficients, respectively, are normed via the $\ell^1$-norm with respect to the basis by singular simplices. They induce semi-norms on $H_\ast(N)$ and $H_\ast(N;\bbR)$, respectively. The latter is denoted by $\norm{\_}$ and called \emph{simplicial norm}. The induced chain homomorphism and homology homomorphism of a map of spaces do not increase the simplicial norms.  
Gromov and Thurston defined the \emph{simplicial volume} of a closed manifold~$M$ as the simplicial norm of its fundamental class~$[M]$. 
We denote it by $\norm{M}$. 

The $\ell^1$-norm on $C(X;\bbZ)$ with respect to the measure~$\mu$ and 
the simplicial norm induce the following norm on each abelian group $C(X;\bbZ)\otimes C_p(N)$ which we call the \emph{$X$-parametrised integral simplicial norm} and denote by $\norm{\_}_\bbZ^X$: For functions $f_1,\dots, f_k\in C(X;\bbZ)$ and distinct singular $p$-simplices $\sigma_1,\dots,\sigma_k$ we set
\[ \norm{f_1\otimes \sigma_1+\dots+f_k\otimes\sigma_k}_\bbZ^X\defq\int_X|f_1|d\mu+\dots+\int_X|f_k|d\mu.\]

Let us now consider the situation where $N$ is a topological space endowed with the action of a group $\Gamma$. Then $C_\ast(N)$ is a chain complex over the group ring $\bbZ[\Gamma]$. We obtain an induced semi-norm on the quotient $C(X;\bbZ)\otimes_{\bbZ[\Gamma]}C_p(N)$ of $C(X;\bbZ)\otimes C_p(N)$ which we call by the same name and denote by the same symbol. 

\begin{definition}\label{def: inclusion into parametrised chains}
Let $Y$ be a connected space with fundamental group~$\Gamma$ and universal cover~$\widetilde Y$. The composition of chain maps 
\[ C_\ast(Y)\xleftarrow{\cong} \bbZ\otimes_{\bbZ[\Gamma]}C_\ast(\widetilde Y)\hookrightarrow C(X;\bbZ)\otimes_{\bbZ[\Gamma]}C_\ast(\widetilde Y)\]
is denoted by $j_\ast^Y$. Here the right hand map is induced by the inclusion of constant functions. 
\end{definition}

\begin{remark}\label{rem: comparision parametrised and real norm}
   Let $i_\ast^\bbR$ be the change of coefficients $C_\ast(Y)\to C_\ast(Y;\bbR)$. We have 
   \[\norm{j_\ast^Y(z)}_\bbZ^X\le \norm{i_\ast^\bbR(z)}\]
   for every chain $z$ in $C_\ast(Y)$ and thus a similar statement for every homology class. This follows from the fact that invariant measure~$\mu$ on $X$ yields by integration a chain map 
   \[ C(X;\bbZ)\otimes_{\bbZ[\Gamma]}C_\ast(\widetilde Y)\to \bbR\otimes_{\bbZ[\Gamma]} C_\ast(\widetilde Y)\xrightarrow{\cong}C_\ast(Y;\bbR)\]
   that does not increase norms. 
\end{remark}

\begin{definition}\label{def: parametrised integral simplicial volume}
   The \emph{$X$-parametrised integral simplicial volume} of a connected closed oriented manifold $M$ with fundamental group~$\Gamma$ is defined as \[\norm{M}_\bbZ^X:= \norm{j^M_\ast([M])},\]
  where $[M]\in H_d(M)$ is the fundamental class.  
   \end{definition}
  
  Note that we take the liberty to skip the dependency on the measure~$\mu$ in the notation of the $X$-parametrised integral simplicial volume. 
  
  \begin{remark}[Relation to integral foliated simplicial volume]\label{rem: relation to foliated simplicial volume}
  Let us denote the free and probability measure preserving (\emph{pmp}) action of $\Gamma$ on $(X,\mu)$ by $\alpha$. Then $\norm{M}_\bbZ^X$ only depends on the measure isomorphism class of $\alpha$, 
  and $\norm{M}_\bbZ^X$ coincides with the $\alpha$-parametrised simplicial volume $|M|^\alpha$ 
  as defined in~\cite{foliated}*{Definition~2.2}. 
  The \emph{integral foliated simplicial volume} is defined as the infimum of $\alpha$-parametrised simplicial volumes over all free measurable pmp actions $\alpha$ of~$\Gamma$, and is thus bounded from above by the $X$-parametrised integral simplicial volume. We refer to~\cites{foliated,schmidt} for more details. 
\end{remark}

\subsection{$\ell^2$-Betti numbers}\label{sub: l2 betti numbers}

We use L\"uck's approach to $\ell^2$-Betti numbers which is based on the 
dimension function for modules over finite von Neumann algebras. This is not just a matter of taste as it is important in our context to work with singular chains and to be able to read off $\ell^2$-Betti numbers from the singular chain complex instead of the simplicial chain complex. 

L\"uck~\cite{lueck-dimension} defines a dimension function $\dim_A$ taking values in $[0,\infty]$ for arbitrary modules over a  
von Neumann algebra $A$ with a finite trace, where $A$ is regarded just as a ring, not as functional-analytic object. 
Our most important example is the \emph{group von Neumann algebra} $L(\Gamma)$ with its canonical trace. The complex group ring $\bbC[\Gamma]$ is a subring of $L(\Gamma)$. The trace of an element 
in $\bbC[\Gamma]$ is the coefficient of $1_\Gamma$. The involution of $\bbC[\Gamma]$ induced by complex conjugation and taking inverses extends to an involution of $L(\Gamma)$ which corresponds to taking adjoint operators. In particular, we can turn any left $L(\Gamma)$-module into a right $L(\Gamma)$-module via this involution. 
The $p$-th \emph{$\ell^2$-Betti number of a 
$\Gamma$-space} $Y$ is then defined as 
\[\betti_p(Y;\Gamma)\defq \dim_{L(\Gamma)}H_p^\Gamma\bigl(Y;L(\Gamma)\bigr).\]
In the case of the universal covering $\widetilde M\to M$ and $\Gamma=\pi_1(M)$ we simply write $\betti_p(M)$ instead of $\betti_p(\widetilde M; \Gamma)$ and call it the \emph{$p$-th $\ell^2$-Betti} number of~$M$. In the case of Riemannian manifolds and simplicial complexes the above definitions coincide with those by Atiyah and Dodziuk, respectively. For more information and proofs we refer to L\"uck's book~\cite{lueck-l2book}.  

Next we describe another von Neumann algebra whose relevance 
to $\ell^2$-Betti numbers became clear in the work of Gaboriau~\cite{gaboriau}. 
The probability space $(X,\mu)$ from Theorem~\ref{thm: existence of action} gives rise to the abelian von Neumann algebra $L^\infty(\mu)$ of complex-valued measurable functions on $X$ with the integral as finite trace. The measure preserving action of $\Gamma$ induces a unitary $\Gamma$-action on $L^\infty(\mu)$. One can then form the \emph{crossed product von Neumann algebra} $\neumann$ which contains $L(\Gamma)$ and $L^\infty(\mu)$ as subalgebras and which possesses a (unique) finite trace that extends 
those of $L(\Gamma)$ and $L^\infty(\mu)$. For $\gamma\in\Gamma\subset\bbC[\Gamma]\subset L(\Gamma)$ and $f\in L^\infty(\mu)$ we have 
\[ \gamma\cdot f=f\bigl(\gamma^{-1}\_\bigr)\cdot\gamma\in \neumann.\]
The involution on $\neumann$ extends the one of $L(\Gamma)$ and the complex conjugation on $L^\infty(\mu)$. We indicate the involution in all cases with a bar. 
We refer for more information 
to~\cites{gaboriau,sauer-groupoid}.

The following 
theorem was suggested by ideas of Connes and Gromov and was proved 
in the PhD thesis of Schmidt~\cite{schmidt}. 

\begin{theorem}\label{thm: betti bound}
Every $\ell^2$-Betti number of a closed oriented manifold is bounded from above by its \mbox{$X$-para}\-metrised integral simplicial volume. 
\end{theorem}

We formulate a slightly more general version (Theorem~\ref{thm: bound von Neumann rank})  based on the notion of \emph{von Neumann rank} which is defined below.  The proof of Theorem~\ref{thm: bound von Neumann rank} can be extracted from Schmidt's proof of Theorem~\ref{thm: betti bound}. To make it easier for the reader we present a proof of Theorem~\ref{thm: bound von Neumann rank} which is a streamlined version of Schmidt's method. 

Some preparations are in order. 
Let $C_\ast$ be a chain complex of left $\bbZ[\Gamma]$-modules. We denote by $C^{-\ast}$ the chain complex whose $p$-th chain module is $\hom_{\bbZ[\Gamma]}\bigl(C_{-p},\bbZ[\Gamma]\bigr)$ with the induced differential. We may extend chain complexes that are indexed over non-negative degrees like the singular chain complex to all degrees in $\bbZ$ by setting them zero in negative degrees.  Since the group ring is a ring with involution we may regard te module $C^{-\ast}$ which is naturally a right $\bbZ[\Gamma]$-module as a left $\bbZ[\Gamma]$-module. Let $D_\ast$ be another $\bbZ[\Gamma]$-chain complex. We consider the following commutative diagram of $\bbZ$-chain complexes: 
\begin{equation}\label{eq: chain complex square}
\begin{tikzcd}
   \bbZ\otimes_{\bbZ[\Gamma]} \bigl(C_\ast\otimes_{\bbZ} D_\ast\bigr)\ar[d,hook]\ar[r]& \hom_{\bbZ[\Gamma]}\bigl(C^{-\ast}, D_\ast\bigr)\ar[d,hook]\\
L^\infty(\mu)\otimes_{\bbZ[\Gamma]}\bigl(C_\ast\otimes_{\bbZ} D_\ast\bigr)\ar[r] & \hom_{\neumann}\bigl (\neumann\otimes_{\bbZ[\Gamma]}C^{-\ast}, \neumann\otimes_{\bbZ[\Gamma]}D_\ast\bigr)
\end{tikzcd}
\end{equation}
The tensor product of chain complexes $C_\ast\otimes D_\ast$ is itself a $\bbZ[\Gamma]$-chain complex via the diagonal $\Gamma$-action. 
The complex on the upper right is the hom-complex; its 
$p$-th chain group consists of chain maps $C^{-\ast}\to D_\ast$ of degree $p$; its $p$-th homology consists of the group of chain homotopy classes of degree~$p$ chain maps, which we denote by $[C^{-\ast}, D_\ast]$. We refer to~\cite{brown}*{I.0} for a detailed description of these standard constructions of chain complexes. 
The left vertical map comes from the inclusion of constant functions. The right vertical map is the induction from $\bbZ[\Gamma]$ to $\neumann$. The upper horizontal arrow sends $1\otimes x\otimes y$ to the map $g\mapsto\overline{g(x)}\cdot y$ for $g\in C^{-\ast}$. The lower horizontal arrow is the map 
\begin{equation}\label{eq: lower horizontal arrow}
f\otimes x \otimes y\mapsto \Bigl(a\otimes g\mapsto a\cdot \overline{f\cdot g(x)}\otimes y   \Bigr)
\end{equation}
To verify that this map is well defined we check that $f(\gamma\_)\otimes x\otimes y$ and $f\otimes \gamma x\otimes \gamma y$ have the same image. This follows from 
\begin{align*}
   a\overline{f(\gamma\_)g(x)}\otimes y=
   a\overline{f(\gamma\_)g(x)}\gamma^{-1}\otimes \gamma y=
   a\overline{\gamma f(\gamma\_)g(x)}\otimes \gamma y
   &= a\overline{f(\_)\gamma g(x)}\otimes \gamma y\\
   &= a\overline{f(\_)g(\gamma x)}\otimes \gamma y.
\end{align*}
We leave the verification of the property of being a chain map to the reader. 

Next let $Y$ be a topological space with a free $\Gamma$-action. Set $C_\ast=D_\ast=C_\ast(Y)$. Let 
$A_\ast\colon C_\ast(Y\times Y)\to C_\ast(Y)\otimes_\bbZ C_\ast (Y)$ 
be the Alexander-Whitney map, and let $\Delta_\ast\colon C_\ast(Y)\to C_\ast(Y\times Y)$ be the map induced by the diagonal embedding. If we compose the horizontal maps in the commutative square above with the chain maps 
$\id_\bbZ\otimes_{\bbZ[\Gamma]}A_\ast\circ \Delta_\ast$ and 
$\id_{L^\infty(\mu)}\otimes_{\bbZ[\Gamma]}A_\ast\circ\Delta_\ast$, respectively, and take homology, we obtain the following commutative square. The left vertical map is induced by the inclusion of 
constant functions. 
\begin{equation}\label{eq: cup product}
\begin{tikzcd}
H_d(\Gamma\bs Y)\cong H_d^\Gamma(Y; \bbZ)\arrow[d]\arrow[r,"\cap\_"] & \bigl[C^{-\ast}(Y),C_{d+\ast}(Y)\bigr]\arrow[d]\\
H_d^\Gamma\bigl(Y; L^\infty(\mu)\bigr)\arrow[r]& \bigl[ \neumann\otimes_{\bbZ[\Gamma]}C^{-\ast}(Y), \neumann\otimes_{\bbZ[\Gamma]}C_{d+\ast}(Y)\bigr]
\end{tikzcd}
\end{equation}
The upper and lower horizontal maps are variants of the cap product.  

\begin{definition}\label{def: von Neumann rank}
Let $Z$ be a connected space with fundamental group~$\Gamma$. 
Let $x\in H_d(Z)$. 
The \emph{von Neumann rank} of $x$ is defined as 
\[\dim_{L(\Gamma)}\Bigl(\bigoplus_{n\ge 0} \im\Bigl(H^n_\Gamma\bigl(\widetilde Z;L(\Gamma)\bigr)\xrightarrow{x\cap\_} H_{d+n}^\Gamma\bigl(\widetilde Z;L(\Gamma)\bigr)\Bigr)\Bigr).
\] 
\end{definition}

\begin{remark}\label{rem: equivariant duality}
Let $M$ be a closed oriented $d$-manifold. The sum of the $\ell^2$-Betti numbers of $M$ is the von Neumann rank of its fundamental class. This is direct consequence of (equivariant) Poincare duality which says that the image of the fundamental class under the upper horizontal map in~\eqref{eq: cup product} is chain homotopy equivalence. 
\end{remark}

\begin{theorem}\label{thm: bound von Neumann rank}
Let $Z$ be a connected space with fundamental group~$\Gamma$.  
 Then the von Neumann rank of a homology class $[x]\in H_d(Z)$ is bounded from above 
by $d\cdot \norm{[j_\ast^Z(x)]}_\bbZ^X$. 
\end{theorem}

\begin{proof}
Suppose the image of $x$ under the map induced by inclusion of constant functions 
is homologous to a cycle $\sum_{k=1}^m a_k\otimes \sigma_k$ where 
$a_k\in C(X;\bbZ)$ and $\sigma_k$ is a singular $d$-simplex in $\widetilde Z$. 
Via the embedding $C(X;\bbZ)\hookrightarrow L^\infty(\mu)$ we obtain a 
homology class $[\sum_{k=1}^m a_k\otimes \sigma_k]\in H_d^\Gamma(\widetilde Z;L^\infty(\mu))$. The cap product with $\sum_{k=1}^m a_k\otimes \sigma_k$, that is the image of $\sum_{k=1}^m a_k\otimes \sigma_k$ under the lower horizontal map of~\eqref{eq: cup product}, is represented by the $\neumann$-chain homomorphism whose degree $i$ part is 
\begin{equation*}
   \neumann\otimes_{\bbZ[\Gamma]} C^{i}(\widetilde Z)\to\neumann\otimes_{\bbZ[\Gamma]} C_{d-i}(\widetilde Z),~~
   1\otimes g\mapsto \sum_{k=1}^m \overline{a_k\cdot g(\sigma_k\rfloor_i)}\otimes \sigma_k\lfloor_{d-i}.
\end{equation*}
Here $\sigma\rfloor_i$ and $\sigma\lfloor_{d-i}$ denote 
the front $i$-face and the back $(d-i)$-face of $\sigma$
respectively.
It clearly factorizes over the $\neumann$-homomorphism 
\begin{equation*}
\neumann\otimes_{\bbZ[\Gamma]}C^i(\widetilde Z)\to \bigoplus_{k=1}^m \neumann\cdot\chi_{\supp(a_k)},~~
y\otimes g\mapsto \bigl( y\overline{a_k g(\sigma_k\rfloor_i)}\bigl)_k=\bigl( y\overline{g(\sigma_k\rfloor_i)}a_k\bigl)_k.
\end{equation*}
Further, we have 
\[ \dim_{\neumann}\Bigl( \bigoplus_{k=1}^m \neumann\cdot\chi_{\supp(a_k)}\Bigr)=\sum_{k=1}^m\mu(\supp(a_k))\le \bigl\lVert \sum_{k=1}^m a_k\otimes \sigma_k\bigr\rVert_\bbZ^X.\]
The last inequality uses the fact that $a_k$ is integer-valued. 
Hence the $\neumann$-dimension of the 
image of the cap product with $[\sum_{k=1}^m a_k\otimes \sigma_k]$ 
\begin{equation}\label{eq: induced image}
\bigoplus_{n\ge 0}\im\Bigl(H^n_\Gamma(\widetilde Z;\neumann)\to H^\Gamma_{d+n}(\widetilde Z;\neumann)\Bigr)
\end{equation}
is bounded by~$\norm{[j_\ast^Z(x)]}_\bbZ^X$. It remains to verify that the 
$\neumann$-dimension of~\eqref{eq: induced image} is the von Neumann rank of~$[x]$. Since $L(\Gamma)\subset \neumann$ is a flat ring extension~\cite{sauer-groupoid}*{Theorem~4.3}, we obtain that 
\begin{multline*}
      \bigoplus_{n\ge 0}\im\Bigl(H^n_\Gamma(\widetilde Z;\neumann)\to H^\Gamma_{d+n}(\widetilde Z;\neumann)\Bigr)\\
      \cong \neumann\otimes_{L(\Gamma)}\bigoplus_{n\ge 0}\im\Bigl(H^n_\Gamma(\widetilde Z; L(\Gamma))\to H^\Gamma_{d+n}(\widetilde Z;L(\Gamma))\Bigr)
   \end{multline*}
where the maps on the right hand side are induced by the cap product with $[x]$. Since the von Neumann dimension is compatible with induction~\cite{sauer-groupoid}*{Theorem~2.6}, the proof is finished. 
\end{proof}

\section{The category of Cantor bundles}\label{sec: category cantor bundles}

In~\ref{sub: cantor bundles} we introduce the central notion of a Cantor bundle. A Cantor bundle comes with a map to the Cantor set~$X$. In general, a Cantor bundle is not a locally trivial bundle over~$X$. See Example~\ref{exa: non trivial Cantor bundle}. It is, however, locally trivial, when restricted to compacta (see Lemma~\ref{lem: decomposition into boxes}). Metric Cantor bundles are also introduced which are Cantor bundle whose fibers over $X$ come with the structure of a metric space. In~\ref{sub: cantor bundle maps} we define Cantor bundle maps. We also consider pushouts of Cantor bundles. 
In~\ref{sub: from chains to measure chains} we study the functoriality of certain chain complexes attached to Cantor bundles.

\subsection{Cantor bundles}\label{sub: cantor bundles}

We define the data of a product atlas for a space over $X$. 

\begin{definition}
Let $W$ be a topological space and $\pr\colon W\to X$ a continuous map to the Cantor set. We denote the fiber over $x\in X$ by $W_x$. For $A\subset X$ we write $W\vert_A$ for $\pr^{-1}(A)\subset W$.  
We introduce the notion of a \emph{product atlas} for $W$:
\begin{itemize}
\item A \emph{product chart} for $W$ consists of a clopen subset $A\subset X$, an open subset $U\subset W$, a space $F$, and a homeomorphism 
$U\to A\times F$ over $A$. 
\item Two product charts $c_i\colon U_i\to A_i\times F_i$, $i\in\{1,2\}$ are \emph{compatible} if 
there are subspaces $F_i'\subset F_i$ such that $c_i(U_1\cap U_2)=(A_1\cap A_2)\times F_i'$ and the transition map \[c_2\circ c_1^{-1}\colon (A_1\cap A_2)\times F_1'\to  (A_1\cap A_2)\times F_2'\] is a product of $\id_{A_1\cap A_2}$ and a homeomorphism $g\colon F_1'\to F_2'$. 
\item A \emph{product atlas} for $W$ consists of a family of compatible product charts whose domains cover $W$; it is \emph{maximal} if it contains every product chart that is compatible with the product charts of the atlas. 
\end{itemize}
\end{definition}

If in the definition of a product atlas we would replace the Cantor set by a connected space then the existence of a product atlas for $W\to X$ would imply that $W\to X$ is trivial. This in stark contrast to our situation. 

\begin{definition}
Let $\pr\colon W\to X$ be a topological space over $X$ endowed with a maximal product atlas. 
\begin{itemize}
\item A relatively compact subset $K\subset W$ 
is called a \emph{box} if there is a product chart $c\colon U\to A\times F$ such that $K\subset U$ and 
$c(K)=A\times F'$ for a subspace $F'\subset F$. 
\item For a box $K$ and for all $x,y\in A\defq\pr(K)$ the map 
\[ \tau_{x,y}\colon K_x\xrightarrow{\cong} K_y\]
defined by 
\[ \tau_{x,y}(p)=c^{-1}\bigl( y, \pr_2(c(p))\bigr)\]
for a choice of product chart $K\subset U\xrightarrow{c} A\times F$ is independent of the choice of chart. We say that $\tau_{x,y}$ is the \emph{parallel transport} inside the box $K$. 
\end{itemize}
\end{definition}

\begin{lemma}\label{lem: decomposition into boxes}
Let $\pr\colon W\to X$ be a locally compact Hausdorff space over $X$ endowed with a maximal product atlas. 
For every compact subset $K\subset W$ there is a relatively compact, open subset $L\subset W$ containing $K$ and a clopen partition $X=A_1\cup\dots\cup A_n$ such that 
$L\vert_{A_i}$ is a box for each $i\in\{1,\dots,n\}$. Further, if $K$ is a finite union of open boxes, we may choose $L=K$. 
\end{lemma}

\begin{proof}
Since $W$ is locally compact, each point lies in an open box. Since $K$ is compact it is covered by finitely many open boxes $B_1, \dots, B_n$. Let $L$ be the union of these boxes. Every box is relatively compact, and so is $L$. 
Since the clopen subsets of $X$ form a set algebra, there is a clopen partition $A_1, \dots, A_m$ of $\pr(L)$ that is subordinate  
to $\pr(B_1), \dots, \pr(B_n)$. 

We claim that $L\vert_{A_i}$ is box: Pick $x_0\in A_i$. We construct a product chart
\[ f_i\colon L\vert_{A_i}\to A_i\times (W_{x_0}\cap L)\] 
as follows. Every $p\in L_x\subset L\vert_{A_i}$ lies in a box $B_k$ with $A_i\subset \pr(B_k)$. We set $f_i(p)=(x, \tau_{x,x_0}^{B_k}(p))$ where $\tau_{x,x_0}^{B_k}$ is the parallel transport inside $B_k$. We have 
\[  \tau_{x,x_0}^{B_k}(p)=\tau_{x,x_0}^{B_l}(p)\text{ for $p\in W_x\cap B_k\cap B_l$ and $\{x,x_0\}\subset \pr(B_k\cap B_l)$},  \]
hence $f_i$ is well defined. The map $f_i$ is a homeomorphism, and its inverse maps $(x, q)$ to $\tau_{x_0,x}^{B_k}(q)$ for any box $B_k$ with $q\in B_k$. 

Since $f_i$ is compatible with all the boxes $B_k$  and its domain is covered by them, $f_i$ lies in the maximal product atlas. So $L\vert_{A_i}$ is box. Moreover, we can add the complement of $\pr(L)$ to the clopen partition above to get the statement of the lemma. 
\end{proof}

\begin{definition}[Cantor bundle]
A \emph{Cantor bundle} is a locally compact Hausdorff 
space~$W$ endowed with a continuous proper $\Gamma$-action and a continuous $\Gamma$-equivariant map 
$\pr\colon W\to X$ and a maximal product atlas such that the $\Gamma$-action on $W$ has a Borel fundamental domain that is a union of finitely many boxes. 
\end{definition}

Note that the action on a Cantor bundle is automatically free since it lies over the free action on $X$. 

\begin{definition}
Let $\pr\colon W\to X$ be a Cantor bundle, and let $V\subset W$ be a $\Gamma$-invariant subspace so that 
for every $p\in V$ there is a product chart $U\to A\times F$ such that $p\in U$ and $U\cap V$ is a box. Then we call $\pr\vert_V\colon V\to X$ a \emph{Cantor subbundle} of $\pr\colon W\to X$. 
\end{definition}

\begin{lemma}
Let $\calA$ be a maximal product atlas of a Cantor bundle $W\to X$. 
A~Cantor subbundle $V\subset W$ is a Cantor bundle with respect to the product atlas 
\[ \calA_V\defq \bigl\{ U\xrightarrow{c} A\times F\mid c\in\calA, U\cap V\text{ is a box}\bigr\}.\]
\end{lemma}

\begin{proof}
The only non-obvious statement is the existence of a fundamental domain for $V$ consisting of finitely many boxes. Let $D\subset W$ be a fundamental domain for $W$ which is a union of boxes $B_1,\dots, B_n$. 
Note that $D$ is relatively compact as every $B_i$ is relatively compact. Then $V\cap D$ is a fundamental 
domain for $V$. At every point $p\in V\cap D$ we can choose a product chart with a domain $U_p\ni p$ such 
that $U_p\cap V$ is a box. By relative compactness we can cover $V\cap D$ with finitely many such product chart 
domains $U_{p_1}, \dots, U_{p_m}$. Every set 
\[U_{p_i}\cap V\cap D =\bigcup_{j=1}^n U_{p_i}\cap V\cap B_j\]
is a box, hence $V\cap D$ is a union of such. 
\end{proof}

\begin{example}\label{exa: cell}
Let $A\subset X$ be a clopen subset and $Y$ any compact space. We consider the trivial Cantor bundle $X\times (\Gamma\times Y)$ with the projection to the first factor and endowed with the $\Gamma$-action 
\[ \gamma\cdot (x,\gamma', y)=(\gamma x, \gamma \gamma', y).\]
Then $A\times \Gamma\times Y$ endowed with the projection 
\[ A\times\Gamma\times Y\to X,~(a,\gamma, y)\mapsto \gamma a\] 
and the left translation $\Gamma$-action on the second factor 
is a Cantor subbundle via the embedding
\[ A\times\Gamma\times Y\hookrightarrow X\times\Gamma\times Y,~(a,\gamma, y)\mapsto (\gamma a, \gamma, y).\]
\end{example}

\begin{remark}[Finite isotropy disappears]\label{rem: finite isotropy}
Let $H<\Gamma$ be a finite subgroup. Then $X\times \Gamma/H$ is a Cantor bundle 
with the projection to $X$ and the diagonal $\Gamma$-action. Let $A\subset X$ be a Borel fundamental domain for the $H$-action on~$X$. Then 
\[ X\times \Gamma/H\cong A\times \Gamma\]
are isomorphic as Cantor bundles where the latter is the one from the previous example.
An isomorphism is given by 
\[ A\times\Gamma\to X\times \Gamma/H,~(a,\gamma)\mapsto (\gamma a, \gamma H).\]
\end{remark}

\begin{example}[Non-trivial Cantor bundle]\label{exa: non trivial Cantor bundle}
We describe an example of a Cantor bundle whose fibers exhibit uncountably many homeomorphism types. In particular, it is not a trivial Cantor bundle. 
Let $\Gamma=\bbZ$ and let $X$ be a minimal subshift of the shift action of $\bbZ$ on $\{0,1\}^\bbZ$ such that the $\bbZ$-action on the Cantor set~$X$ is free. Such a minimal subshift exists due to~\cite{glasner}*{Theorem~4.2}. 
Let $L$ be the following infinite $1$-dimensional simplicial complex: 
\begin{figure}[h]
\begin{tikzpicture}
\node[circle, fill, scale=0.5] (s0) at (0,0) {};
\node[below right=1mm of s0] {-2};
\node[circle,fill, scale=0.5]  (s1) at (2,0) {};
\node[below right=1mm of s1] {-1};
\node[circle,fill, scale=0.5]  (s2) at (4,0) {};
\node[below right=1mm of s2] {0};
\node[circle,fill, scale=0.5]  (s3) at (6,0) {};
\node[below right=1mm of s3] {1};
\node[circle,fill, scale=0.5]  (s4) at (8,0) {};
\node[below right=1mm of s4] {2};
\path[thick] 
   (s0) edge node[above] {} (s1)
   (s1) edge node[above] {} (s2)
   (s2) edge node[above] {} (s3)
   (s3) edge node[above] {} (s4);
\draw (s0) edge (s0.south);

\node [left=of s0] {} edge[thick, dashed] (s0);
\node [right=of s4] {} edge[thick, dashed] (s4);

\node [draw,circle, fill, scale=0.5] [below=4mm of s0] {} edge (s0); 
\node [draw,circle, fill, scale=0.5] [above right=4mm of s0] {} edge (s0); 
\node [draw,circle, fill, scale=0.5] [above left	=4mm of s0] {} edge (s0); 

\node [draw,circle, fill, scale=0.5] [below=4mm of s1] {} edge (s1); 
\node [draw,circle, fill, scale=0.5] [above right=4mm of s1] {} edge (s1); 
\node [draw,circle, fill, scale=0.5] [above left	=4mm of s1] {} edge (s1); 

\node [draw,circle, fill, scale=0.5] [below=4mm of s2] {} edge (s2); 
\node [draw,circle, fill, scale=0.5] [above right=4mm of s2] {} edge (s2); 
\node [draw,circle, fill, scale=0.5] [above left	=4mm of s2] {} edge (s2); 

\node [draw,circle, fill, scale=0.5] [below=4mm of s3] {} edge (s3); 
\node [draw,circle, fill, scale=0.5] [above right=4mm of s3] {} edge (s3); 
\node [draw,circle, fill, scale=0.5] [above left	=4mm of s3] {} edge (s3); 

\node [draw,circle, fill, scale=0.5] [below=4mm of s4] {} edge (s4); 
\node [draw,circle, fill, scale=0.5] [above right=4mm of s4] {} edge (s4); 
\node [draw,circle, fill, scale=0.5] [above left	=4mm of s4] {} edge (s4); 

\end{tikzpicture}
\caption{The simplicial complex~$L$.}
\end{figure}
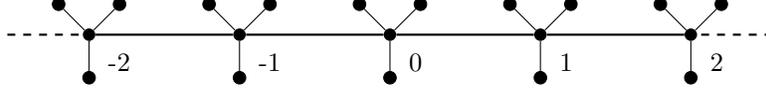
We have an obvious $\bbZ$-action on $L$ by translation. 
For $x=(x_i)\in X\subset \{0,1\}^\bbZ$ let $L_x\subset L$ be the subcomplex that consists of the horizontal line and of an upward caret at each $n$ with $x_n=1$   and a downward segment at each $n$ with $x_n=0$. 

\begin{figure}[h]
\begin{tikzpicture}
\node[circle, fill, scale=0.5] (s0) at (0,0) {};
\node[below right=1mm of s0] {-2};
\node[circle,fill, scale=0.5]  (s1) at (2,0) {};
\node[below right=1mm of s1] {-1};
\node[circle,fill, scale=0.5]  (s2) at (4,0) {};
\node[below right=1mm of s2] {0};
\node[circle,fill, scale=0.5]  (s3) at (6,0) {};
\node[below right=1mm of s3] {1};
\node[circle,fill, scale=0.5]  (s4) at (8,0) {};
\node[below right=1mm of s4] {2};
\path[thick] 
   (s0) edge node[above] {} (s1)
   (s1) edge node[above] {} (s2)
   (s2) edge node[above] {} (s3)
   (s3) edge node[above] {} (s4);
\draw (s0) edge (s0.south);

\node [left=of s0] {} edge[thick, dashed] (s0);
\node [right=of s4] {} edge[thick, dashed] (s4);

\node [draw,circle, fill, scale=0.5] [above right=4mm of s0] {} edge (s0); 
\node [draw,circle, fill, scale=0.5] [above left	=4mm of s0] {} edge (s0); 

\node [draw,circle, fill, scale=0.5] [above right=4mm of s1] {} edge (s1); 
\node [draw,circle, fill, scale=0.5] [above left	=4mm of s1] {} edge (s1); 

\node [draw,circle, fill, scale=0.5] [below=4mm of s2] {} edge (s2); 

\node [draw,circle, fill, scale=0.5] [above right=4mm of s3] {} edge (s3); 
\node [draw,circle, fill, scale=0.5] [above left	=4mm of s3] {} edge (s3); 

\node [draw,circle, fill, scale=0.5] [below=4mm of s4] {} edge (s4); 

\end{tikzpicture}
\caption{$L_x$ for $x=(\dots, 1, 1, \mathbf{0}, 1, 0,\dots)\in\{0,1\}^\bbZ$.}
\end{figure}
Then 
\[ W=\bigl\{ (x, p)\mid p\in L_x\}\subset X\times L\]
is a Cantor subbundle of the trivial Cantor bundle $X\times L$: The diagonal $\bbZ$-action on $X\times L$ restricts to $W$. For every $x=(x_i)\in X$ we consider the clopen neighborhood of $x$ 
\[ A_x(n)=\bigl\{ (y_i)\in X\mid \text{  $y_i=x_i$ for $i\in\{-n, \dots, n\}$}\bigr\},\] 
and let $L_x(n)\subset L_x$ be the finite subgraph obtained from $L_x$ by cutting off the horizontal line at $-n$ and $n$. 
We then have 
\[ W\cap \bigl(A_x(n)\times L_x(n)\bigr)= A_x(n)\times L_x(n)\subset X\times L.\]
Running through $x\in X$ and $n\in\bbN$ we cover all of $W$. So $W$ is a Cantor subbundle. 

Since the valency of $L_x$ at each vertex is at least~$3$, two fibers $L_x$ and $L_y$ are homeomorphic if and only if they are simplicially isomorphic. 
Since $X$ is uncountable, $W$ has uncountably many homeomorphism types of 
fibers. 
\end{example}

\begin{definition}[Metric and Riemannian Cantor bundles]\hfill\\
A Cantor bundle $\pr\colon W\to X$ is \emph{metric} if 
\begin{itemize}
\item each fiber $W_x$ is endowed with a metric 
inducing the topology of $W_x$, and   
\item  the maps $W_x\to W_{\gamma\cdot x}$ induced by multiplication are isometries for every $x\in X$ and every $\gamma\in \Gamma$.
\item for each product chart $c\colon U\to A\times F$ there is a metric on $F$ such that $c$ is fiberwise an isometry. 
\end{itemize}
If, in addition, each fiber $W_x$ is a $d$-dimensional Riemannian manifold with the induced Riemannian metric, we say that $\pr\colon W\to X$ is a \emph{$d$-dimensional Riemannian Cantor bundle}. 

Finally, \emph{metric Cantor subbundles} are defined similarly to Cantor subbundles. 
\end{definition}

\begin{example}\label{exa: manifold example}
The product space $\xxm$ with the diagonal $\Gamma$-action is a Riemannian Cantor bundle. 
Each fiber $\{x\}\times \widetilde{M}\cong\widetilde{M}$ carries the Riemannian metric lifted from the Riemannian 
metric from $M$. The maximal product atlas is defined to be the set of all product charts that are compatible with 
$\id\colon \xxm\to\xxm$. The $\Gamma$-action possesses a relatively compact Borel fundamental domain $F\subset\widetilde M$. Then the box $X\times F$ is a Borel fundamental domain of the $\Gamma$-action on~$\xxm$.  
\end{example}

\subsection{Cantor bundle maps}\label{sub: cantor bundle maps}

To obtain a category of Cantor bundles we define the morphisms next. 

\begin{definition}\label{def: product-like map}
Let $V$ and $W$ be topological spaces over $X$ endowed with maximal product atlases. 
Let $\Phi\colon V\to W$ be a continuous map over $X$.  

Let $c\colon U_V\to A_V\times F_V$ be a product chart of $V$. We say that $\Phi\vert_U$ 
is a \emph{product map} if there is a product chart $d\colon U_W\to A_W\times F_W$ of $W$ such that $d\circ \Phi\circ c^{-1}\colon A_V\times F_V\to A_W\times F_W$ 
is a product of the identity on $X$ and a continuous map. 

We say that $\Phi$ is \emph{locally product-like} if every point of $V$ 
lies in the domain of a product chart on which $\Phi$ is a product map. 
\end{definition}

\begin{definition}
   A continuous map over $X$ between Cantor bundles is called a \emph{Cantor bundle map} if it is $\Gamma$-equivariant and locally product-like. 
   A Cantor bundle map between metric Cantor bundles is called \emph{Lipschitz} if there is some $L>0$ such that it is $L$-Lipschitz on each fiber. 
   \end{definition}
   
   \begin{remark}\label{rem: automatic properness}
   A Cantor bundle map $V\to W$ is automatically proper since the $\Gamma$-actions on $V$ and $W$ are proper and both actions possess relatively compact fundamental domains. 
   \end{remark}
   
   The composition of (Lipschitz) Cantor bundle maps is a (Lipschitz) Cantor bundle map. So we obtain a category of Cantor bundles with Cantor bundle maps as morphisms. 

The notion of product map does not depend on the choices of product charts as we show next. 

\begin{lemma}\label{lem: box decomposition for maps}
Let $V$ and $W$ be locally compact Hausdorff spaces over $X$ equipped with maximal product atlases. 
   Let $\Phi\colon V\to W$ be locally product-like and proper. 
Then every compact subset $K\subset W$ is contained in a relatively compact, open subset $L\subset W$ such that 
there is a clopen partition $X=A_1\cup\dots\cup A_n$ with the following properties. 
\begin{itemize}
\item $L\vert_{A_i}$ is a box for each $i\in\{1,\dots,n\}$.
\item $\Phi^{-1}(L)\vert_{A_i}$ is a box for each $i\in\{1,\dots,n\}$. 
\item The restriction of $\Phi$ to \[\Phi^{-1}(L)\vert_{A_i}\to L\vert_{A_i}\] is 
a product map for each $i\in\{1,\dots,n\}$.
\end{itemize}
If $K$ is an open box, we may choose $L$ to be $K$. 
\end{lemma}

\begin{proof}
Every compact 
subset of $W$ is contained in a relatively 
compact, open subset $L\subset W$ with the properties as in 
Lemma~\ref{lem: decomposition into boxes}. We may assume that $L$ itself is a box. Since $\Phi$ is proper, $\Phi^{-1}(L)$ is relatively compact as well. We cover $\Phi^{-1}(L)$ by finitely many open boxes $B_i$, $i\in I$, such that $\Phi$ is a product of maps on each box. The intersection of two boxes is a box. The preimage of a box under $\Phi\vert_{B_i}$ is a box. 
Hence $\Phi^{-1}(L)$ is the union of boxes $B_i'\defq \Phi^{-1}(L)\cap B_i$, $i\in I$. On each $B_i'$ the map $\Phi$ is a product. Let $X=A_1\cup \dots\cup A_n$ be a clopen partition subordinate to $\pr_V(B_i')$, $i\in I$. It exists since the clopen sets of $X$ form a set algebra. By the same argument as in the proof 
of Lemma~\ref{lem: decomposition into boxes} each $\Phi^{-1}(L)\vert_{A_j}$, $j\in\{1,\dots, n\}$, is a box. As in the proof of Lemma~\ref{lem: decomposition into boxes} one sees that the parallel transport on the boxes $B_i'$ and the choice of some $x_0\in A_j$ yields a product chart 
\[ \Phi^{-1}(L)\vert_{A_j}\xrightarrow{\cong} A_j\times \bigl(\Phi^{-1}(L)\cap V_{x_0}\bigr).\]
Similarly for $L\vert_{A_j}$. Since $\Phi$ is a product map on each $B_i'$, it is compatible with the parallel transport within each $B_i'$, and so the restriction of $\Phi$ to $\Phi^{-1}(L)\vert_{A_j}\to L\vert_{A_j}$ is a product map. 
\end{proof}

\begin{lemma}\label{lem: box decomposition for maps - domain focus}
   Let $V$ and $W$ be locally compact Hausdorff spaces over $X$ equipped with maximal product atlases. 
      Let $\Phi\colon V\to W$ be locally product-like and proper. 
   Then every compact subset $K\subset V$ is contained in a relatively compact, open subset $L\subset V$ such that 
   there is a clopen partition $X=A_1\cup\dots\cup A_n$ with the following properties. 
   \begin{itemize}
      \item $L\vert_{A_i}$ is a box for each $i\in\{1,\dots,n\}$.
      \item The restriction of $\Phi$ to $L\vert_{A_i}$ is a product map for each $i\in\{1,\dots,n\}$. 
      \end{itemize}
      If $K$ is an open box, we may choose $L$ to be $K$. 
   \end{lemma}

\begin{proof}
   This follows from applying Lemma~\ref{lem: box decomposition for maps} to 
   the relatively compact subset $\Phi(K)$. The last statement follows from the fact that the intersection of two boxes is a box. 
\end{proof}

Next we discuss categorical pushouts in the category of Cantor bundles.

\begin{lemma}\label{lem: pushouts of Cantor bundles}
A commutative square of Cantor bundles and Cantor bundle maps that is a pushout of topological spaces is a pushout in the category of Cantor bundles. 
\end{lemma}

\begin{proof}
Let the following diagram
\[\begin{tikzcd}
  A\ar[r, "f"]\ar[d,"i"] & C\ar[d, "j"]\\
  B\ar[r, "g"] & D
\end{tikzcd}\]
be commutative and consist of Cantor bundles and Cantor bundle maps. We further assume that it is a pushout of topological spaces. Let $Z$ be a Cantor bundle, 
and let $r\colon B\to Z$ and $s\colon C\to Z$ be Cantor bundle maps compatible 
with~$A$. By the pushout property there is unique map $t\colon D\to Z$ that makes 
everything commute. It is obvious that $t$ is equivariant and lies over $X$. It remains to show that $t$ is locally product-like. Let $p\in D$ be a point over $x\in X$. Let $U\subset D$ be an open box around $p$. By Lemma~\ref{lem: box decomposition for maps - domain focus} there are clopen subsets $X_C$ and $X_B$ of $X$ that contain $x$ such that $j^{-1}(U)\vert_{X_C}$ and $g^{-1}(U)\vert_{X_B}$ are boxes 
and the restrictions of $j$ and $g$ to $j^{-1}(U)\vert_{X_C}$ and $g^{-1}(U)\vert_{X_B}$ are product maps. By the same lemma we obtain a clopen neighborhood $X_A\subset X_B\cap X_C$ of $x$ such that $i^{-1}(g^{-1}(U))\vert_{X_A}=f^{-1}(j^{-1}(U))\vert_{X_A}$ is a box 
and the restrictions of $i$ and $f$ to $i^{-1}(g^{-1}(U))\vert_{X_A}$ are product maps. By applying 
Lemma~\ref{lem: box decomposition for maps - domain focus} to the boxes $g^{-1}(U)\vert_{X_A}$ and 
$j^{-1}(U)\vert_{X_A}$ we find a smaller clopen neighborhood $X_Z\subset X_A$ of $x$ such that $r$ and $s$ are product maps on these boxes. Now the box $U\vert_{X_Z}$ is a pushout of the boxes $g^{-1}(U)\vert_{X_Z}$ and 
$j^{-1}(U)\vert_{X_Z}$ along the box $i^{-1}(g^{-1}(U))\vert_{X_Z}$. 
All maps in the pushout square are product maps as well as the restrictions of $r$ and $s$ to the corners. Hence $t$ restricted to $U\vert_{X_Z}$ is a product map. 
\end{proof}



\begin{example}[Continuing Example~\ref{exa: non trivial Cantor bundle}]
The Cantor bundle~$W$ in Example~\ref{exa: non trivial Cantor bundle} can be written as a pushout. 
Let 
\[A=\bigl\{x\in X\subset \{0,1\}^\bbZ\mid x_0=1\bigr\}\] 
and $A^c\subset X$ its complement. 
Let $X\times \bbR$ be the trivial Cantor bundle with the diagonal $\bbZ$-action. On $\bbR$ we consider the usual translation action of $\bbZ$. The pushout for $W$ can be written semi-formally as
\[
\begin{tikzcd}
A\times \bbZ\times\begin{tikzpicture}
\node [draw,circle, color=gray,fill, scale=0.5] [above=.1mm]{}; 
\node [draw,opacity=0,circle, fill, scale=0.5] [above right=4mm of s0] {} edge[opacity=0] (s0); 
\node [draw,opacity=0,circle, fill, scale=0.5] [above left	=4mm of s0] {} edge[opacity=0] (s0); 
\end{tikzpicture} ~\coprod~ A^c\times\bbZ\times
\begin{tikzpicture}
\node [draw,circle, opacity=0,fill, scale=0.5] [below=4mm of s0] {} edge[opacity=0] (s0); 
\node [draw,circle, color=gray, fill, scale=0.5]  [above=.1mm]{}; 
\end{tikzpicture}\ar[d,hook,shorten <= -1.5em]\ar[r] &X\times\bbR\ar[d]\\
A\times \bbZ\times \begin{tikzpicture}
\node [draw,circle, color=gray,fill, scale=0.5] {}; 
\node [draw,circle, fill, scale=0.5] [above right=4mm of s0] {} edge (s0); 
\node [draw,circle, fill, scale=0.5] [above left	=4mm of s0] {} edge (s0); 
\end{tikzpicture} ~\coprod~ A^c\times\bbZ\times
\begin{tikzpicture}
\node [draw,circle, color=gray,fill, scale=0.5] {}; 
\node [draw,circle, fill, scale=0.5] [below=4mm of s0] {} edge (s0); 
\end{tikzpicture} \ar[r] & W.
\end{tikzcd}
\]
Here we regard a space of the type $A\times\bbZ\times M$ as a Cantor bundle as in Example~\ref{exa: cell}.
\end{example}

\begin{lemma}\label{lem: map to classifying space}
Let $N$ be a cocompact proper $\Gamma$-CW complex. Then 
$X\times N$ with the diagonal $\Gamma$-action is a Cantor bundle. Further, there is a locally product-like map $X\times N\to X\times E\Gamma$ over $X$ to the classifying space of~$\Gamma$ which is $\Gamma$-equivariant with respect to the diagonal actions.  
\end{lemma}

\begin{proof} 
The $\Gamma$-CW complex $N$ is built via equivariant pushouts where we successively attach finitely many equivariant cells of the form $\Gamma/H\times D^n$ with finite $H<\Gamma$: 
\[
\begin{tikzcd}
\coprod_{i\in I_n} \Gamma/H_i \times S^{n-1}\ar[d,hook]\ar[r] &N^{(n-1)}\ar[d, hook]\\
\coprod_{i\in I_n} \Gamma/H_i\times D^n\ar[r] & N^{(n)}
\end{tikzcd}
\]
Taking a product with the compact Hausdorff space $X$ preserves pushouts. 
With Remark~\ref{rem: finite isotropy} we see that 
$X\times N$ is inductively built via finitely many pushouts of the form 
\[
\begin{tikzcd}
\coprod_{i\in I_n} A_i\times \Gamma \times S^{n-1}\ar[d,hook]\ar[r] &X\times N^{(n-1)}\ar[d, hook]\\
\coprod_{i\in I_n} A_i\times \Gamma\times D^n\ar[r] & X\times N^{(n)}.
\end{tikzcd}
\]
There is a Borel fundamental domain of $X\times N$ consisting of the finite union of 
the products of $A_i$ and the open $n$-cell associated with $i\in I_n$ over all $n$ and $i\in I_n$. Hence $X\times N$ is a Cantor bundle. 

Next we apply repeatedly Lemma~\ref{lem: pushouts of Cantor bundles} to the 
above pushout for $n=1,2,\dots, \dim(N)$ and to the target $X\times E \Gamma$ to construct equivariant, locally product-like maps $X\times N^{(n)}\to X\times E\Gamma$. 
Strictly speaking, the target $X\times E\Gamma$ is not necessarily a Cantor bundle as required in the lemma since 
$E\Gamma$ might not be a finite $\Gamma$-CW complex. However, the image of the maps 
below lie in the product of $X$ and a finite $\Gamma$-CW subcomplex which allows us to use Lemma~\ref{lem: pushouts of Cantor bundles}. 

On the $0$-skeleton $\coprod_{i\in I_0} A_i\times\Gamma$ we define an equivariant and locally product-like map to $X\times E\Gamma$ as the equivariant extension that maps $(a,1)$ to $(a,p)$ 
for some chosen point $p\in E\Gamma$. 
 To proceed inductively via the pushout property, we need to extend a continuous locally product-like equivariant map $A_i\times \Gamma\times S^{n-1}\to X\times E\Gamma$ to 
$A_i\times \Gamma\times D^n$. By decomposing each $A_i$ into a suitable clopen partition we may assume 
that the restriction $A_i\times\{1\}\times S^{n-1} \to X\times E\Gamma$ is a product of the inclusion $A\hookrightarrow X$ and a continuous map $S^{n-1}\to E\Gamma$. Since $E\Gamma$ is contractible, we can extend the map 
$A\times\{1\}\times S^{n-1} \to E\Gamma$ to $A\times\{1\}\times D^n$ and then extend further to $A\times\Gamma\times D^n$ by equivariance. The resulting map is locally product-like. 
\end{proof}

\subsection{Chains and norms of chains in the context of Cantor bundles}\label{sub: from chains to measure chains}
We consider various chain complexes involving singular chains, locally finite chains and measure chains for  Cantor bundles or spaces over $X$. 

The abelian group $C(X;\bbZ)$ carries a left $\bbZ[\Gamma]$-module structure and via the involution on the group ring also a right $\bbZ[\Gamma]$-module structure. 
Let $N$ be a $\Gamma$-space. 
For every $x\in X$ the map 
\begin{gather*} \ev_x\colon C(X;\bbZ)\otimes_{\bbZ[\Gamma]} C_\ast(N)\to C_\ast^\mathrm{lf}(N)\\
\sum_i f_i\otimes\sigma_i\mapsto \sum_{\gamma\in\Gamma} \sum_if_i(\gamma^{-1}x)\gamma \sigma_i
\end{gather*}
to locally finite chains is a chain map~\cite{foliated}*{Lemma~2.5}. 
We consider the map 
\[ \map\bigl(\Delta^n, N\bigr)\to \calC_n\bigl(X\times N\bigr),~~\sigma\mapsto\sigma^X\]
that sends a singular $n$-simplex $\sigma$ in $N$ to the 
finite, compactly supported Borel measure $\sigma^X$ on $\map(\Delta^n, X\times N)$ characterized by the property 
\[ \int_{\map(\Delta^n, X\times N)}gd\sigma^X=\int_X g\bigl(\Delta^n\xrightarrow{\sigma}\{x\}\times N\hookrightarrow X\times N\bigr)d\mu(x)\]
for every compactly supported continuous function $g$ on $\map(\Delta^n, X\times N)$. If $f\in C(X)$, then $f\ast\sigma^X$ denotes the measure with 
\[ \int_{\map(\Delta^n, X\times N)}gd(f\ast\sigma^X)=\int_X f(x)g\bigl(\Delta^n\xrightarrow{\sigma}\{x\}\times N\hookrightarrow X\times N\bigr)d\mu(x).\]
The map 
\begin{gather*}
C(X;\bbZ)\otimes C_\ast(N)\hookrightarrow \calC_\ast(X\times N)\\
\sum_i f_i\otimes \sigma_i\mapsto \sum_i f_i\ast \sigma_i^X
\end{gather*}
is a $\Gamma$-equivariant injective chain map. 
Here $\Gamma$ acts diagonally on the left hand side, and the left action on the right hand side is induced by the diagonal action on $X\times N$. 
\begin{definition}\label{def: diffusion}
The chain map $C(X;\bbZ)\otimes C_\ast(N)\hookrightarrow \calC_\ast(X\times N)$ is called the \emph{diffusion embedding}. 
\end{definition}

\begin{lemma}\label{lem: induced map in smaller subcomplex}
   Let $U$ and $V$ be topological Hausdorff spaces. We endow $X\times U$ and $X\times V$ with the obvious maximal product atlasses. Let $\Phi\colon X\times U\to X\times V$ be a locally product-like map over~$X$. 
   Then there is a chain map, indicated by the dashed arrow, such that the following diagram commutes. Furthermore, this chain is non-increasing with respect to the $X$-parametrised integral simplicial norm. 
   \begin{equation*}
      \begin{tikzcd}
         \calC_\ast(X\times U)\arrow[r, "\Phi_\ast"] & \calC_\ast(X\times V)\\
         C(X; \bbZ)\otimes C_\ast(U)\arrow[r,dashed]\arrow[u,hook] & C(X;\bbZ)\otimes C_\ast(V)\arrow[u, hook]
      \end{tikzcd}
     \end{equation*}
   Here the upper horizontal map is the induced map in measure chains. The vertical maps are the diffusion embeddings. We will also denote 
   the dashed arrow by~$\Phi_\ast$. 
   \end{lemma}
   
   \begin{proof}
   Let $\sigma\colon \Delta_p\to U$ and $f\in C(X;\bbZ)$. We apply Lemma~\ref{lem: box decomposition for maps} to the map 
   \[ X\times\Delta^n\xrightarrow{\id_X\times \sigma}X\times U\xrightarrow{\Phi} X\times V.\] 
   The image of this map is a compact subspace of the Hausdorff space $X\times V$, hence we can apply 
   Lemma~\ref{lem: box decomposition for maps} even if $V$ is not locally compact. As a result 
   there is a finite Borel partition $X=A_1\cup \dots\cup A_n$ and there are continuous maps 
   $g_i\colon U\to V$ for $i\in\{1,\dots, n\}$ such that 
   \[ \Bigl(\Delta_p\xrightarrow{\sigma}\{x\}\times U\xrightarrow{\Phi_x} X\times V\xrightarrow{\pr} V\Bigr)=g_i\circ \sigma~\text{ for $x\in A_i$.}\]
   Hence~$\Phi_\ast$ maps $f\otimes\sigma$ to the measure that is the image 
   of $\sum_{i=1}^n f\cdot\chi_{A_i}\otimes g_i\circ\sigma$ under the diffusion embedding. The statement about the $X$-parametrised integral simplicial norm follows directly from this description. 
   \end{proof}
   
\begin{remark}
   The advantage of measure homology is that $\Phi$ obviously induces a chain map. It automatically follows that its restriction to the complex $C(X;\bbZ)\otimes C_\ast(\_)$ is a $\Gamma$-equivariant chain map provided $\Phi$ is $\Gamma$-equivariant, for example a Cantor bundle map. A direct verification of functoriality that avoids measure homology would be more cumbersome. 
\end{remark}

\section{Transverse measure theory on Cantor bundles}\label{sec: transverse measure}

In this short section we define the notion of transverse measure which gives, in particular, a finite measure on $\Gamma$-invariant Borel subsets of $\xxm$. As before, $\mu$ denotes the $\Gamma$-invariant 
probability measure on the Cantor set $X$. 

\begin{definition}
	Let $W$ be a standard $\Gamma$-space endowed with a $\Gamma$-invariant Borel measure $\lambda$. 
	For any choice of a measurable $\Gamma$-fundamental domain $F\subset W$ we define a Borel measure $\lambda^{\text{tr}}$ on the $\sigma$-algebra of $\Gamma$-invariant Borel subsets of $W$ by 
	\[ \lambda^{\text{tr}}(A)=\lambda(A\cap F).\]
We call $\lambda^{\text{tr}}$ the \emph{transverse measure} induced by $\lambda$. \end{definition}

\begin{remark}The definition of the transverse measure does not depend on the choice of the Borel fundamental domain. 
	This is proved similarly to the situation of a lattice in a locally compact group. Let $F'\subset W$ another Borel $\Gamma$-fundamental domain. Let $A\subset W$ be a $\Gamma$-invariant Borel subset. Then 
\begin{align*}
\lambda(A\cap F)=\lambda \Bigl(A\cap \bigcup_{\gamma\in\Gamma}\gamma F'\cap F\Bigr)	
                &=\lambda \Bigl(\bigcup_{\gamma\in\Gamma} A\cap \gamma F'\cap F\Bigr)\\
                &=\lambda \Bigl(\bigcup_{\gamma\in\Gamma} \gamma^{-1}A\cap F'\cap \gamma^{-1} F\Bigr)\\
                &=\lambda \Bigl(A\cap F'\cap\bigcup_{\gamma\in\Gamma}\gamma^{-1}F'\Bigr)\\
                &=\lambda (A\cap F').
\end{align*}
\end{remark}

\begin{definition}
Retaining the setting of the previous definition, the pushforward of $\lambda^{\text{tr}}$ under the 
restriction of the quotient map $F\to \Gamma\bs W$ is a Borel measure on $\Gamma\bs W$ that we also call the \emph{transverse measure} induced by $\lambda$ and denote by the same symbol. 
\end{definition}

\begin{definition}
Let $\pr\colon W\to X$ be a metric Cantor bundle. Let $\calH^d_x$ be the $d$-dimensional Hausdorff measure on $W_x$. Then 
\[ \int_X\calH^d_xd\mu(x)\]
defines a $\Gamma$-invariant Borel measure on $W$ that we call the \emph{$d$-dimensional Hausdorff measure} of the Cantor bundle $\pr$. We denote it generically by $\vol_d$. The induced transverse measure is denoted by $\voltrans_d$. 
\end{definition}

The fact that $A\mapsto \int_X \calH^d_x(A\cap W_x)d\mu(x)$ is Borel measurable for a Borel subset $A\subset W$ follows from the existence of a product atlas. The $\Gamma$-invariance follows from the fact that $\mu$ is $\Gamma$-invariant and the multiplication with an element $\gamma\in\Gamma$ is fiberwise an isometry. 
So the above definition is justified. 

\begin{lemma}
Let $\Phi\colon V\to W$ be a Lipschitz Cantor bundle map, and let $\phi$ be the induced map on $\Gamma$-quotients. Then the map 
\[ \Gamma\bs W\to \bbN\cup\{\infty\},~w\mapsto \# \phi^{-1}(\{w\})\]
is $\voltrans_d$-measurable, and we define the \emph{transverse $d$-volume} of $\Phi$ as 
\[ \voltrans_d(\Phi)=\int_{\Gamma\bs W} \# \phi^{-1}(\{w\})d\voltrans_d(w).\]
\end{lemma}

\begin{proof}
Let $F_W\subset W$ be a Borel fundamental domain of $\Gamma\acts W$ that is a finite union of boxes. Then $F_V=\Phi^{-1}(F_W)$ is a fundamental domain 
of $\Gamma\acts V$. 
By Lemma~\ref{lem: box decomposition for maps} there is a clopen partition 
$X=A_1\cup\dots\cup A_n$ such that each $V\vert_{A_i}\cap F_V$ is a box, 
each $W\vert_{A_i}\cap F_W$ is a box and the restriction 
\[ \Phi\colon V\vert_{A_i}\cap F_V\to W\vert_{A_i}\cap F_W\]
is a product $\id_{A_i}\times h_i$ (in product charts) with $h_i$ being Lipschitz.   

The statement is equivalent to the measurability of the map 
\[ f\colon F_W\to \bbN\cup\{\infty\},~~w\mapsto \#\bigl(\Phi^{-1}(\{w\})\cap F_V\bigr).\]
In product chart coordinates $(x,p)\in A_i\times F\cong (F_W)\vert_{A_i}$ we have 
$f(x,p)=\#h_i^{-1}(\{p\})$. 
So the measurability over $A_i$ and hence everywhere follows from~\cite{federer}*{2.10.9}.
\end{proof}

Let $f\colon M\to N$ be a Lipschitz map between Riemannian manifolds. By Rade\-macher's theorem $f$ is 
differentiable almost everywhere. Let $J_df$ be the almost everywhere defined \emph{$d$-dimensional Jacobian} of $Df$. If $f$ is smooth and $M$ and $N$ are $d$-dimensional, then $J_df(m)$ is the absolute value of the determinant of the differential at $m\in M$ with respect to orthonormal bases. 
Let $\Phi\colon V\to W$ be a Lipschitz Cantor bundle map between metric Cantor bundles for which every 
fiber is a Riemannian manifold. We consider the quotient map $\phi\colon \Gamma\bs V\to\Gamma\bs W$. By equivariance of $\Phi$ the 
$d$-dimensional Jacobian of $\phi$ is well defined on a conull set by 
\[ J_d\phi([p])=J_d\Phi(p).\]

We prove the following version of the area formula for Cantor bundles. 

\begin{theorem}[Area formula]\label{thm: Area formula}
Let $\Phi\colon V\to W$ be a Lipschitz Cantor bundle maps between Riemannian Cantor bundles, where each fiber of $V$ is $d$-dimensional. Then the $d$-dimensional Jacobian $J_d\phi$ is 
$\voltrans_d$-measurable and 
\[ \voltrans_d(\Phi)=\int_{\Gamma\bs W} \# \phi^{-1}(\{w\})d\voltrans_d(w)=\int_{\Gamma\bs V} J_d\phi(v) d\voltrans_d(v).\]
\end{theorem}

\begin{proof}
Let $F_W\subset W$ be a Borel fundamental domain consisting of finitely many boxes. 
Let $F_V$ be its $\Phi$-preimage. The statement is equivalent to 
\begin{align}\label{eq: area formula} \int_{F_W}\#\Phi^{-1}(\{w\})d\vol_d(w)=\int_{F_V} J_d\Phi(v) d\vol_d(v).
\end{align}
Let $X=A_1\cup\dots\cup A_n$ be a clopen partition as in Lemma~\ref{lem: box decomposition for maps} such that $(F_V)\vert_{A_i}$ and $(F_W)\vert_{A_i}$ are boxes. Choose homeomorphisms $A_i\times F_i\cong (F_V)\vert_{A_i}$ 
and $A_i\times F_i'\cong (F_W)\vert_{A_i}$ such that under these homeomorphisms the map $\Phi$ is of the form $\id_{A_i}\times h_i$.  
The left hand side of~\eqref{eq: area formula} becomes 
\[ \sum_{i=1}^n\mu(A_i)\int_{F_i'}\#h_i^{-1}(\{w\})d\calH^d(w).\]
The right hand side of~\eqref{eq: area formula} becomes 
\[ \sum_{i=1}^n\mu(A_i)\int_{F_i} J_dh_i(v) d\calH^d(v).\]
The $i$-th summand of left and right hand side coincide by the classical area formula~\cite{federer}*{3.2.3 on p.~243}. 
\end{proof}

\section{Rectangular Cantor nerves and Cantor covers}\label{sec: rectangular cantor nerves}

In Section~\ref{subsec: cantor covers} we introduce the notion of Cantor cover on $\xxm$ which is an analog of a $\Gamma$-equivariant cover by balls on $\widetilde M$. Then we define a good Cantor cover. The proof of its existence is postponed to Section~\ref{sub: vitali layering}. In Section~\ref{sub: Cantor nerve} we introduce the analog of Guth's rectangular nerve in our setting. In Section~\ref{sub: map to Cantor nerve} we introduce the nerve map as a Cantor bundle map.

\subsection{Cantor covers}\label{subsec: cantor covers}
A \emph{Cantor cover} of $\xxm$ is an open cover
    \[\calU=\{A_j\times B_j\mid j\in J\}\] of $\xxm$ by product sets 
    of clopen sets $A_j\subset X$ and open balls $B_j\subset \widetilde M$ indexed over a free cofinite $\Gamma$-set $J$ such that 
   $A_{\gamma j}=\gamma A_j$ and $B_{\gamma j}=\gamma B_j$ for all $\gamma\in\Gamma$ and $j\in J$. We further require that $\{A_j\times\frac{1}{2}B_j\mid j\in J\}$ still covers $\xxm$. If we replace the property of being a cover by requiring that the elements of $\calU$ are pairwise disjoint then we call $\calU$ a \emph{Cantor packing}. 
   
   Since the index set is cofinite, i.e.~consists of finitely many orbits, and the $\Gamma$-action on $\widetilde M$ is proper, a Cantor cover is always locally finite. Let $\calU=\{A_j\times B_j\mid j\in J\}$  
be a Cantor cover or Cantor packing of $\xxm$. Let $\calV$ be an arbitrary family of subsets of a space. 
\begin{itemize}
\item We denote the union of the elements of $\calV$ by $\bigcup\calV$. 
\item For $x\in X$ we denote by \[\calU_x=\{B_j\mid j\in J, x\in A_j\}\] the induced open cover (packing, respectively) of $\widetilde M\cong \{x\}\times \widetilde M$. 
\item We say that $\calU$ has \emph{no self intersections} if $\bigl(\gamma A_j\times \gamma B_j\bigr)\cap \bigl(A_j\times B_j\bigr)\ne\emptyset$ implies $\gamma=1$ for every $j\in J$ and $\gamma\in \Gamma$. 
\item For $a>0$ we write $a\,\calU\defq \{A_j\times aB_j\mid j\in J\}$.
\end{itemize}

We produce a suitable Cantor cover of $\xxm$ that consists of good balls in every fiber. The notion of goodness goes back to Gromov. We refer 
to~\cite{guthvolume}*{Section~1} for this notion. A cover by good balls will be called a good cover which is a bit unfortunate since this terminology is also used for covers with contractible sets and intersections. 

	Let $N$ be a $d$-dimensional Riemannian manifold and $V_N(1)$ be the supremal volume of $1$-balls in $N$. The ball $B(p,r)\subseteq N$ of radius $r$ around a point $p\in M$ is called a \emph{good ball} if the following conditions are satisfied.
	\begin{enumerate}
		\item Reasonable growth: $\vol(B(p,100r))\leq 10^{4(d+3)}\vol(B(p,\frac{1}{100}r))$.
		\item Volume bound: $\vol(B(p,r))\leq 10^{2(d+3)}V_N(1)r^{d+3}$.
		\item Small radius: $r\leq \frac{1}{100}$.
	\end{enumerate}
A \emph{good cover} of a Riemannian manifold is an open cover by good balls where the concentric $\frac{1}{6}$-balls are disjoint and the $\frac{1}{2}$-balls provide a cover of the manifold as well. A Cantor cover $\calU$ of $X\times\widetilde M$ is called \emph{good} if $\calU_x$ is a good cover of $\widetilde M$ for every $x\in X$. 

Guth showed that any closed Riemannian manifold has a good cover \mbox{\cite{guthvolume}*{Lemma 2}}. At the end of Section~\ref{sub: vitali layering} we will be able to give the proof of the equivariant statement: 

\begin{theorem}\label{thm:EquivariantCover}
There exists a good Cantor cover on $\xxm$ that has no self intersections. 
\end{theorem}

\subsection{The rectangular Cantor nerve of a Cantor cover}\label{sub: Cantor nerve}

In the sequel we consider a Cantor cover  
\[\calU=\{A_j\times B_j\mid j\in J\}\] 
of $\xxm$. We adhere to the following notation. 
\begin{itemize}
\item By picking a set of 
$\Gamma$-representatives we write the $\Gamma$-set $J$ as $\Gamma\times I$ with finite~$I$. 
\item Let $r_j$ denote the radius 
of the ball $B_j$ and $m_j$ the center of $B_j$. 
\item Let $\calV\defq\{B_j\mid j\in J\}$. This is  a locally finite cover of~$\widetilde M$ since 
$\Gamma$ acts freely and properly on $\widetilde M$.  
\end{itemize}

The nerve $N(\calV)$ satisfies the requirements of Lemma~\ref{lem: proper Gamma CW}. In particular, its barycentric subdivision is a proper $\Gamma$-CW-complex. 
By properness and cofiniteness of $J$ the maximal multiplicity of $\calV$ is finite, hence $N(\calV)$ is cocompact. Since $\calU_x$ is a subcover of $\calV$, $N(\calU_x)$ is a subcomplex of $N(\calV)$ for every $x\in X$.

\begin{definition}\label{def: Cantor nerve}
The \emph{rectangular Cantor nerve} $\cantornerve(\calU)$ 
of the Cantor cover $\calU$ is the subset 
\[ \bigl\{(x,p)\mid p\in N(\calU_x), x\in X\bigr\}\subset X\times N(\calV).\]
\end{definition}

Clearly, $\cantornerve(\calU)$ is a $\Gamma$-invariant subset of $N(\calV)$. We restrict the metric of $N((r_j)_{j\in J})$ to $N(\calV)$ and, further, to each $N(\calU_x)$.  

\begin{lemma}\label{lem: nerve as cantor subbundle}
The rectangular Cantor nerve $\cantornerve(\calU)$ is a metric Cantor subbundle of 
the trivial metric Cantor bundle $X\times N(\calV)$ endowed with its diagonal $\Gamma$-action. 
\end{lemma}

\begin{proof}
It suffices to construct an open box neighborhood around each 
point 
\[(x,p)\in \cantornerve(\calU)\subset X\times N(\calV)\subset X\times N\bigl((r_j)_{j\in J}\bigr).\] 
Let $F$ be an open face in $N(\calU_x)$ that contains $p$.  The star of $F$ within $N\bigl((r_j)_{j\in J}\bigr)$ consists of all open faces in $N\bigl((r_j)_{j\in J}\bigr)$ that contain $F$ in their closure. A face $E$ of $N\bigl((r_j)_{j\in J}\bigr)$ is in the star of $F$ if and only if $J_+(F)\subset J_+(E)$. 
For a face $E$ which lies in the star of $F$ and in $N(\calV)$ the subset $J_+(E)$ lies in the subset 
\[J_F\defq\bigl\{ j\in J\mid B_j\cap\bigcap_{i\in J_+(F)}B_i\ne\emptyset\bigr\},\]
which is finite since $\calV$ is locally finite. 
As a finite intersection of clopen sets the set 
\[ C\defq \bigcap_{\substack{j\in J_F\\x\in A_j}} A_j\cap  \bigcap_{\substack{j\in J_F\\x\not\in A_j}} X\bs A_j\]
is a clopen neighborhood of~$x$. Let $S$ be the star of $F$ within $N(\calV)$. 
Let $S'$ be the star of $F$ within $N(\calU_x)$. 

We claim that every face $E$ of $S'$ lies in $N(\calU_y)$ for all $y\in C$. 
For such $E$ we have 
$J_1(E)\ne \emptyset$ and $\bigcap_{j\in J_+(E)}B_j\ne\emptyset$. 
Since $E$ lies in $N(\calU_x)$ we have $x\in\bigcap_{j\in J_+(E)}A_j$. 
The inclusion $J_+(E)\subset J_F$ implies that $C\subset\bigcap_{j\in J_+(E)}A_j$. 
Thus $E$ lies in $N(\calU_y)$ for every $y\in C$. Therefore 
\[ \cantornerve(\calU)\cap \bigl(C\times S'\bigr)=C\times S',\]
where the left hand intersection is taken within $X\times N(\calV)$.

Next we show that $C\times S'$ is open in $\cantornerve(\calU)$. 
To this end, we show that 
$S'=N(\calU_y)\cap S$ for all $y\in C$. Since $S$ is open in $N(\calV)$ 
this proves that $C\times S'$ is open. Let $E$ be a face in 
$N(\calU_y)\cap S$. In particular, $y\in \bigcap_{j\in J_+(E)}A_j$. 
Then $E$ lies in $S'$ if $x\in\bigcap_{j\in J_+(E)}A_j$. Suppose there is $j_0\in J_+(E)$ with $x\not\in A_{j_0}$. Since $y\in C$, this 
would imply $y\in X\bs A_{j_0}$ and contradict $y\in \bigcap_{j\in J_+(E)}A_j$. Hence $E$ lies in $S'$. 

So $C\times S'$ is a box neighborhood containing $(x,p)\in \cantornerve(\calU)$ with respect to the global product chart on $X\times N(\calV)$. 
\end{proof}

Next we try to understand the rectangular Cantor nerve by pushouts.

\begin{lemma}\label{lem: cantor nerve as pushout}
   We assume that $\calU$ has no self-intersections. 
   Let $C_n$ be 
a complete set of representatives of the $\Gamma$-orbits of the $n$-dimensional faces of $N(\calV)$. 
The $n$-skeleton $\cantornerve(\calU)^{(n)}=\cantornerve(\calU)\cap X\times N(\calV)^{(n)}$ 
of $\cantornerve(\calU)$ arises from the $(n-1)$-skeleton as a pushout
\begin{equation}\label{eq: cantornerve as pushout}
\begin{tikzcd}
\coprod\limits_{F\in C_n} \Bigl(\bigcap\limits_{j\in J_+(F)} A_j\Bigr)\times \Gamma\times\partial\Bigl( \prod\limits_{k=1}^{n} [0, r_k(F)]\Bigr)\arrow[r,hook]\arrow[d,hook] & \cantornerve(\calU)^{(n-1)}\arrow[d,hook]\\
\coprod\limits_{F\in C_n} \Bigl(\bigcap\limits_{j\in J_+(F)} A_j\Bigr)\times \Gamma\times\Bigl(\prod\limits_{k=1}^n [0, r_k(F)]\Bigr)\arrow[r,hook] & \cantornerve(\calU)^{(n)}.
\end{tikzcd}
\end{equation}
whose maps are Cantor bundle maps. The lower horizontal map is fiberwise 
an isometric embedding of $ \prod_{k=1}^n [0, r_k(F)]$. 
\end{lemma}
\begin{proof}
The map 
\begin{gather*}\coprod\limits_{F\in C_n} \Bigl(\bigcap\limits_{j\in J_+(F)} A_j\Bigr)\times \Gamma\times\Bigl(\prod\limits_{j\in J_\half(F)}[0,r_j]\Bigr)\xrightarrow{\Psi}  X\times N\bigl((r_j)_{j\in J}\bigr)\\
   \bigl(a,\gamma, (w_j)_{j\in J_\half(F)}\bigr)\mapsto \bigl(\gamma a, \gamma\cdot (\bar w_{j})_{j\in J}\bigr)
\end{gather*}
where 
\[
\bar w_j=\begin{cases} w_j & \text{if $j\in J_\half(F)$,}\\
                       0   & \text{if $j\in J_0(F)$,}\\
                       r_j & \text{if $j\in J_1(F)$,}   
\end{cases}
\]
lands in $\cantornerve(\calU)^{(n)}$ and is a Cantor bundle map into $\cantornerve(\calU)^{(n)}$. Next we verify that the restriction $\Psi_0$ of $\Psi$ 
\[ \coprod\limits_{F\in C_n} \Bigl(\bigcap\limits_{j\in J_+(F)} A_j\Bigr)\times \Gamma\times\prod\limits_{j\in J_\half(F)}(0,r_j)\xrightarrow{\Psi_0} \cantornerve(\calU)^{(n)}\bs \cantornerve(\calU)^{(n-1)}\]
is bijective. Suppose $\Phi_0$ maps two points $(a,\gamma, (w_j))$ and $(a',\gamma', (w_j'))$ in the left hand summands associated with the $n$-faces $F$ and $F'$ to the same point. By equivariance it suffices to consider the case $\gamma'=1$. The open faces $\gamma F$ and $F'$ intersect, hence concide as subsets $\gamma F=F'$\footnote{Since $N(\calV)$ is not a $\Gamma$-CW complex (only after barycentric subdivision) one might have $F=\gamma F$ as subsets but not pointwise.}. Since $F,F'$ are from a complete set of $\Gamma$-representatives $C_n$ we obtain that $F'=F$ and $\gamma F=F$ as subsets. Let $x\defq \gamma a=a'$. 
The $n$-faces of $N(\calU_x)$ are exactly the $n$-faces $E$ with 
$x\in \bigcap_{j\in J_+(E)}A_j$ and $\bigcap_{j\in J_+(E)}B_j\ne\emptyset$ and $J_1(E)\ne\emptyset$. 
From $\gamma F=F$ we obtain that $\emptyset\ne\bigcap_{j\in J}B_j=\bigcap_{j\in J}\gamma B_j$ and $x\in \bigcap_{j\in J_+(\gamma F)}A_j=\bigcap_{j\in J_+(F)}\gamma A_j$. In particular, $B_j\cap\gamma B_j\ne\emptyset$ and $A_j\cap \gamma A_j\ne\emptyset$ for every $j\in J_+(F)$. Since $\calU$ has no self-intersections, this implies $\gamma=1$ and proves injectivity. 
By~\cite{tomdieck-topology}*{Proposition~8.3.1 on p.~203} and rewriting 
$\prod_{j\in J_\half(F)}[0,r_j]$ as $\prod_{k=1}^n[0,r_k(F)]$ 
the pushout property of~\eqref{eq: cantornerve as pushout} follows from the bijectivity of $\Psi_0$ and the fact that $\Psi$ is a quotient map onto its closed image. The image of $\Psi$ is the union of closed $n$-faces which is closed. Let $C\subset \im(\Psi)$ be a subset such that $\Psi^{-1}(C)$ is closed. 
Let $F$ be a compact subset of the domain of $\Psi$ whose translates cover the domain. 
Then 
\[C=\Phi\bigl(\bigcup_{\gamma\in\Gamma} \gamma F\cap \Phi^{-1}(C)\bigr)=\bigcup_{\gamma\in\Gamma}\gamma\cdot\Phi\bigl( F\cap \gamma^{-1}\Phi^{-1}(C)\bigr).\] 
Moreover, each subset $\Phi\bigl( F\cap \gamma^{-1}\Phi^{-1}(C)\bigr)$ is compact and lies in the compact subset $\Phi(F)$. 
The $\Gamma$-action on $N(\calV)$ is proper. Hence $C$ is closed, and $\Phi$ is a quotient map onto its image. 
\end{proof}

\subsection{The map to the rectangular Cantor nerve}\label{sub: map to Cantor nerve}
A Cantor bundle map $\Phi\colon\xxm\to\cantornerve(\calU)\subset X\times N((r_j)_{j\in J})$ is \emph{subordinate to $\calU$} if each component $\Phi_j\colon \xxm\to [0,r_j]$, $j\in J$, of the map $\Phi$ is supported in $A_j\times B_j$. 

Next we construct a specific Cantor bundle map subordinate to $\calU$. 
We define the continuous component map $\Phi_j$ for each $j\in J$ by 
\begin{gather*}
 \Phi_j\colon \xxm\to [0,r_j]\\
 \Phi_j(x,p)=\begin{cases}
0 & \text{ if $d_{\widetilde M}(p, m_j)\ge r_j$ or $x\not\in A_j$,}\\
2(r_j-d_{\widetilde M}(p, m_j)) & \text{ if $\frac{r_j}{2}\le d_{\widetilde M}(p, m_j)\le r_j$ and $x\in A_j$,}\\
r_j & \text{ if $d_{\widetilde M}(p, m_j)<\frac{r_j}{2}$ and $x\in A_j$.}
\end{cases}
\end{gather*}

\begin{definition}
	The \emph{Cantor nerve map} $\Phi$ associated with $\calU$ is the product of the maps $\Phi_j$
\[ \Phi\colon \xxm\to \cantornerve(\calU),~\Phi(x,p)=\bigl( x, (\Phi_j(x,p))_{j\in J}\bigr).\]
\end{definition}

One sees immediately that $\Phi$ is subordinate to $\calU$.

\begin{lemma}\label{lem: cantor nerve map as cantor bundle map}
The Cantor nerve map associated with $\calU$ is a Lipschitz Cantor bundle map. 
\end{lemma}

\begin{proof}
Since each $A_j$ is clopen and $\Phi_j$ is clearly continuous when restricted to $A_j\times\widetilde M$ or $(X\bs A_j)\times\widetilde M$, $\Phi_j$ is continuous. Thus $\Phi$ is continuous. For each $j\in J$ and $\gamma\in \Gamma$ we have 
\[ \Phi_{\gamma j}(\gamma x, \gamma p)=\Phi_j(x, p),\] 
which implies that $\Phi$ is $\Gamma$-equivariant. Let $(x,p)\in \xxm$. 
Let $F$ be an open face in $\cantornerve(\calU)= N(\calU_x)$ that contains $\Phi(x,p)$. Then $x\in \bigcap_{j\in J_+(F)}A_j$ and $p\in \bigcap_{j\in J_+(F)}B_j$. 
We consider the following sets 
\begin{align*}
J_F &\defq\bigl\{ j\in J\mid B_j\cap\bigcap_{i\in J_+(F)}B_i\ne\emptyset\bigr\},\\
C &\defq\bigcap_{\substack{j\in J_F\\x\in A_j}} A_j\cap \bigcap_{\substack{j\in J_F\\x\not\in A_j}} X\bs A_j.
\end{align*}
Let $S$ be the star of $F$ within $N(\calU_x)$. In the proof 
of Lemma~\ref{lem: nerve as cantor subbundle} we showed that $C\times S$ is an open 
box ($S$ was denoted $S'$ in the proof), that is, 
\[ \cantornerve(\calU)\cap C\times S=C\times S.\]
Let $y\in C$ and $q\in \bigcap_{j\in J_+(F)}B_j$. Next we show that $\Phi_j(y,q)=\Phi_j(x,q)$ for every $j\in J$ which implies that $\Phi$ is a product of maps on $C\times \bigcap_{j\in J_+(F)}B_j$. 

First assume that $\Phi_j(y,q)=0$.  
If $q\not\in B_j$, then $\Phi_j(x,q)=0$. 
If $q\in B_j$, then $j\in J_F$, thus $y\not\in A_j$. This implies that $y\not\in C$ or $x\not\in A_j$. Because of $y\in C$ we must have $x\not\in A_j$. Hence 
$\Phi_j(y,q)=\Phi_j(x,q)=0$. 

Second assume that $\Phi_j(y,q)>0$. Then $q\in B_j$ and $j\in J_F$ and $y\in A_j$. 
If $x\not\in A_j$ then $y\in C$ would imply that $y\not\in A_j$. Hence $x\in A_j$. 
Therefore 
\[\Phi_j(y,q)=\Phi_j(x,q)=\begin{cases}
   2(r_j-d_{\widetilde M}(p, m_j)) & \text{ if $\frac{r_j}{2}\le d_{\widetilde M}(p, m_j)\le r_j$,}\\
r_j & \text{ if $d_{\widetilde M}(p, m_j)<\frac{r_j}{2}$.}
\end{cases}.
\]

It remains to show that $\Phi$ is Lipschitz. Each $\Phi_j$ has Lipschitz constant~$2$. Hence $\Phi_j$ has local Lipschitz constant at $(x,p)$ bounded by $2m_x(p)^{1/2}$ where $m_x(p)$ is the multiplicity at $p$ of the cover $\calU_x$. The multiplicity is uniformly bounded by the multiplicity of the cover $\calV$ (albeit not by a dimensional constant). Hence $\Phi$ is a Lipschitz Cantor bundle map. 
\end{proof}

\begin{remark}\label{rem: fiberwise is Guth}
Restricted to a fiber $x\in X$ the map $\Phi_x\colon \widetilde M\to N$ is exactly Guth's nerve map~\cite{guthvolume}*{Section~3} associated 
with the cover $\calU_x$. 
\end{remark}

\section{Volume estimates}\label{sec: volume estimates}

The goal of this section is to prove the existence of a good Cantor cover and to prove the analog of Lemma~4 in~\cite{guthvolume} in the Cantor setting. In Section~\ref{sub: vitali layering} we define layers of an ordinary cover and of a Cantor cover.  
The most important technical result is the existence of layers for Cantor covers with no self-intersections. 
After that, the proof of the analog of Guth's Lemma~4 in Sections~\ref{sub: exponential decay} and~\ref{sub: bounding the volume of image} runs similar to Guth's proof. 

\subsection{Cantor-Vitali layerings of equivariant covers}\label{sub: vitali layering}

Let $\calV$ be a cover of a Riemannian manifold by balls.  
A \emph{Vitali layering} of $\calV$ consists of a finite sequence of subsets $\calV(1), \dots, \calV(n)$ of $\calV$, called \emph{layers}, with the following property: 
\begin{enumerate}
\item The balls within each layer are pairwise disjoint. 
\item For every pair $i<j$ in $\{1,\dots,n\}$ and every ball $B\in\calV(j)$ there is a ball in $\calV(i)$ that meets $B$ and whose radius is greater or equal than the one of $B$.
\item Every ball of $\calV$ appears in precisely one of the layers. 
\end{enumerate} 
We say that a layer $\calV(j)$ is \emph{lower} than a layer $\calV(i)$ if $i<j$. 
The relation $<$ on $\calV(i)$ \emph{associated to the Vitali layering} is defined as the 
smallest partial order $<$ on the layer $\calV(i)$ such that $B<B'$ whenever there is a ball $B''$ from a lower layer that meets both $B$ and $B'$ and the radii of these balls satisfy 
\[ 2r\le r''\le r'.\]
The \emph{core} of a layer is the union of all balls $\frac{1}{10}B$ where $B$ is maximal with respect to the relation $<$ on that layer.

\begin{lemma}\label{lem: layering non-equivariantly}
Let $\calV=\{B_j\mid j\in J\}$ be a good cover of a Riemannian manifold by balls. Let 
$\calV(1), \calV(2), \dots,\calV(n)$ be a Vitali layering of $\calV$ with cores $\calV^c(1),\dots, \calV^c(n)$. 
The following holds for every integer $l\in\{1,\dots,n\}$.   
\begin{enumerate}
\item Every point in $\bigcup\calV^c(l)$ is contained in at most $10^{8(d+3)}$ balls 
from lower layers.  
\item For $l'\ge l$ we have $\bigcup\calV(l')\subset \bigcup 3\calV(l)$. 
\item $\bigcup\calV(l)\subset \bigcup 10\calV^c(l)$. 
\end{enumerate}
\end{lemma}

\begin{proof}
Both statements (1) and~(3) are extracted from the proof of Lemma~4 in Guth's paper~\cite{guthvolume}*{p.~60}. Concerning (1), 
Guth shows that the radii of balls in layers $\ge l$ that contain a point $p\in \frac{1}{10}B_j$ in the $l$-th core, where $B_j\in \calV(l)$ is a maximal ball of radius $r_j$, are pinched in the interval $[\frac{1}{15}r_j, 2r_j]$. The number of such balls is bounded by a dimensional constant~\cite{guthvolume}*{Lemma~3} which can be taken to be $10^{8(d+3)}$. 
Ad (2): Let $l'\ge l$. Let $B_j\in \calV(l')$. Then there is a ball $B_k\in \calV(l)$ that meets $B_j$ and has $r_k\ge r_j$. Hence $B_j\subset 3B_k\subset \bigcup 3\calV(l)$. 
\end{proof}

A \emph{Cantor-Vitali layering} of a Cantor cover $\calU$ consists of a finite sequence of Cantor packings $\calU(1),\dots, \calU(n)$ such that 
	$\calU(1)_x, \dots, \calU(n)_x$ is a Vitali layering of $\calU_x$ for every $x\in X$. 
	Further, the \emph{core} of the $l$-th layer $\calU(l)$ is defined to be union of the cores of $\calU(l)_x$ over all $x\in X$. 

\begin{lemma}\label{lem: expansion of cover}
Let $\calU=\{A_j\times B_j\mid j\in J\}$ be a Cantor packing of~$\xxm$ such that each ball $B_j$ is good. 
Then 
\[ \voltrans_d\Bigl(\bigcup 10\,\calU\Bigr)\le 10^{4(d+3)}\voltrans_d\Bigl(\bigcup\calU\Bigr).\]
\end{lemma}

\begin{proof}
We write $J=\Gamma\times I$ as $\Gamma$-sets. 
Let $\calU_0=\{ A_i\times B_i\mid i\in I\}$. 
Since the $\Gamma$-translates of $A_i\times 10B_i$, $i\in I$, cover the set $\bigcup 10\,\calU$ and 
$\bigcup\calU_0$ is a $\Gamma$-fundamental domain of $\bigcup\calU$ by the packing property, 
we obtain for the transverse measure that 
\begin{align*}
\voltrans_d\Bigl(\bigcup 10\,\calU\Bigr)\le \vol_d\Bigl( \bigcup 10\,\calU_0\Bigl)
        &\le \sum_{i\in I} \mu(A_i)\vol_d(10B_i)\\
        &\le \sum_{i\in I}10^{4(d+3)}\mu(A_i)\vol_d(B_i)\qquad\text{(since $B_i$ is good)}\\
        &=10^{4(d+3)}\voltrans_d\Bigl(\bigcup\calU\Bigr). \qedhere 
\end{align*}
\end{proof}

\begin{lemma}\label{lem: layering}
Let $\calU$ be a good Cantor cover of $\xxm$. Let $\calU(1), \calU(2), \dots$ be a Cantor-Vitali layering of $\calU$. Further, let $\calU^c(l)$ denote the core of $\calU(l)$. Then the following hold. 
\begin{enumerate}
\item For every $x\in X$ and $p\in \bigcup\calU^c(l)_x$ the number of balls that contain $p$ and lie in $\calU(l')_x$ for some $l'\ge l$ is bounded by $10^{8(d+3)}$. \label{eq: layering lemma; well insulated}
\item For $l'\ge l$ we have $\bigcup\calU(l')\subset \bigcup 3\,\calU(l)$. \label{eq: layering lemma; reasonable growth}
\item For $l\ge 1$ we have $\voltrans_d\bigl(\bigcup \calU(l)\bigr)\le 10^{4(d+3)}\cdot\voltrans_d\bigl(\bigcup\calU^c(l)\bigr)$.  \label{eq: layering lemma substantial volume of core}
\end{enumerate}
\end{lemma}

\begin{proof}
(1) and~(2) follow directly from Lemma~\ref{lem: layering non-equivariantly}. 
By applying Lemma~\ref{lem: layering non-equivariantly} (3) fiberwise to the packings $\calU(l)_x$ and then taking unions over $x\in X$ we obtain that 
\[ \bigcup\calU(l)\subset \bigcup 10\,\calU^c(l).\]
Statement (3) follows from~\ref{lem: layering non-equivariantly} (3) and~\ref{lem: expansion of cover}. 
\end{proof}

\begin{theorem}\label{thm: existence of equivariant layers}
Every Cantor cover of $\xxm$ without self-intersections possesses a Cantor-Vitali layering. 
\end{theorem}

\begin{proof}
Let $\calU=\{\gamma\cdot A_j\times \gamma\cdot B_j\mid (\gamma,j)\in \Gamma\times \{1,\dots,n\}\}$ be a Cantor cover 
of $\xxm$ without self-intersections. Let $r_j$ be the radius of $B_j$. 
We assume that the enumeration of balls is such that $r_1\ge r_2\ge\dots\ge r_n$. 
For the purpose of this proof, we call a set of subsets of $\xxm$ of \emph{product type} if it is of the form $\{\gamma\cdot Y_j\times \gamma\cdot B_j\mid (\gamma,j)\in \Gamma\times \{1,\dots,n\}\}$ where each $Y_j$ is a (not necessarily non-empty) clopen subset of $X$. 

We construct the layers by a double induction. For every $l\in\bbN$ and $s\in\{1,\dots,n\}$ we define Cantor packings $\calU^s(l)$ and $\calU(l)$ depending on $\calU(1), \dots, \calU(l-1)$ and on $\calU^{s-1}(l)$ in the following way. 
\begin{align*}
\calU(0)&\defq\emptyset\\
\calU^0(l)&\defq\emptyset~\text{ for every $l\in\{1,\dots, n\}$}&\\
Y^s(l)&\defq\Bigl\{x\in X
\mid\text{ $B_s\in \calU_x$, $B_s$ is disjoint from all balls in $\calU^{s-1}(l)_x$,}\\ &~\qquad\qquad\qquad\text{and $B_s$ is not contained in $\calU(1)_x,\dots, \calU(l-1)_x$}\Bigr\}\\
\calU^s(l)&\defq  \calU^{s-1}(l)\cup\bigl\{ \gamma Y^s(l)\times \gamma B_s\mid\gamma\in\Gamma\bigr\}\\
\calU(l)&\defq\calU^n(l)
\end{align*}  

\noindent\emph{Each set $\calU^s(l)$ is of product type}: It suffices to verify that each set $Y^s(l)$ is clopen. By induction we may assume that $\calU^{s-1}(l)$ and $\calU(i)$ are of product type for every $i\in\{1,\dots, l-1\}$, that is, 
    \begin{align*}
    \calU^{s-1}(l) &= \bigl\{\gamma\cdot Y^{s-1}_j(l)\times \gamma\cdot B_j\mid (\gamma,j)\in \Gamma\times \{1,\dots,n\}\bigr\},\\
    \calU(i) &= \bigl\{\gamma\cdot Y_j(i)\times \gamma\cdot B_j\mid (\gamma,j)\in \Gamma\times \{1,\dots,n\}\bigr\}    \end{align*}
     for suitable clopen subsets $Y_j^{s-1}(l)\subset X$ and $Y_j(i)\subset X$. For every $j\in\{1,\dots,n\}$ let 
    \[ F_j^s\defq \bigl\{\gamma\in\Gamma\mid B_s\cap \gamma B_j\ne\emptyset\bigr\}.\]
    The subset $F_j^s$ is finite. Then we have 
 \[ Y_s(l)=A_s\cap \Bigg( X\bs~\Bigl(\bigcup_{\mathclap{\substack{j\in \{1,\dots,n\}\\\gamma\in F_j^s}}} \gamma Y^{s-1}_j(l)~\cup~\bigcup_{\mathclap{\substack{j\in\{1,\dots,n\}\\ i\in\{1,\dots, l-1\}\\\gamma\in F^s_j}}} \gamma Y_j(i)\Bigr)\Bigg),
     \]
    which is clearly a clopen subset. \smallskip\\
\noindent\emph{Each $\calU^s(l)$ is a Cantor packing}: Equivalently, we may show that $\calU^s(l)_x$ is a packing for every $x\in X$. Since $Y^s(l)\subset A_s$  and $\calU$ has no self intersections, 
$\bigl\{ \gamma Y^s(l)\times \gamma B_s\mid\gamma\in\Gamma\bigr\}$ is a Cantor packing for every $s\in \{1,\dots, n\}$ and $l\in\bbN$. 
By induction we may assume that $\calU^{s-1}(l)$ is a Cantor packing or, equivalently, $\calU^{s-1}(l)_x$ is packing for every $x\in X$. Hence if for some $x\in X$ there is a non-empty intersection of two balls in $\calU^s(l)_x$ it has to be a ball $\gamma B_s$ in $\calU^s(l)_x$ intersecting a ball in $\calU^{s-1}(l)_x$. In particular, $x\in \gamma Y^s(l)$. 
By $\Gamma$-equivariance the ball $B_s$ lies in $\calU^s(l)_{\gamma^{-1}x}$ and intersects a ball in $\calU^{s-1}(l)_{\gamma^{-1} x}$. This contradicts $\gamma^{-1}x\in Y^s(l)$. So $\calU^s(l)_x$ is indeed a packing for every $x\in X$. \smallskip\\
\emph{The sequence $\calU(1), \calU(2), \dots, \calU(n)$ is a Cantor-Vitali layering of $\calU$}: 
Let $x\in X$. We show that $\calU(1)_x, \calU(2)_x, \dots, \calU(n)_x$ is a Vitali layering of $\calU_x$. 
Let $i,j\in \{1,\dots, n\}$ with $i<j$. Consider a ball $\gamma B_s$ in $\calU(j)_x$. In particular, $x\in\gamma Y^s(j)$. 
We have to find a ball in 
$\calU(i)_x$ that meets $\gamma B_s$ and has radius at least $r_s$. Equivalently, we have to find a ball in $\calU(i)_{\gamma^{-1}x}$ that meets $B_s\in \calU(j)_{\gamma^{-1}x}$ and has radius at least~$r_s$. If such a ball did not exist, then $B_s$ would lie in $\calU^s(i)_{\gamma^{-1}x}\subset \calU(i)_{\gamma^{-1}x}$ or $B_s$ would lie in $\calU(1)_{\gamma^{-1}x},\dots,\calU(i-1)_{\gamma^{-1}x}$. Both possibilities 
are absurd. 

Let $x\in X$. If $s\in\{1,\dots,n\}$ is the smallest number so that $B_s\in \{B_1,\dots, B_n\}\cap\calU_x$ but $B_s$ is not in one of the layers $\calU(1)_x,\dots, \calU(n-1)_x$ then $B_s\in \calU^s(n)_x\subset \calU(n)_x$. Hence every ball of $\{B_1,\dots, B_n\}\cap \calU_x$ is in one of the layers $\calU(1)_x, \dots, \calU(n)_x$. 
By equivariance each ball of $\calU_x$ appears in one of the layers $\calU(1)_x, \dots, \calU(n)_x$. 
It is clear from the construction that each ball also appears in at most one layer. 
\end{proof}

\begin{proof}[Proof of Theorem~\ref{thm:EquivariantCover}]
   By \cite{guthvolume}*{Lemma 1} around every point of $\widetilde{M}$ there is a good ball. We choose a $\Gamma$-fundamental domain of $\widetilde M$ and a good ball for every point in the fundamental domain. Since $M$ is compact we can select a finite subset $B_1,\dots ,B_n$ of these balls such that the projections of $\frac{1}{12}B_1,\dots, \frac{1}{12}B_n$ cover all of $M$. Hence the translates of $X\times \tfrac{1}{12}B_1,\dots, X\times\tfrac{1}{12}B_n$ form a Cantor cover of~$\xxm$.   
   
   By properness of the $\Gamma$-action on $\widetilde M$ the set 
   \[F\defq\bigl\{\gamma\in\Gamma\mid \exists_{i\in\{1,\dots,n\}} B_i\cap\gamma B_i\ne\emptyset   \bigr\}\] 
   is finite. 
   Next we show that there is clopen partition $X=A_1\cup \dots\cup A_r$ of $X$ such that 
	$\gamma A_i\cap A_i=\emptyset$ for every $\gamma\in F$ and every $i\in\{1,\dots, r\}$. 
	To this end, choose a metric $d$ on $X$ that induces the topology on $X$. 
    For every $\gamma\in F$ the continuous map 
	\[ X\to [0,\infty),~~x\mapsto d(\gamma x, x)\] 
	takes on a minimum $\epsilon_\gamma$ which is strictly positive as the $\Gamma$-action on $X$ is free. Let $\epsilon\defq\min_{\gamma\in F}\epsilon_\gamma>0$. 
	Now we pick a cover of $X$ by clopen subsets of diameter at most $\epsilon/2$. Then there is 
   a subordinate clopen partition $X=A_1\cup \dots\cup A_r$ of $X$. Since the diameter of each $A_i$ is at most $\epsilon/2$ we have $\gamma A_i\cap A_i=\emptyset$ for every $\gamma\in F$ and every $i\in\{1,\dots, r\}$. Then the translates of the sets $A_i\times B_j$, $i=1,\dots, r$, $j=1,\dots, n$, form a Cantor cover $\calU'$ indexed over $\Gamma\times\{1,\dots,r\}\times\{1,\dots,n\}$ without self-intersections. Also the Cantor cover $6\;\calU'$ has no self-intersections. 

	According to Theorem~\ref{thm: existence of equivariant layers} the Cantor cover $\calU'$ has a Cantor-Vitali layering. Let $\calU'(1)$ be the top layer. We claim that $\calU\defq 6\;\calU'(1)$ is a 
   good Cantor cover without self-intersections: Since the top layer is always a Cantor packing, 
$\tfrac{1}{6}\,\calU=\calU'(1)$ is a Cantor packing. Further, $\tfrac{1}{2}\,\calU=3\,\calU'$ is a Cantor cover by~Lemma~\ref{lem: layering}~(2). Finally, since $6\;\calU'$ has no self intersections, $\calU$ has no self-intersections either. 
\end{proof}

\subsection{Exponential decay of the volume of the high multiplicity set}\label{sub: exponential decay}
Similar remarks as in Guth's paper on the multiplicity are valid here: 
The (fiberwise) multiplicity of a Cantor cover is bounded but not in terms of a universal constant. Therefore we cannot bound later the Lipschitz constant of the Cantor nerve map universally. However, the volume of the high multiplicity set decays exponentially. The argument for that is basically the same as the one in~\cite{guthvolume}*{p.~61/62}, only 
with volume replaced by transverse measure and so on.

\begin{theorem}\label{thm: volume of high multiplicity set}
Let $\calU$ be a good Cantor cover of $\xxm$ with no self-intersections.  
For $(x,p)\in\xxm$ let $m_x(p)$ be the multiplicity of the point $p$ with respect to the cover $\calU_x$ of $\widetilde M$. There are dimensional constants $\alpha(d)>0$ and $\beta(d)>0$ such that for every $\lambda\ge 1$
\[\voltrans_d\Bigl(\bigl\{(x,p)\mid m_x(p)\ge \lambda+\beta(d)\bigr\}\Bigr)\le e^{-\alpha(d)\lambda}\cdot\vol(M).\]
\end{theorem}

\begin{remark}
	In the above statement we can choose $\alpha=-\log(1-10^{-16(d+3)})$ and 
	$\beta=10^{8(d+3)}$. This is a consequence of the proof below. 
\end{remark}

\begin{proof}
According to Theorem~\ref{thm: existence of equivariant layers} we pick a Cantor-Vitali layering $\calU(1), \calU(2), \dots$ of $\calU$. Let $\calU^c(l)\subset\calU(l)$ be the 
associated core of $\calU(l)$. 
Consider the subsets 
\[ L^{\theta}(\lambda)=\bigl\{ (x,p)\in\xxm\mid (x,p)\in \bigcup\calU(l)\text{ for at least $\theta$ values of $l$ in the range $l\ge \lambda$}\bigr\}.\]
By Lemma~\ref{lem: layering non-equivariantly}~\eqref{eq: layering lemma; reasonable growth} and Lemma~\ref{lem: expansion of cover} we obtain that 
\begin{equation}\label{eq: estimate for L1 layers}
 \voltrans_d\bigl(L^1(\lambda)\bigr)\le \voltrans_d\bigl(\bigcup 3\calU(\lambda)\bigr)\le 10^{4(d+3)}\voltrans_d\bigl(\bigcup\calU(\lambda)\bigr).
\end{equation}
With the constant $\beta(d)=10^{8(d+3)}$ from Lemma~\ref{lem: layering non-equivariantly}~\eqref{eq: layering lemma; well insulated} we define $T(\lambda)$ as the average volume 
\[ T(\lambda)\defq \frac{1}{\beta(d)}\sum_{\theta=1}^{\beta(d)}\voltrans_d\bigl(L^\theta(\lambda)\bigr).\]
An element $(x,p)\in L^\theta(\lambda)\bs L^\theta(\lambda+1)$ lies in $\calU(\lambda)$ and in exactly $\theta-1$ different layers lower than~$\lambda$. With Lemma~\ref{lem: layering non-equivariantly}~\eqref{eq: layering lemma; well insulated} this implies that 
\[\bigcup \calU^c(\lambda)\subset \bigcup_{\theta=1}^{\beta(d)}\bigl( L^\theta(\lambda)\bs L^\theta(\lambda+1)\bigr).\]
Note that $\calU^c(\lambda)$ and $L^\theta(\lambda)$ are $\Gamma$-invariant subsets to which we can apply the measure $\voltrans_d$. The above inclusion yields 
\begin{align*}
\voltrans_d\bigl(\calU^c(\lambda)\bigr) \le \voltrans_d\bigl( \bigcup_{\theta=1}^{\beta(d)}\bigl( L^\theta(\lambda)\bs L^\theta(\lambda+1)     \bigr)
&\le \sum_{\theta=1}^{\beta(d)} \voltrans_d\bigl (L^\theta(\lambda)\bs L^\theta(\lambda+1)\bigr)\\
&=\sum_{\theta=1}^{\beta(d)} \voltrans_d\bigl( L^\theta(\lambda)\bigr)-\voltrans_d\bigl( L^\theta(\lambda+1)\bigr)\\
&\le \beta(d)\bigl( T(\lambda)-T(\lambda+1)\bigr).
\end{align*}
We conclude further that 
\begin{align*}
T(\lambda)-T(\lambda+1) \ge \frac{1}{\beta(d)}\voltrans_d\bigl( \bigcup\calU^c(\lambda)\bigr)
  	&\ge \frac{10^{-4(d+3)}}{\beta(d)}\voltrans_d\bigl( \bigcup\calU(\lambda)\bigr)~\text{ (Lemma~\ref{lem: layering non-equivariantly}~\eqref{eq: layering lemma substantial volume of core})}  \\
  	&\ge \frac{10^{-8(d+3)}}{\beta(d)}\voltrans_d\bigl( L^1(\lambda)\bigr)~~~\text{ (using~\eqref{eq: estimate for L1 layers})} \\
     	&= 10^{-16(d+3)}\voltrans_d\bigl( L^1(\lambda)\bigr)\\
                              	&\ge 10^{-16(d+3)} T(\lambda).
\end{align*}
Hence $T(\lambda+1)\le (1-10^{-16(d+3)})T(\lambda)$. So $T$ decays exponentially. More precisely, 
we obtain that for $\lambda\ge 1$ 
\begin{equation}\label{eq: exponential decay of T}
	T(\lambda)\le e^{-\alpha(d)\lambda}\cdot T(1)\le  e^{-\alpha(d)\lambda}\cdot\vol(M),
\end{equation}
where $\alpha=\alpha(d)=-\log(1-10^{-16(d+3)})$. 
Finally, we relate the function $T$ to the volume of the high multiplicity subset. 
Let $(x,p)\in\xxm$ be a point with $m_x(p)\ge\lambda+\beta(d)$. Since the balls in layer $\calU(l)_x$ are disjoint, the point $(x,p)$ lies in at most $\lambda$ many balls from the layers $\bigcup\calU(1)_x, \dots, \bigcup\calU(\lambda)_x$. 
Hence $(x,p)\in L^{\beta(d)}(\lambda)$. We conclude that 
\begin{equation*}
\voltrans_d\Bigl(\bigl\{(x,p)\mid m_x(p)\ge \lambda+\beta(d)\bigr\}\Bigr)\le \voltrans_d\bigl(L^{\beta(d)}(\lambda)\bigr)
\le T(\lambda)
\le e^{-\alpha(d)\lambda}\vol(M).\qedhere 
\end{equation*}
\end{proof}

\subsection{Bounding the transverse volume of the image of the nerve map}\label{sub: bounding the volume of image}
In the sequel let $\calU=\{A_j\times B_j\mid j\in J\}$ be a good Cantor cover of $\xxm$ with no self-intersections. Let $\Phi\colon\xxm\to \cantornerve(\calU)$ be the Cantor nerve map. The following (non-equivariant) statement only concerns the fiberwise nerve map $\Phi_x\colon\widetilde M\to\calN(\calU_x)$. In view of Remark~\ref{rem: fiberwise is Guth} we can cite the following theorem from Guth's paper. Recall that the constant $V_1$ denotes an upper bound on the volume of $1$-balls of $\widetilde M$ (see Theorem~\ref{thm: main result}) and that $d(F)$ denotes the dimension of a face~$F$. 

\begin{theorem}[\cite{guthvolume}*{Lemma~5}]
   \label{thm: volume restricted to star}
There are dimensional constants $C(d)>0$ and $\beta(d)>0$ so that for every $x\in X$ and every open face $F\in N(\calU_x)$ we have \[ \vol_d\Bigl(\Phi\vert_{\Phi^{-1}\bigl(\{x\}\times \star(F)\bigr)}\Bigr)<C(d)\cdot V_1\cdot r_1(F)^{d+1}\cdot e^{-\beta(d)\cdot d(F)}.\]
\end{theorem}

We now fix a dimensional constant $\beta(d)>0$ that satisfies the conclusions of Theorems~\ref{thm: volume of high multiplicity set} and~\ref{thm: volume restricted to star}. 

\begin{theorem}\label{thm: volume of nerve map}
There is a dimensional constant $C(d)>0$ such that 
\[\voltrans_d\bigl( \Phi  \bigr)	\le C(d)\cdot\vol(M).\] 
\end{theorem}

\begin{proof}
Let $n$ be the maximal multiplicity of the cover $\{B_j\mid j\in J\}$ of 
$\widetilde M$. 
For $i\in\bbN_0$ we define the $\Gamma$-invariant subsets 
\begin{align*}
	S_i&\defq \bigl\{(x,p)\in\xxm\mid i+\beta(d)\le m_x(p)< 1+i+\beta(d)\bigr\},\\
    S&\defq \bigl\{(x,p)\in\xxm\mid m_x(p)<\beta(d)\bigr\}.
\end{align*}
Restricted to $S_i$ or $S$ the map $\Phi$ is fiberwise Lipschitz with Lipschitz constant at most $2(1+i+\beta(d))^{1/2}$ or $2\beta(d)^{1/2}$, respectively (cf.~the proof of Lemma~\ref{lem: cantor nerve map as cantor bundle map}). 
Therefore we have 
\begin{align*}
\voltrans_d(\Phi)&\le \voltrans_d\bigl(\Phi\vert_S\bigr)+\sum_{i=0}^{n}\voltrans_d\bigl( \Phi\vert_{S_i}\bigr)\\
	&\le\bigl( 2\beta(d)^{1/2})^d\cdot \vol(M)+\sum_{i=0}^n \bigl( 2(1+i+\beta(d))^{1/2}\bigr)^d\cdot\voltrans_d(S_i)\\
	&\le \Bigr(2^d\beta(d)^{d/2}+\sum_{i=0}^n 2^d(1+i+\beta(d))^{d/2}e^{-\alpha(d)\cdot i}\Bigr)\cdot\vol(M)\qquad\text{ (Theorem~\ref{thm: volume of high multiplicity set})}\\
	&\le C(d)\cdot \vol(M),
\end{align*}
where we set $C(d)$ to be the value of the convergent series 
\[ C(d)\defq \Bigr(2^d\beta(d)^{d/2}+\sum_{i=0}^\infty 2^d(1+i+\beta(d))^{d/2}e^{-\alpha(d)\cdot  i}\Bigr).\]
Since $\alpha(d), \beta(d)$ are dimensional constants, so is $C(d)$. 
\end{proof}

We now fix a dimensional constant $C(d)>0$ that satisfies the conclusions of Theorems~\ref{thm: volume restricted to star} 
and~\ref{thm: volume of nerve map}. 

\section{Pushing the equivariant nerve map down to the \texorpdfstring{$d$}{d}-skeleton}\label{sec: homotoping down}

In this section we deform the Cantor nerve map of a Cantor cover $\calU$ to the $d$-skeleton with $d=\dim(M)$. The non-equivariant counterpart in Guth's paper~\cite{guthvolume} is the one where tools from geometric measure theory enter. An essential tool is the following result. 
\begin{theorem}[Pushout lemma~\cite{guthuryson}*{Lemma~0.6}]\label{thm: pushout lemma}
For each dimension $d\ge 2$ there is a constant $\sigma(d)>0$ so that the following 
holds. Suppose that $N$ is a compact piecewise smooth $d$-dimensional manifold with boundary. Suppose that $K\subset \bbR^n$ is a convex set, and $\phi\colon (N,\partial N)\to (K,\partial K)$ is a piecewise smooth map. Then $\phi$ may be homotoped into a map $\phi'$ so that the following holds. 
\begin{itemize}
\item The map $\phi'$ agrees with $\phi$ on $\partial N$. 
\item $\vol_d(\phi')\le \vol_d(\phi)$.	
\item The image $\phi'(N)$ lies in the $\sigma(d)\cdot\vol_d(\phi)^{1/d}$-neighborhood of $\partial K$. 
\end{itemize}
\end{theorem}
Here is a list of dimensional constants to be used below. 
\begin{center}
\begin{tabular}{l|l}
$\beta(d)$ & defined after Theorem~\ref{thm: volume restricted to star}.\\
$C(d)$     & defined after Theorem~\ref{thm: volume of nerve map}.\\
$\sigma(d)$ & see the Pushout lemma above.\\
\end{tabular}
\end{center}
Next we recall the definition of thin and thick faces from Guth's paper. 
To this end, we choose 
$\epsilon>0$ small enough so that 
\begin{equation}\label{eq: epsilon constant}
\prod_{k=d+1}^\infty \Bigl( 1-2\bigl(3\cdot\epsilon\cdot\sigma(d)^d\cdot e^{-\beta(d)\cdot k}\bigr)^{1/d}\Bigr)^{-d}<2~\text{ and }~
2\cdot\epsilon \cdot e^{\beta(d)\cdot d}<1.
\end{equation}
The infinite product converges by the exponential decay in the term $e^{-\beta(d)\cdot k}$. Since the value of $\epsilon$ only depends on $d$ and the dimensional constant $\beta(d)$, it is a dimensional constant and we write $\epsilon=\epsilon(d)$. 
Let $F$ be an open $k$-face with side lengths $r_1(F)\le \dots \le r_k(F)$. We call the face $F$ \emph{thin} if 
\begin{equation}\label{eq: thin face}
	 C(d)\cdot V_1\cdot r_1(F)<\epsilon(d).
\end{equation}
Otherwise it is called \emph{thick}. 
Next we play off the framework developed in Sections~\ref{sec: category cantor bundles} and~\ref{sec: rectangular cantor nerves} to transfer Guth's methods to our setting. 

\subsection{Compression map}

Let $\delta\in (0,\half)$. The \emph{$\delta$-truncation} $\cantornerve(\calU)_\delta^{(n)}$ of the $n$-skeleton $\cantornerve(\calU)^{(n)}$ is obtained from $\cantornerve(\calU)^{(n)}$ by removing 
a smaller cuboid inside each $n$-dimensional face. Referring to the pushout~\eqref{eq: cantornerve as pushout}, we obtain $\cantornerve(\calU)_\delta^{(n)}$ by removing 
\[ \coprod\limits_{F\in C_n} \Bigl(\bigcap\limits_{j\in J_+(F)} A_j\Bigr)\times \Gamma\times\Bigl(\prod\limits_{k=1}^n [\delta r_k(F), (1-\delta) r_k(F)]\Bigr).\]
The self map $R_\delta$ of the cuboid given by $F$ stretches linearly the 
interval $[\delta r_k(F), (1-\delta)r_k(F)]$ to $[0, r_k(F)]$ and sends $[0,\delta r_k(F)]$ to $0$ and $[(1-\delta)r_k(F),r_k(F)]$ to $r_k(F)$ in each coordinate. 
The \emph{$\delta$-compression map} on the $n$-skeleton is the map $\Rho_\delta\colon \cantornerve(\calU)^{(n)}\to \cantornerve(\calU)^{(n)}$ such that 
$\Rho_\delta$ is the identity on the $(n-1)$-skeleton and 
on every summand of the left lower corner of the pushout~\eqref{eq: cantornerve as pushout} it is the equivariant 
extension of 
\[\Bigl(\bigcap\limits_{j\in J_+(F)} A_j\Bigr)\times \{1\}\times\prod\limits_{k=1}^n [0, r_k(F)]\xrightarrow{\id\times R_\delta} \cantornerve(\calU)^{(n)}.
\]  
By Lemma~\ref{lem: pushouts of Cantor bundles} the map $\Rho_\delta$ is a Cantor bundle map. 
\begin{remark}\label{rem: compression map}
Obviously, we have 
\[ \Rho_\delta\bigl( \cantornerve(\calU)_\delta^{(n)}\bigr)\subset\cantornerve(\calU)^{(n-1)}.\]
The map $\Rho_\delta$ is a Lipschitz Cantor bundle map with Lipschitz constant $(1-2\delta)^{-1}$. 
\end{remark}

\subsection{Federer-Fleming deformation in thick faces}\label{sub: federer-fleming}

Let $n>d$. Let \[\Phi\colon\xxm\to\cantornerve(\calU)^{(n)}\] be a Lipschitz Cantor bundle map which is subordinate to $\calU$. 
Referring to the pushout~\eqref{eq: cantornerve as pushout}, we consider the subset 
\[ L_F\defq \Bigl(\bigcap_{j\in J_+(F)} A_j\Bigr)\times \{1\}\times\prod_{k=1}^n [0, r_k(F)]\]
of $\cantornerve(\calU)^{(n)}$. Let $L_F^\circ\subset L_F$ be similarly defined as $L_F$ by   
taking the interior of the cuboid in the right hand factor. 
By applying Lemma~\ref{lem: box decomposition for maps} to each box $L_F$ and taking a common refinement we obtain a clopen partition $X=B_1\cup \dots \cup B_m$ such that the following holds. 
\begin{itemize}
\item For every $i\in\{1,\dots,m\}$ and every $F\in C_n$ we have either $B_i\subset \bigcap_{j\in J_+(F)}A_j$ or $B_i\cap \bigcap_{j\in J_+(F)}A_j=\emptyset$. 	
\item $\Phi^{-1}(L_F)\vert_{B_i}$ is a box (possibly empty). So we have $\Phi^{-1}(L_F)\vert_{B_i}=B_i\times W_{i,F}$ for some subset $W_{i,F}\subset\widetilde M$. 
\item If $B_i\subset \bigcap_{j\in J_+(F)}A_j$, then $\Phi\vert_{B_i\times W_{i,F}}=\id_{B_i}\times h_{i,F}$ for some Lipschitz map 
$h_{i,F}\colon W_{i,F}\to \prod_{k=1}^n [0,r_k(F)]$. 
\end{itemize}
Let us denote the restriction of $h_{i, F}$ to the $h_{i,F}$-preimage of the interior of the cube by $h_{i,F}^\circ$. 
We apply the Federer-Fleming deformation theorem to $h_{i,F}^\circ$ for each thick $F\in C_n$ in the same way as in~\cite{guthvolume}*{p.~70}. It gives us points $p_{i, F}$ in the 
interior of the cube $\prod_{k=1}^n [0,r_k(F)]$ such that for the radial projections $\pr_{i,F}$ from the interior of the cube minus the point $p_{i,F}$ to the boundary of the cube we have 
\begin{equation}\label{eq: federer fleming volume}
\int J_{\pr_{i,F}\circ h_{i,F}^\circ}d\vol_d^{\widetilde M}\le G(V_1,d)\cdot \int J_{h_{i,F}^\circ}d\vol_d^{\widetilde M}
\end{equation}
for a constant $G(V_1,d)\ge 1$ only depending on $V_1$ and $d$. Here $\vol_d^{\widetilde M}$ is 
the Riemannian volume measure on $\widetilde M$ induced by $M$. 

The same two remarks in~\cite{guthvolume}*{p.~70} apply here: 
First, the stretching factor $G(V_1,d)$ depends on the dimension $d(F)$ of the face. However, the maximal dimension of a thick face only depends on $V_1$ and $d$ as noted before. 
Second, the usual Federer-Fleming construction takes place in a cube rather than a cuboid. The fact that the face is thick puts a limit on how distorted it is in comparison to a cube. 
By properness of $\Phi$ the infimal distance $\epsilon$ of $p_{i,F}$ to $\im h_{i,F}$ over all thick $F\in C_n$ and $i\in\{1,\dots,m\}$ is strictly positive. 

Next we describe two Cantor subbundles $Z_1^{(n)}$ and $Z_2^{(n)}$ of 
$\cantornerve(\calU)^{(n)}$. 
The first one $Z_1^{(n)}$ is obtained by removing $\epsilon$-balls around the $\Gamma$-translates of the points $p_{i,F}$, more precisely, by removing 
\[
\coprod\limits_{\substack{\text{$F\in C_n$ thick}\\i=1,\dots,m}} \Bigl(\bigcap\limits_{j\in J_+(F)} A_j\cap B_i\Bigr)\times \Gamma\times B(p_{i,F},\epsilon).
\]
The second one $Z_2^{(n)}$ is obtained by removing all thick $n$-faces, 
that is, $Z_2^{(n)}$ is given by a similar pushout as 
in~\eqref{eq: cantornerve as pushout} with the coproduct in the lower left corner running only over thin $F\in C_n$. 

By equivariance, $\im\Phi\subset Z_1^{(n)}$, so the map $\Phi$ factors as \[\xxm\to Z_1^{(n)}\hookrightarrow \cantornerve(\calU)^{(n)}.\] The maps $\pr_{i,F}$ and the identity on the $(n-1)$-skeleton and thin faces yield by the pushout property (see Lemma~\ref{lem: pushouts of Cantor bundles}) a Cantor bundle map $Z_1^{(n)}\to Z_2^{(n)}$. It depends on the choice of the points $p_{i,F}$. 
For every such choice the composition of $\Phi\colon \xxm\to Z_1^{(n)}$ with the radial projection map $Z_1^{(n)}\to Z_2^{(n)}\hookrightarrow\cantornerve(\calU)^{(n)}$ 
is called a \emph{Federer-Fleming deformation} of $\Phi$. 
 
Since each $\pr_{i,F}$ is Lipschitz when restricted to the complement of a small ball around the center, a Federer-Fleming deformation of $\Phi$ is still Lipschitz. We cannot bound the Lipschitz constant by a dimensional constant, though, as we cannot control the above quantity~$\epsilon$. 

\begin{lemma}\label{lem: federer-fleming volume}
Let $\Phi\colon \xxm\to\cantornerve(\calU)^{(n)}$ be a Lipschitz Cantor bundle map which is subordinate to $\calU$. 
Let $\Phi'$ be a Federer-Fleming deformation of $\Phi$. Then $\Phi'$ is a Lipschitz Cantor bundle map subordinate to $\calU$ and 
\[\voltrans_d(\Phi')\le G(V_1,n)\cdot \voltrans_d(\Phi).\] 
\end{lemma}

\begin{proof}
Let $E\subset \cantornerve(\calU)^{(n-1)}$ be a Borel $\Gamma$-fundamental domain of the $(n-1)$-skeleton. Then 
\[\Phi^{-1}(E)\cup \bigcup_{F\in C_n} \Phi^{-1}\bigl(L_F^\circ\bigr)\]
is a Borel $\Gamma$-fundamental domain of $\xxm$. The above union is disjoint. By~\eqref{eq: area formula} we obtain that 
\begin{equation*}
\voltrans_d(\Phi') = \int_{\Phi^{-1}(E)} J_{\Phi'} d\vol_d+\sum_{F\in C_n} \int_{\Phi^{-1}(L_F^\circ)}J_{\Phi'}d\vol_d.
\end{equation*}
On the $\Phi$-preimage of the $(n-1)$-skeleton the maps $\Phi$ and $\Phi'$ coincide. The maps $\Phi$ and $\Phi'$ also coincide on $\Phi^{-1}(L_F^\circ)$ for each thin $F\in C_n$. Hence 
\[ \voltrans_d(\Phi') =\int_{\Phi^{-1}(E)} J_{\Phi} d\vol_d+\sum_{\substack{F\in C_n\\\text{$F$ thin}}} \int_{\Phi^{-1}(L_F^\circ)}J_{\Phi}d\vol_d+\sum_{\substack{F\in C_n\\\text{$F$ thick}}} \int_{\Phi^{-1}(L_F^\circ)}J_{\Phi'}d\vol_d.
\]

For thin $F\in C_n$ the set 
$\Phi^{-1}(L_F^\circ)$ is the disjoint union of products of $B_i\cap \bigcap_{j\in J_+(F)}A_j$ and the domain of $h_{i, F}^\circ$ where $i$ runs over $1,\ldots, m$. Recall that each $B_i$ is either disjoint from or contained in $\bigcap_{j\in J_+(F)}A_j$. 
For thick $F\in C_n$ we obtain from~\eqref{eq: federer fleming volume} that 
\begin{align*}
\int_{\Phi^{-1}(L_F^\circ)}J_{\Phi'}d\vol_d &= \sum_{\substack{i=1,\ldots,m\\B_i\subset \bigcap_{j\in J_+(F)}A_j}} \mu(B_i)\int J_{\pr_{i,F}\circ h_{i,F}^\circ}d\vol_d^{\widetilde M}\\
&\le G(V_1,d)\cdot\sum_{\substack{i=1,\ldots,m\\B_i\subset \bigcap_{j\in J_+(F)}A_j}} \mu(B_i)\int J_{h_{i,F}^\circ}d\vol_d^{\widetilde M}\\
&=G(V_1,d)\cdot\int_{\Phi^{-1}(L_F^\circ)}J_{\Phi}d\vol_d.
\end{align*}
The claimed inequality follows from Theorem~\ref{thm: Area formula}. 
\end{proof}

\subsection{Guth's pushout lemma for thin faces}

While the Federer-Fleming deformation allows us to deform the nerve map away from thick faces, the pushout deformation of this subsection, in combination with the compression map, allows us to deform the nerve map away from thin faces. 

We retain the setup at the beginning of Section~\ref{sub: federer-fleming}. We additionally require 
that $\Phi$ is piecewise smooth on each fiber. 
We apply exactly the same argument as in~\cite{guthuryson}*{p.~206} to each thin face and the maps $h_i$;  
we only have to take care that everything fits together to a Cantor bundle map in the end. 

For each open thin face $F\in C_n$ and each $i\in\{1,\dots,m\}$ one chooses 
a convex subset $K_{i, F}$ of $F$ containing almost all of $F$ but in general position with respect to $h_i$: The $h_i$-preimage of $K_{i,F}$ is a piecewise smooth submanifold $S_{i, F}$ of $\widetilde M$ with boundary $\partial S_{i, F}$ which is the $h_i$-preimage of $\partial K_{i,F}$. We apply the Pushout Lemma~\ref{thm: pushout lemma} to 
\[h_i\vert_{S_{i,F}}\colon (S_{i,F}, \partial S_{i,F})\to (K_{i,F},\partial K_{i,F}).\]
The result is a map $\tilde h_{i,F}$ so that 
$\tilde h_{i,F}$ coincides with 
$h_i$ on $\partial S_{i,F}$ and 
\begin{equation}\label{eq: pushout lemma volume}
\vol_d(\tilde h_{i, F})\le \vol_d(h_i\vert_{S_{i,F}}), 
\end{equation}
and  
the image of $\tilde h_{i, F}$ lies in the $w_{i,F}$-neighborhood of 
$\partial K_{i,F}$ where 
\[ w_{i,F}\defq \sigma(d)\cdot\vol_d\bigl(h_i\vert_{S_{i,F}}\bigr)^{1/d}.\]
We modify the map $\Phi$ as follows. 
The Cantor bundle $\xxm$ contains the $\Gamma$-invariant subspace 
\begin{equation}\label{eq: subspace}
\bigcup_{\substack{\gamma\in\Gamma\\F\in C_n\text{ thin}\\i=1,\dots,m}} \mkern-12mu \gamma\cdot\Bigr(\bigcap_{j\in J_+(F)} A_j\cap B_i\Bigl)\times \gamma\cdot S_{i, F}=\Phi^{-1}\Bigl(\mkern-12mu\bigcup_{\substack{\gamma\in\Gamma\\F\in C_n\text{ thin}\\i=1,\dots,m}}\Bigr(\bigcap_{j\in J_+(F)} A_j\cap B_i\Bigl)\times\Gamma\times K_{i,F}\Bigr)
\end{equation}
which is a disjoint union of subspaces and in each fiber a disjoint union of piecewise smooth $d$-dimensional submanifolds with boundaries. 
We make the subspace~\eqref{eq: subspace} slightly smaller by replacing each $S_{i,F}$ with its interior and 
then consider the complement, which we denote by $R\subset\xxm$, of this smaller subspace. Then $\xxm$ can be 
expressed as the pushout of Cantor bundles 
\[\begin{tikzcd}
   \coprod\limits_{\substack{\gamma\in\Gamma\\F\in C_n\text{ thin}\\i=1,\dots,m}} \mkern-12mu \Bigr(\bigcap_{j\in J_+(F)} A_j\cap B_i\Bigl)\times\Gamma \times\partial S_{i, F}\ar[r,hook] \ar[d,hook,shorten <= -2.5em]& R\ar[d,hook]\\
   \coprod\limits_{\substack{\gamma\in\Gamma\\F\in C_n\text{ thin}\\i=1,\dots,m}} \mkern-12mu \Bigr(\bigcap_{j\in J_+(F)} A_j\cap B_i\Bigl)\times\Gamma \times S_{i, F}\ar[r,hook] &\xxm.
\end{tikzcd}
\]
By the pushout property (Lemma~\ref{lem: pushouts of Cantor bundles}) we obtain a new Cantor bundle map $\Phi'\colon \xxm\to \cantornerve(\calU)^{(n)}$ that coincides with $\Phi$ on $R$ and is the 
equivariant extension of $\id\times \tilde h_{i,F}$ on each 
summand of the left lower corner of the pushout. 
The map $\Phi'$ is still Lipschitz and piecewise smooth on each fiber. We say that the map $\Phi'$ is a \emph{pushout deformation} of~$\Phi$. 

Similarly as in Lemma~\ref{lem: federer-fleming volume}, we conclude the following statement from~\eqref{eq: pushout lemma volume}. 

\begin{lemma}\label{lem: pushout deformation volume}
If $\Phi'$ is a pushout deformation of $\Phi$, then 
\[\voltrans_d(\Phi')\le \voltrans_d(\Phi)\] and 
\[\vol_d\Bigl(\Phi'_x\vert_{{\Phi'_x}^{-1}(\star(F))}\Bigr)\le \vol_d\Bigl(\Phi_x\vert_{\Phi_x^{-1}(\star(F))}\Bigr) \]
for every thin face $F\in N(\calV)$ and every $x\in X$.  
\end{lemma}

\subsection{Pushing down the skeleta}

We consider the Cantor nerve map $\Phi$ of a Cantor cover $\calU$ with no self-intersections. 
Let $\calV$ be the locally finite cover of $\widetilde M$ obtained by the right factors of elements in $\calU$ (see Definition~\ref{def: Cantor nerve}). 
Let $N\in\bbN$ 
be such that the nerve map $\Phi\colon \xxm\to\cantornerve(\calU)$ lands in the $N$-skeleton. 
We define 
\[ \delta(k)\defq\Bigl( 3\cdot\epsilon(d)\cdot \sigma(d)^d\cdot e^{-\beta(d)\cdot k}\Bigr)^{1/d}.\]
We set $\Phi_N\defq \Phi$ and construct, by a finite downward induction, a sequence of Lipschitz Cantor bundle maps subordinate to $\calU$ 
\[\Phi_{i}\colon \xxm\to \cantornerve(\calU)^{(i)},~i=N, \dots, d,\] 
 such that for every $i\in \{N, N-1, \dots, d+1\}$ 
\begin{equation}\label{eq: inductive global volume}
\voltrans_d\bigl(\Phi_{i-1}\bigr)\le  \bigl(1-2\cdot\delta(i)\bigr)^{-d}\cdot G(V_1,d)\cdot \voltrans_d\bigl(\Phi_{i}\bigr),
\end{equation} 
and for every thin face $F\in N(\calV)$ and every $x\in X$ and $i\in \{N, N-1, \dots, d+1\}$
\begin{equation}\label{eq: inductive local volume}
\vol_d\Bigl(\Phi_{i-1}\vert_{\Phi_{i-1}^{-1}\bigl(\{x\}\times \star(F)\bigr)}\Bigr)\le  \bigl(1-2\cdot\delta(i)\bigr)^{-d}\cdot \vol_d\Bigl(\Phi_{i}\vert_{\Phi_{i}^{-1}\bigl(\{x\}\times \star(F)\bigr)}\Bigr)
\end{equation}
and for every thin face $F\in N(\calV)$ and every $x\in X$ and $i\in \{N, N-1, \dots, d\}$
\begin{equation}\label{eq: any local volumes}
\vol_d\Bigl(\Phi_{i}\vert_{\Phi_{i}^{-1}\bigl(\{x\}\times \star(F)\bigr)}\Bigr)
\le 2\cdot\epsilon(d)\cdot r_1(F)^d\cdot e^{-\beta(d)\cdot d(F)}.
\end{equation}

We combine the deformation steps in the previous subsections to inductively deform $\Phi_i$ 
with $i>d$ which lands in the $i$-skeleton to map $\Phi_{i-1}$ which lands in the $(i-1)$-skeleton. The application of Guth's pushout lemma requires that the map to the nerve is fiberwise piecewise smooth. This is true for the original Cantor nerve map. All deformation steps preserve that property. The first map $\Phi_N=\Phi$ satisfies~\eqref{eq: any local volumes} by Theorem~\ref{thm: volume restricted to star} 
and~\eqref{eq: thin face}. Next we construct $\Phi_{i-1}$ from $\Phi_i$. 

First, let $\Phi_i'$ be a Federer-Fleming deformation of $\Phi_i$. The image of $\Phi_i'$ does not meet any open thick $i$-faces; it lies in the Cantor subcomplex $Z_2^{(i)}$. By Lemma~\ref{lem: federer-fleming volume}, 
\begin{equation}\label{eq: G estimate of Federer Fleming}
\voltrans_d\bigl(\Phi_i'\bigr)\le G(V_1,d)\cdot \voltrans_d\bigl(\Phi_i\bigr).
\end{equation}
Let $F$ be an open thin face in $\cantornerve (\calU)_x\subset \{x\}\times N(\calV)$. Let $x\in X$. We have 
\[\Phi_{i}'\vert_{{\Phi_{i}'}^{-1}\bigl(\{x\}\times F\bigr)}=
\Phi_{i}\vert_{\Phi_{i}^{-1}\bigl(\{x\}\times F\bigr)}.\]
Since every face in the open star of a thin face is thin, we also obtain that 
\begin{equation}\label{eq: equality thin star}
\Phi_{i}'\vert_{{\Phi_{i}'}^{-1}\bigl(\{x\}\times \star(F)\bigr)}=
\Phi_{i}\vert_{\Phi_{i}^{-1}\bigl(\{x\}\times \star(F)\bigr)}.
\end{equation}
Next we assume that $F$ is $i$-dimensional and we consider a pushout deformation $\Phi_i''$ of $\Phi_i'$ which does not increase volumes according to Lemma~\ref{lem: pushout deformation volume}. Let 
\[w_{i,F}\defq \sigma(d)\cdot \vol_d\Bigl(\Phi_{i}'\vert_{\Phi_{i}'^{-1}\bigl(\{x\}\times F\bigr)}\Bigr)^{1/d}=\sigma(d)\cdot \vol_d\Bigl(\Phi_{i}\vert_{\Phi_{i}^{-1}\bigl(\{x\}\times F\bigr)}\Bigr)^{1/d}.
   \]
The image of $(\Phi_i'')_x$ within~$F$ lies in the $w_{i,F}$-neighborhood of the boundary of 
a convex subset of $F$ which we can choose arbitrarily large within $F$. 
By~\eqref{eq: any local volumes} we have 
\begin{equation*}
w_{i,F}\le \sigma(d)\cdot\vol_d\Bigl(\Phi_{i}\vert_{\Phi_{i}^{-1}\bigl(\{x\}\times \star(F)\bigr)}\Bigr)^{1/d}\le \bigl(2\cdot\sigma(d)^d\cdot \epsilon(d)\cdot e^{-\beta(d)\cdot i}\bigr)^{1/d}\cdot r_1(F).
\end{equation*}
Hence we choose the convex subset so that 
$\im (\Phi_i'')_x\cap F$ lies in the $\delta(i)\cdot r_1(F)$-neighborhood of $\partial F$. 
Hence the composition with a suitable compression map \[\Phi_{i-1}\defq \Rho_{\delta(i)}\circ \Phi_i''\] 
lands in the $(i-1)$-skeleton. Now~\eqref{eq: inductive global volume} follows 
from~\eqref{eq: G estimate of Federer Fleming}, the fact that the pushout deformation does not increase volume, and $\Rho_{\delta(i)}$ having Lipschitz constant at most $\bigl(1-2\cdot\delta(i)\bigr)^{-1}$. Similarly and 
because of~\eqref{eq: equality thin star} for every thin face $F$ and 
$\Rho_{\delta(i)}^{-1}(\star(F))\subset \star(F)$ for every face $F$ 
we obtain~\eqref{eq: inductive local volume}. Note that~\eqref{eq: epsilon constant} says that 
\[ \prod_{l=i}^N\bigl(1-2\cdot\delta(l)\bigr)^{-d}<2.\]
Using the induction hypothesis, we obtain~\eqref{eq: any local volumes} from  
\begin{align*} 
\vol_d\Bigl(\Phi_{i-1}\vert_{\Phi_{i-1}^{-1}\bigl(\{x\}\times \star(F)\bigr)}\Bigr)&\le \vol_d\Bigl(\Phi\vert_{\Phi^{-1}\bigl(\{x\}\times \star(F)\bigr)}\Bigr)\cdot\prod_{l=i}^N \bigl(1-2\cdot\delta(l)\bigr)^{-d}\\
&\le 2\cdot\vol_d\Bigl(\Phi\vert_{\Phi^{-1}\bigl(\{x\}\times \star(F)\bigr)}\Bigr)\\
&\le 2\cdot C(d)\cdot V_1\cdot r_1(F)^{d+1}\cdot e^{-\beta(d)\cdot d(F)}\quad\text{(Theorem~\ref{thm: volume restricted to star})}\\
&\le 2\cdot \epsilon(d)\cdot r_1(F)^d\cdot e^{-\beta(d)\cdot  d(F)}\quad\text{(see~\eqref{eq: thin face})}.
\end{align*}

\begin{theorem}\label{thm: homotoped down map}
	There is constant $C(V_1, d)>0$ only depending on~$V_1$ and~$d$ such that the map $\Phi_d\colon \xxm\to\cantornerve(\calU)^{(d)}$ satisfies 
	\[\voltrans_d\bigl(\Phi_d\bigr)\le C(V_1,d)\cdot\vol(M).\]
Furthermore, for every thin $d$-face $F$ in $N(\calV)$ and every $x\in X$ we have 	
\[\vol_d\Bigl(\Phi_d\vert_{\Phi_{d}^{-1}\bigl(\{x\}\times \star(F)\bigr)}\Bigr)< r_1(F)^d.\] 
\end{theorem}

\begin{proof}
By the same argument as in~\cite{guthvolume}*{Proof of Lemma~9} based on~\cite{guthvolume}*{Lemma~3}, there is a constant $D(V_1, d)>0$ only depending on $V_1$ and the dimension $d$ such that every thick face in $\cantornerve(\calU)$ is at most $D(V_1,d)$-dimensional. Hence we have to apply the Federer-Fleming deformation step at most $D(V_1, d)$ times. The constant $G(V_1,d)$ in~\eqref{eq: inductive global volume} only appears if a Federer-Fleming deformation was used when deforming $\Phi_i$ to $\Phi_{i-1}$. Therefore we obtain that 
\begin{align*}
\voltrans_d\bigl(\Phi_d\bigr)&\le \prod_{l=d+1}^\infty \bigl( 1-2\delta(l)\bigr)^{-d}\cdot G(V_1,d)^{D(V_1,d)}\cdot \voltrans_d\bigl(\Phi\bigr)\\
&\le 2 \cdot G(V_1,d)^{D(V_1,d)}\cdot \voltrans_d\bigl(\Phi\bigr).
\end{align*}
The first assertion now follows from Theorem~\ref{thm: volume of nerve map}. The second assertion follows from~\eqref{eq: epsilon constant} and~\eqref{eq: any local volumes}. 
\end{proof}

\section{From volume to simplicial volume}\label{sec: from volume to simplicial volume}

In this section we complete the proofs of the main results in the introduction. 

Let us recall the common setup of the main theorems. Let $M$ be a closed connected $d$-dimensional Riemannian manifold with fundamental group~$\Gamma$. Let $X$ be a Cantor space endowed with a free continuous action of $\Gamma$ and a $\Gamma$-invariant Borel probability measure~$\mu$ (Theorem~\ref{thm: existence of action}). 
Let $V_1>0$ be an upper bound of the volumes of $1$-balls in the universal cover $\widetilde M$. 
We choose a good Cantor cover $\calU$ on~$\xxm$ with no self-intersections (Theorem~\ref{thm:EquivariantCover}). The transverse volume of the associated Cantor nerve map is universally bounded by the volume of~$M$ (Theorem~\ref{thm: volume of nerve map}). According to the previous section, in particular Theorem~\ref{thm: homotoped down map}, we can deform the Cantor nerve map to the $d$-skeleton without 
loosing the control on its transverse volume. More precisely, we obtain a Cantor bundle map $\Phi\colon\xxm\to \cantornerve(\calU)^{(d)}$ subordinate to~$\calU$ such that 
\begin{equation} \voltrans_d(\Phi)\le C(d,V_1)\cdot\vol(M)\label{eq: volume bound final map}
\end{equation}
for a constant $C(d,V_1)>0$ only depending on $V_1$ and the dimension $d$. 
The second assertion of Theorem~\ref{thm: homotoped down map} and the fact that $r_1(F)^d\le \vol_d(F)$ for an open $d$-face $F$ implies that for every $x\in X$ and every 
thin open $d$-face $F$ in $N(\calU_x)$ the image of $\Phi_x$ misses at least a point of $F$. As before, we will denote by $\calV$ the locally finite cover of $\widetilde M$ by all balls appearing as factors of elements of $\calU$. For the rest of this section, let $N$ denote the 
$d$-skeleton of the nerve of $\calV$. In particular, we have 
\[\cantornerve(\calU)\subset X\times N.\]
The remaining steps to complete the proofs of the main theorems are as follows. 

In~\ref{sub: modules over the group ring} we present an auxiliary  result on the freeness of certain $\bbZ[\Gamma]$-modules. This is, for example, needed in the proof of Lemma~\ref{lem: htp equivalence with caterina} 
where we invoke the fundamental theorem of homological algebra to show that a homology isomorphism between free $\bbZ[\Gamma]$-chain complexes is induced by a $\bbZ[\Gamma]$-chain homotopy equivalence. 
In~\ref{sub: classifying maps} we discuss classifying maps to classifying spaces. In~\ref{sub: cellular chains vs volume} we see how to read off geometric information from the coefficients of a suitable representative of the image of the fundamental class under $\Phi_\ast$. 
We finish the proofs of the main theorems in~\ref{sub: conclusion}.

\subsection{A result on modules over the group ring}\label{sub: modules over the group ring}

If $H<\Gamma$ is a finite subgroup, then $\bbQ[\Gamma]\otimes_{\bbQ[H]}\bbQ$ is a projective $\bbQ[\Gamma]$-module. However, if $H$ is a non-trivial finite subgroup, $\bbZ[\Gamma]\otimes_{\bbZ[H]}\bbZ$ is not a projective $\bbZ[\Gamma]$-module. The following lemma shows how to remedy the situation for integral coefficients using the $\bbZ[\Gamma]$-module $C(X;\bbZ)$. 

\begin{lemma}\label{lem: free module}
   The following statements hold true: 
   \begin{enumerate}
   \item Let $H<\Gamma$ be a finite subgroup and $\chi\colon H\to \{\pm 1\}$ a character. Let $\bbZ^\chi$ denote the $\bbZ[H]$-module $\bbZ$ endowed with the action $h\cdot x=\chi(h)x$ 
   for $h\in H$ and $x\in \bbZ$. Then the $\bbZ[\Gamma]$-module 
   \[ C(X;\bbZ)\otimes\Bigl( \bbZ[\Gamma]\otimes_{\bbZ[H]} \bbZ^\chi\Bigr)\]
   with the module structure 
   induced by the diagonal (left) $\Gamma$-action is free.
   \item Let $H<\Gamma$ be a finite subgroup. 
   The $\bbZ[\Gamma]$-module $C(X;\bbZ)\otimes\bbZ[\Gamma/H]$ with the module structure 
   induced by the diagonal (left) $\Gamma$-action is free.
   \end{enumerate}
   \end{lemma}
   
   \begin{proof}
Ad 1). Since $H$ is finite there is a clopen fundamental domain $A$ of the $H$-action on $X$. 
   Consider the $\bbZ$-homomorphism 
   \[ C(A;\bbZ)\to C(X;\bbZ)\otimes\Bigl( \bbZ[\Gamma]\otimes_{\bbZ[H]} \bbZ^\chi\Bigr)\]
   that maps $f\in C(A;\bbZ)$ to $f\otimes 1\otimes 1$. Here we regard $C(A;\bbZ)$ as a subgroup 
   of $C(X;\bbZ)$ by extending functions by zero. The above homomorphism extends to a $\bbZ[\Gamma]$-homomorphism 
   \[ g\colon\bbZ[\Gamma]\otimes C(A;\bbZ)\to C(X;\bbZ)\otimes\Bigl( \bbZ[\Gamma]\otimes_{\bbZ[H]} \bbZ^\chi\Bigr)\]
   from the induced module. Since $C(A;\bbZ)$ is a free $\bbZ$-module by Corollary~\ref{cor: continuous function free} it suffices to show that $g$ is bijective. 
   
   Let $S\subset \Gamma$ be a set of representatives for the right $H$-cosets. We obtain a natural isomorphism as $\bbZ$-modules 
   \[ \bbZ[\Gamma]\otimes_{\bbZ[H]} \bbZ^\chi\cong \bigoplus_{\gamma\in S} \bbZ\]
   and thus 
   \begin{equation}\label{eq: coset representation of target}
   C(X;\bbZ)\otimes \Bigl( \bbZ[\Gamma]\otimes_{\bbZ[H]} \bbZ^\chi\Bigr)\cong 
   \bigoplus_{\gamma\in S} C(X;\bbZ).
   \end{equation}

   The domain of~$g$ is in an obvious way isomorphic to  
   \begin{equation}\label{eq: coset representation of domain}
   \bbZ[\Gamma]\otimes C(A;\bbZ)\cong \bbZ[\Gamma]\otimes_{\bbZ[H]} C(X;\bbZ)\cong \bbZ[\Gamma/H]\otimes C(X;\bbZ)\cong \bigoplus_{\gamma\in S}C(X;\bbZ).
   \end{equation}
   Both isomorphisms~\eqref{eq: coset representation of target} and~\eqref{eq: coset representation of domain} are compatible with $g$. Thus $g$ is an isomorphism. 
%

   Ad 2). This is a special case of~(1) for the trivial character. 
   \end{proof}

\subsection{Classifying maps and the simplicial norm}\label{sub: classifying maps}

Upon passing to the barycentric subdivision $N$ becomes a proper $\Gamma$-CW complex (Lemma~\ref{lem: proper Gamma CW}). 
We choose a map $\Psi\colon X\times N\to X\times E\Gamma$ as 
in Lemma~\ref{lem: map to classifying space}. 
We consider the following composition of maps of $\bbZ[\Gamma]$-chain complexes. 
\begin{equation}\label{eq: composition up to classifying space}
   C_\ast(\widetilde M)\to C(X;\bbZ)\otimes C_\ast(\widetilde M)\xrightarrow{\Phi_\ast} C(X;\bbZ)\otimes C_\ast(N)\xrightarrow{\Psi_\ast} C(X;\bbZ)\otimes C_\ast(E\Gamma)
\end{equation}
The first map is induced by the inclusion of constant function $\bbZ\hookrightarrow C(X;\bbZ)$. The next two maps are induced by  $\Phi$ and $\Psi$ according to Lemma~\ref{lem: induced map in smaller subcomplex}. Recall that the group $\Gamma$ acts diagonally on the tensor products. The abelian group 
$C(X;\bbZ)$ is free, in particular flat, due to Corollary~\ref{cor: continuous function free}. Hence we have a resolution of the $\bbZ[\Gamma]$-module $C(X;\bbZ)$ on the right. 
Again by Lemma~\ref{lem: free module} it is a free $\bbZ[\Gamma]$-resolution. The singular chain complex $C_\ast(\widetilde M)$ is a free $\bbZ[\Gamma]$-chain complex, and on $0$-th homology the above composition is the inclusion of constant functions. Let $c\colon \widetilde M\to E\Gamma$ be the classifying map. Then \[C_\ast(\widetilde M)\xrightarrow{c_\ast} C_\ast(E\Gamma)\hookrightarrow C(X;\bbZ)\otimes C_\ast(E\Gamma)\] is a $\bbZ[\Gamma]$-chain map with the same behaviour on $0$-th homology. By the fundamental theorem of homological algebra the two chain maps are chain homotopic which we record for later use. 

\begin{remark}\label{rem: homotopic chain maps to classifying space}
Let $c\colon \widetilde M\to E\Gamma$ be the classifying map. The map~\eqref{eq: composition up to classifying space} and the chain map 
$C_\ast(\widetilde M)\xrightarrow{c_\ast} C_\ast(E\Gamma)\hookrightarrow C(X;\bbZ)\otimes C_\ast(E\Gamma)$ are chain homotopic as $\bbZ[\Gamma]$-chain maps. Further, the latter map is equal to the composition 
$C_\ast(\widetilde M)\hookrightarrow C(X;\bbZ)\otimes C_\ast(\widetilde M)\xrightarrow{\id\otimes c_\ast} C(X;\bbZ)\otimes C_\ast(E\Gamma)$. 
\end{remark}

\

\subsection{Cellular chains and volume in the Cantor nerve}\label{sub: cellular chains vs volume}

Let $S_n$ be a complete set of representatives of the $\Gamma$-orbits of $n$-faces of~$N$. 
For each $n$-face $F$ let $\Gamma_F<\Gamma$ be the finite subgroup of elements $\gamma\in \Gamma$ with $\gamma F=F$ as subsets of $N$. After choosing an orientation for an $n$-face $F$ we obtain a character $\eta_F\colon\Gamma_F\to \{\pm 1\}$ which indicates whether $\gamma\in\Gamma_F$ preserves or reverses the orientation; the character is independent of the choice of orientation. Note that the character would be trivial if the CW-structure of $N$ would be a $\Gamma$-CW structure. 
There is an isomorphism of $\bbZ[\Gamma]$-modules 
\begin{equation}\label{eq: cellular chain complex as Gamma module}
C_n^\mathrm{cell}(N)\cong 
\bigoplus_{F\in S_n}\bbZ[\Gamma]\otimes_{\bbZ[\Gamma_F]}\bbZ^{\eta_F}.
\end{equation}
Lemma~\ref{lem: free module} now implies the first statement of the following lemma. The second statement follows similarly by noting that the singular chain groups 
$C_n(N;\bbZ)$ 
is a direct sum of $\bbZ[\Gamma]$-modules of the type 
$\bbZ[\Gamma/H]\cong \bbZ[\Gamma]\otimes_{\bbZ[H]}\bbZ$ with $H<\Gamma$ being finite.  

\begin{lemma}\label{lem: singular chains free}
   For every $n\in\bbN$ the $\bbZ[\Gamma]$-module $C(X;\bbZ)\otimes_\bbZ C_n^\mathrm{cell}(N)$ endowed with the diagonal $\Gamma$-action is free.
   The same is true when $C_n^\mathrm{cell}(N)$ is replaced by $C(X)$. 
   \end{lemma}
   
There is a natural chain map from the cellular chain complex of $N$ to the oriented singular chain complex of $N$ -- but not to the singular chain complex of $N$. Recall that 
 the \emph{oriented singular chain complex} $C_\ast^o(Y)$ of a space $Y$ is the  quotient complex of $C_\ast(Y)$ obtained by introducing the relation  
$g\sigma-\operatorname{sign}(g)\sigma$ for $\sigma\in C_p(X)$, $g\in S(p+1)$ and the natural action of the symmetric group $S(p+1)$ on singular $p$-simplices, and the relation $\sigma=0$ if there is a transposition $t$ with $t\sigma=\sigma$. 

The barycentric subdivision of a (closed) $n$-face in $N$ consists of $2^n\cdot n!$ many $n$-simplices. For each oriented $n$-face $F$ of $N$ 
the sum of the affine singular $n$-simplices matching the simplices of the barycentric subdivision with orientation is a chain $c_F\in C^o_n(N)$. We obtain an equivariant chain map 
\[ s_\ast\colon C^\mathrm{cell}_\ast(N)\to C^o_\ast(N)\]
that maps an oriented $n$-face $F$ to $c_F$. 
We endow the oriented singular chain complex with the quotient norm. 
Since the integral foliated simplicial volume is defined in terms of the norm on singular chains it is important to know that we do not lose too much by passing to oriented singular chains. The following result can be found in~\cite{campagnolo+sauer}*{Theorem~3.3 and Remark~3.4}.

\begin{theorem}\label{thm: comparision non-oriented and oriented homology}
Let $Y$ be a topological space. The projection $\pr_\ast\colon C_\ast(Y)\to C_\ast^o(Y)$ 
is a natural chain homotopy equivalence. 
The norm of the map $\pr_\ast$ is at most~$1$, and the norm 
of a suitable chain homotopy inverse is at most~$(p+1)!$ in degree~$p$. 
\end{theorem}

\begin{remark}\label{rem: geometric form in paper with caterina}
 The natural chain homotopy inverse constructed in~\cite{campagnolo+sauer}*{Theorem~3.3 and Remark~3.4} takes the equivalence class of a  singular simplex $\sigma$ and maps it to a linear combination of singular simplices with coefficients in $\{1,-1\}$ which corresponds to a barycentric subdivision of $\sigma$. 
\end{remark}

\begin{lemma}\label{lem: htp equivalence with caterina}
The composition 
\[ C(X;\bbZ)\otimes C^{\mathrm{cell}}_\ast(N)\xrightarrow{\id\otimes  s_\ast} C(X;\bbZ)\otimes C^o_\ast(N)\xrightarrow{\id\otimes q_\ast^N} C(X;\bbZ)\otimes C_\ast(N),\]
where $q_\ast^N$ is any natural chain homotopy inverse as in Theorem~\ref{thm: comparision non-oriented and oriented homology}, is a $\bbZ[\Gamma]$-chain homotopy equivalence (with regard to the diagonal $\Gamma$-actions). 
\end{lemma}

\begin{proof}
Both $s_\ast$ and $q_\ast^N$ are homology isomorphisms. Since $C(X;\bbZ)$ is a free abelian group (Corollary~\ref{cor: continuous function free}), also $\id\otimes s_\ast$ and $\id\otimes q_\ast^N$ are homology isomorphisms. Both the domain and the codomain of 
the composition are free $\bbZ[\Gamma]$-modules by Lemma~\ref{lem: singular chains free}. By the fundamental 
theorem of homological algebra the composition is a $\bbZ[\Gamma]$-homotopy equivalence. 
\end{proof}

In the following we consider the local degree of a map $f\colon \widetilde M\to S^d$ which is proper outside a fixed basepoint of $S^d$~\cite{dold}*{VIII.4}. 
Recall that the \emph{local degree} of $f$ at point $z\in S^d$ different from the basepoint is the integer $\deg_z(f)$ such that the locally finite fundamental class is sent to $\deg_z(f)\cdot [S^d]$ under  
\[H_d^\mathrm{lf}\bigl(\widetilde M\bigr)\rightarrow H_d\bigl(\widetilde M, \widetilde M\bs f^{-1}(\{z\})\bigr)\xrightarrow{f_\ast} H_d\bigl(S^d, S^d\bs\{z\}\bigr)\xleftarrow{\cong}H_d(S^d). \]
The local degree does not depend on the choice of $z$~\cite{dold}*{Proposition~4.4 on p.~267}; thus we denote it by $\deg(f)$. 
If $f$ is Lipschitz then 
\begin{equation}\label{eq: local degree as differentials} \deg(f)=\sum_{y\in f^{-1}(\{z\})} \sign\det Df_y
\end{equation}
for almost every $z\in S^d$ by~\cite{federer}*{Corollary~4.1.26 on p.~383}. 

For a $d$-face $F$ in $N$ we write $S(F)$ 
for the quotient of the closure of $F$ by its boundary which is homeomorphic to $S^d$. We take 
the collapsed boundary as the basepoint of $S(F)$. 

Below we refer to the map $j_\ast^M$ 
from Definition~\ref{def: inclusion into parametrised chains}.  
\begin{lemma}\label{lem: jacobian formula in homology}
Let $S_d$ and $\Gamma_F$ for a $d$-face be as in~\eqref{eq: cellular chain complex as Gamma module}.
Let $A_F\subset X$ be a 
fundamental domain for the $\Gamma_F$-action on $X$. 
Then there are $a_F\in C(X;\bbZ)$ supported on $A_F$ and integral chains 
$e_F\in C_d(N)$ of $\ell^1$-norm at most $2^d\cdot d!\cdot (d+1)!$ such that the image of 
the fundamental class under 
\[ H_d(M)\xrightarrow{j_\ast^M} H_d^\Gamma\bigl(\widetilde M; C(X;\bbZ)\bigr)\xrightarrow{\Phi_\ast} H_d^\Gamma\bigl(N; C(X;\bbZ)\bigr)
\]
is represented 
by the cycle  
$\sum_{F\in S_d} a_F\otimes  e_F$. 
For $x\in A_F$, we have  
\[\deg\bigl(\widetilde M\xrightarrow{\Phi_x} N\xrightarrow{\pr_F} S(F)\bigr)=\pm a_F(x)\]
and  
\[ \abs{a_F(x)}\le\frac{1}{\vol_d(F)}\int_{\Phi_x^{-1}(F)}J_d\Phi(y)d\vol_d^{\widetilde M}(y).
\]
\end{lemma}

\begin{proof}
Every $d$-cycle in $C(X;\bbZ)\otimes_{\bbZ[\Gamma]} C_\ast(N)$  
is homologous to a cycle coming from 
$C(X;\bbZ)\otimes_{\bbZ[\Gamma]} C_d^{\mathrm{cell}}(N)$
via the map in Lemma~\ref{lem: htp equivalence with caterina} (after passing to $\Gamma$-coinvariants). 
Since $C_d^{\mathrm{cell}}(N)$ is an 
abelian group generated by $\Gamma$-translates of $F\in S_d$, it follows that every $d$-cycle is homologous to a $d$-cycle of the form 
\[ \sum_{F\in S_d} b_F\otimes  q_\ast^N\bigl([c_F]\bigr)\in C(X;\bbZ)\otimes_{\bbZ[\Gamma]}C_d(N) \]
for some $b_F\in C(X;\bbZ)$. Set $e_F\defq q_\ast^N([c_F])$. The statement about the norm $e_F$ follows from the fact that $c_F$ consists of $2^d \cdot d!$ singular simplices and Theorem~\ref{thm: comparision non-oriented and oriented homology}. 
Next we rewrite the tensor products to obtain functions supported on $A_F$: 
Let $b_F'=\chi_{A_F}\cdot b_F$ where $\chi_{A_F}$ is the characteristic function of $A_F$. We have 
\[ b_F\otimes e_F=\sum_{h\in \Gamma_F}b_F'(h^{-1}\_)\otimes e_F=\sum_{h\in \Gamma_F}b_F'\otimes \eta_F(h)\cdot e_F=\Bigl(\sum_{h\in\Gamma_F} \eta_F(h)b_F'\Bigr)\otimes e_F\]
where $\eta_F\colon \Gamma_F\to \{\pm 1\}$ is the character as in~\eqref{eq: cellular chain complex as Gamma module}. Therefore every homology class is of the form $\bigl[\sum_{F\in S_d} a_F\otimes e_F\bigr]$ where $a_F$ are functions supported on~$A_F$ and $e_F$ are integral chains of $\ell^1$-norm at most $2^d\cdot d!\cdot (d+1)!$. We represent $z_M\defq \Phi_\ast\circ j_\ast^M([M])$ like that with suitable $a_F$ and $e_F$. The image of $z_M$ under the map $\ev_x$, $x\in X$, 
from Section~\ref{sub: from chains to measure chains}, which is a 
locally finite homology class of $N$, is denoted by $(z_M)_x^\mathrm{lf}$. We obtain that 
\[ (\Phi_x)_\ast\bigl([\widetilde M]^\mathrm{lf}\bigr)=(z_M)_x^\mathrm{lf}=\bigl[\sum_{\gamma\in\Gamma}\sum_{F\in S_d} a_F(\gamma^{-1}x)\cdot \gamma\cdot e_F  \bigr ]
\]
as elements in the locally finite 
homology of $N$. 
Let $F_0$ be an open $d$-face in $N$ and $z_0$ a point in $F_0$. We consider the image of $(z_M)_x^\mathrm{lf}$
under the homomorphism 
\[
   H_d^\mathrm{lf}\bigl(N\bigr)\to H_d\bigl(N, N\bs \pr_{F_0}^{-1}(\{z_0\})\bigr)\xrightarrow{H_d(\pr_{F_0})} H_d\bigl(S(F_0), S(F_0)\bs\{z_0\}\bigr)\xleftarrow{\cong} H_d\bigl(S(F_0)\bigr).
\]
For $\gamma\not\in \Gamma_{F_0}$ or $F\ne F_0$ the chain $\gamma\cdot e_F$ is mapped to zero under 
the chain map $C_d^\mathrm{lf}(N)\to C_d(N, N\bs\pr_{F_0}^{-1}(\{z_0\}))\to C_d(S(F_0), S(F_0)\bs\{z_0\})$. 
Therefore only the terms $a_F(\gamma^{-1}x)\gamma\cdot e_F$ with $\gamma\in \Gamma_{F_0}$ and $F=F_0$ contribute something potentially non-zero to the image of the homology class.  
But for $x\in A_{F_0}$ and $\gamma\in \Gamma_{F_0}\bs\{1\}$ we have $a_F(\gamma^{-1}x)=0$. Hence $(z_M)_x^\mathrm{lf}$ is mapped to $a_{F_0}(x)$ times the generator, which implies the statement about the local degree. The bound for $\abs{a_F(x)}$ is now a direct consequence of the area formula~\cite{federer}*{Theorem~3.2.5 on p.~244 and the remark before 3.2.47 on p.~282} and the 
characterization~\eqref{eq: local degree as differentials} of the local degree. 
\end{proof}

\subsection{Conclusion of proofs of main results}\label{sub: conclusion}

For the next result we refer the reader to the overview of 
dimensional constants after Theorem~\ref{thm: pushout lemma}. 

\begin{theorem}\label{thm: main auxiliary statement}
For every $V_1>0$ and $d\in\bbN$ there are constants $\const(d, V_1)>0$ and $\epsilon(d)>0$ 
with the following properties. 

Let $(M,g)$ be a $d$-dimensional closed Riemannian manifold with $V_{(\widetilde M,\tilde g)}(1)<V_1$. 
Let $\Gamma=\pi_1(M)$, and let $c\colon M\to B\Gamma$ 
be the classifying map. Then 
\[
\norm{i_\ast^\bbR\circ c_\ast([M])}\le\norm{j_\ast^{B\Gamma}\circ c_\ast([M])}_\bbZ^X\le \const(d, V_1)\cdot \vol(M).   
\]
Furthermore, if $V_{(\widetilde M,\tilde g)}(1)<C(d)^{-1}\cdot \epsilon(d)$, then 
\[i_\ast^\bbR\circ c_\ast([M])=0\in H_d(B\Gamma;\bbR).\] 
\end{theorem}

\begin{proof} 
The following diagram contains all the maps we have to consider.

\[
\begin{tikzcd}
   & & H_d^\Gamma(N;C(X;\bbZ))\ar[d, "\Psi_\ast"] & &\\
   H_d(M)\ar[rr, bend right=20, "j_\ast^{B\Gamma}\circ c_\ast"]\ar[rrr, bend right=25,"i_\bbR\circ c_\ast"]\ar[r, "j_\ast^M"] &H_d^\Gamma(\widetilde M; C(X;\bbZ))\ar[ru, "\Phi_\ast"]\ar[r]& H_d^\Gamma(E\Gamma; C(X;\bbZ))\ar[r] & H_d(B\Gamma;\bbR)
\end{tikzcd}
\]
The middle horizontal map is induced by the (equivariant) classifying map $\widetilde M\to E\Gamma$ and the identity on $C(X;\bbZ)$. The right-hand horizontal map is induced by integration $C(X;\bbZ)\to \bbR$ (see Remark~\ref{rem: comparision parametrised and real norm}). 
The upper triangle is commutative 
by Remark~\ref{rem: homotopic chain maps to classifying space}.  
That the lower part commutes is straightforward. 

Let $z_M$ be the image of $[M]\in H_d(M)$ in $H^\Gamma_d(N; C(X;\bbZ))$. 
According to Lemma~\ref{lem: jacobian formula in homology} the homology class $z_M$ 
is represented by a cycle of the form 
\[\sum_{F\in S_d} a_F\otimes e_F\] 
where the function $a_F$ is supported on~$A_F$ and 
$a_F(x)$ is the local degree of $\Phi_x$ followed by the projection to $S(F)$ for $x\in A_F$. 
If $F\in S_d$ is thin and $x\in X$, then there is at least one point in the interior of $F$ that is not in the image of $\Phi_x$ according to the volume estimate of Theorem~\ref{thm: homotoped down map}.
In this case $a_F(x)=0$. Under the assumption $V_{(\widetilde M,\tilde g)}(1)<C(d)^{-1}\cdot \epsilon(d)$ every $d$-face is thin. Hence 
$z_M=0$. The commutativity of the diagram implies the second statement. 

The smallest side length of a thick $d$-face in $N$ and hence its volume are bounded from below by a constant that only depends on the dimension $d$ and $V_1$. Let $\const'(d,V_1)>0$ be such that $1/\const'(d, V_1)$ is a lower volume bound of thick $d$-faces. We now set 
$\const(d,V_1)\defq 2^d\cdot d!\cdot (d+1)!\cdot\const'(d, V_1)\cdot C(d,V_1)$, where 
$C(d, V_1)$ is the constant in~\eqref{eq: volume bound final map}.  
Since $\Psi_\ast$ does not increase the integral foliated norm, the above diagram commutes and because of Remark~\ref{rem: comparision parametrised and real norm}, it suffices to show that $\norm{z_M}_\bbZ^X\le\const(d,V_1)\cdot\vol(M)$ to obtain the first statement of the theorem. 
Let $T_d\subset S_d$ be the subset of thick $d$-faces. With the norm bound on $e_F$ from Lemma~\ref{lem: jacobian formula in homology} and the above argument for thin $d$-faces 
we obtain that 
\begin{equation*} \norm{z_M}_\bbZ^X \le 
   2^d\cdot d!\cdot (d+1)!\cdot\sum_{F\in T_d}\int_{A_F} \abs{a_F(x)}d\mu(x)
\end{equation*}
Again with Lemma~\ref{lem: jacobian formula in homology} we conclude that 
\begin{equation*}
   \norm{z_M}_\bbZ^X \le 2^d\cdot d!\cdot (d+1)!\cdot\const'(d,V_1)\sum_{F\in T_d}\int_{A_F}\int_{\Phi_x^{-1}(F)}\bigl\vert J_d\Phi_x(y)\bigr\vert d\vol_d^{\widetilde M}(y)
\end{equation*}
The subset  
\[ \Bigl\{ (x,y)\mid F\in T_d, x\in A_F, y\in \Phi_x^{-1}(F)\Bigr\}
\]
is contained 
in a $\Gamma$-fundamental domain of~$\xxm$. Hence the  
Area Formula (Theorem~\ref{thm: Area formula}) and the definition of $\const(d,V_1)$ imply that  
\begin{equation*}
   \norm{z_M}_\bbZ^X \le 2^d\cdot d!\cdot (d+1)!\cdot\const'(d,V_1)\cdot\vol_d(\Phi)\le \const(d,V_1)\cdot\vol(M).\qedhere
\end{equation*}
\end{proof}

\begin{proof}[Proof of Theorem~\ref{thm: main result}]
   By Gromov's mapping theorem~\cite{gromovvbc}*{Section~3.1. on p.~248} 
   we have $\norm{i_{\bbR,\ast}\circ c_\ast([M])}=\norm{M}$. 
Therefore Theorem~\ref{thm: main result} is implied by Theorem~\ref{thm: main auxiliary statement}. 
\end{proof}

\begin{proof}[Proof of Theorem~\ref{thm: main result about essential manifolds}]
By scaling the metric, it is enough to prove the case $R=1$. 
The case $R=1$ is the second statement of 
Theorem~\ref{thm: main auxiliary statement}. 
\end{proof}

\begin{proof}[Proof of Theorem~\ref{thm: main l2 result}]
   According to Theorem~\ref{thm: main auxiliary statement} the $X$-parametrised integral simplicial norm of $c_\ast([M])\in H_d(B\Gamma)$ is bounded from above by $\const(d, V_1)\cdot\vol(M)$. By Theorem~\ref{thm: bound von Neumann rank} 
   the von Neumann rank of $c_\ast([M])$ is bounded by $d\cdot C(d,V_1)\cdot \vol(M)$.  
If $M$ is, in addition, aspherical, then $c$ is a homotopy equivalence and the von Neumann rank of $c_\ast([M])$ is the von Neumann rank of $[M]$ which is the sum of the $\ell^2$-Betti numbers of $M$ according to Remark~\ref{rem: equivariant duality}. This implies the second statement of Theorem~\ref{thm: main l2 result}. The statement about the Euler characteristic is an immediate consequence since the alternating sum of $\ell^2$-Betti numbers equals the Euler characteristic. 
\end{proof}

\begin{proof}[Proof of Theorem~\ref{thm: main foliated simplicial volume}]
   Let $\alpha$ be a free measurable pmp action of $\Gamma$ on a standard probability space $(Y,\mu)$. 
   By~\cite{elek}*{Theorem~2} there is a free continuous action of $\Gamma$ on the Cantor set~$X$ and an equivariant Borel embedding $X\hookrightarrow Y$ such that $\mu(X)=1$. This means that we can realize every free measurable pmp action by a free continuous action on the Cantor set. 
   Therefore the $\alpha$-parametrised simplicial volume $|M|^\alpha$ coincides with the $X$-parametrised simplicial volume (with regard to $\mu$). 
   According to Theorem~\ref{thm: main auxiliary statement} the $X$-parametrised integral simplicial norm of $c_\ast([M])\in H_d(B\Gamma)$ is bounded from above by $\const(d, V_1)\cdot\vol(M)$.
   Since $c$ is a homotopy equivalence the $X$-parametrised integral simplicial norm of $c_\ast([M])$ is the $X$-parametrised integral simplicial volume of $M$. 
\end{proof}

\end{document}